\documentclass[11pt]{amsart}

\usepackage{amssymb}
\usepackage{amsthm}
\usepackage{amsmath}
\usepackage{graphicx}
\usepackage{amsaddr}
\usepackage{comment}
\usepackage[usenames,dvipsnames]{xcolor}
\usepackage[margin = 1in]{geometry}
\usepackage{enumerate}
\usepackage[toc,page]{appendix}
\usepackage{lineno}
\usepackage{color}
\usepackage{bbm}
\usepackage{hyperref}
\usepackage{mhchem}
\usepackage{siunitx}
\usepackage{booktabs}
\usepackage{float}

\definecolor{mygreen}{rgb}{0.1,0.75,0.2}

\providecommand{\bbs}[1]{(#1)}

 \newtheorem{thm}{Theorem}[section]
 \newtheorem{cor}[thm]{Corollary}
 \newtheorem{lem}[thm]{Lemma}
 \newtheorem{prop}[thm]{Proposition}

 \theoremstyle{definition}

 \theoremstyle{remark}
 \newtheorem{rem}[thm]{Remark}

 \numberwithin{equation}{section}

\newcommand{\pt}{\partial}

\newcommand{\eps}{\varepsilon}

\newcommand{\ud}{\,\mathrm{d}}
\newcommand{\8}{\infty}

\newcommand{\bR}{\mathbb{R}}
\newcommand{\bZ}{\mathbb{Z}}
\newcommand{\bE}{\mathbb{E}}

\newcommand{\vx}{\vec{x}}

\newcommand{\V}{\scriptscriptstyle{\mathrm{h}}}
\newcommand{\sr}{\scriptscriptstyle{\mathrm{r}}}
\newcommand{\sk}{\scriptscriptstyle{\mathrm{k}}}
\newcommand{\sm}{\scriptscriptstyle{\mathrm{m}}}
\newcommand{\sbe}{\scriptscriptstyle{\beta}}
\newcommand{\sn}{\scriptscriptstyle{\mathrm{n}}}
\newcommand{\zero}{\scriptscriptstyle{\mathrm{0}}}
\newcommand{\sa}{\scriptscriptstyle{\mathrm{a}}}
\newcommand{\scb}{\scriptscriptstyle{\mathrm{b}}}
\newcommand{\sci}{\scriptscriptstyle{\mathrm{i}}}

\newcommand{\rh}{r_{\V}}

\newcommand{\vxzh}{\vx_{\zero}^{\V}}

\newcommand{\beh}{\beta_{\V}}
\newcommand{\kzh}{k_{\zero}^{\V}}
\newcommand{\alzh}{\alpha_{\zero}^{\V}}
\newcommand{\aldh}{\alpha_{\V}}
\newcommand{\aluh}{\alpha^{\V}}
\newcommand{\alhi}{\alpha_{\V, \sci}}
\newcommand{\ti}{t_{\sci}}

\newcommand{\uh}{u_{\scriptscriptstyle{\mathrm{h}}}}
\newcommand{\tuh}{\tilde{u}_{\scriptscriptstyle{\mathrm{h}}}}
\newcommand{\Hh}{H_{\scriptscriptstyle{\mathrm{h}}}}

\newcommand{\tHrh}{\tilde{H}_{\sr,\V}}
\newcommand{\tHrhb}{\tilde{H}_{\sr, \V}^{\sbe}}
\newcommand{\tHrhpb}{\tilde{H}_{\sr, \V'}^{\sbe}}

\newcommand{\Lh}{L_{\scriptscriptstyle{\mathrm{h}}}}

\newcommand{\vh}{v_{\scriptscriptstyle{\mathrm{h}}}}

\newcommand{\wrhb}{w_{\sr,\V}^{\sbe}}
\newcommand{\wrhpb}{w_{\sr,\V'}^{\sbe}}

\newcommand{\zh}{\zeta_{\scriptscriptstyle{\mathrm{h}}}}

\newcommand{\Omegah}{\Omega_{\V}}
\newcommand{\horh}{\hat{\Omega}_{\sr, \V}}
\newcommand{\torh}{\tilde{\Omega}_{\sr, \V}}
\newcommand{\hobrh}{\hat{\Omega}_{\sbe, \sr, \V}}
\newcommand{\tobrh}{\tilde{\Omega}_{\sbe, \sr, \V}}
\newcommand{\tobrhp}{\tilde{\Omega}_{\sbe, \sr, \V'}}
\newcommand{\torhh}{\tilde{\Omega}_{\sr_{\V}, \V}}
\newcommand{\tobhrhh}{\tilde{\Omega}_{\sbe_{\V}, \sr_{\V}, \V}}
\newcommand{\tornhn}{\tilde{\Omega}_{\sr_{\sn}, \V_{\sn}}}

\newcommand{\wrh}{w_{\sr, \V}}
\newcommand{\wrhh}{w_{{\sr}_{\V}, \V}}
\newcommand{\wrhhath}{w_{{\sr}_{\V}, \hat{\V}}}

\newcommand{\wrnhn}{w_{{\sr_{\sn}}, {\V_{\sn}}}}
\newcommand{\wrhhbh}{w_{\sr_{\V},\V}^{\sbe_{\V}}}

\newcommand{\wzk}{w_{\scriptscriptstyle{\mathrm{0}},\sk}}
\newcommand{\uzk}{u_{\scriptscriptstyle{\mathrm{0}},\sk}}

\newcommand{\wrhk}{w_{\sr, \V, \sk}}
\newcommand{\twrh}{\tilde{w}_{\sr, \V}}
\newcommand{\twrhpb}{\tilde{w}_{\sr,\V'}^{\sbe}}

\newcommand{\cv}{X^{\scriptscriptstyle{\mathrm{h}}}}

\newcommand{\vxv}{\vec{x}_{i}}
\newcommand{\vxi}{\vec{x}_i}

\newcommand{\vy}{\vec{y}}

\begin{document}

\title[State-Constrained Chemical Reactions]{State-Constrained Chemical Reactions: Discrete-to-Continuous Hamilton--Jacobi Equations and Large Deviations}

\author[Y. Gao]{Yuan Gao}
\address{Department of Mathematics, Purdue University, West Lafayette, IN}
\email{gao662@purdue.edu}

\author[Y. Han]{Yuxi Han}
\address{Department of Mathematics, Purdue University, West Lafayette, IN}
\email{han891@purdue.edu}



\date{\today}

\begin{abstract}
We study the macroscopic behavior of chemical reactions modeled as random time-changed Poisson processes on   discrete state spaces. Using the WKB reformulation, the backward equation of the rescaled process leads to a discrete Hamilton–Jacobi equation with state constraints. As the grid size tends to zero, the limiting solution and its associated variational representation are closely connected to the good rate function of the large deviation principle for state-constrained chemical reactions in the thermodynamic limit. In this work, we focus on the limiting behavior of discrete Hamilton–Jacobi equations defined on bounded domains with state-constraint boundary conditions. For a single chemical reaction, we show that, under a suitable reparametrization, the solution of the discrete Hamilton–Jacobi equation converges to the solution of a continuous Hamilton–Jacobi equation with a Neumann boundary condition. Building on this convergence result and the associated variational representation, we establish the large deviation principle for the rescaled chemical reaction process in bounded domains.

\end{abstract}

\maketitle

\section{Introduction}







Chemical reactions can be modeled by a random time-changed Poisson process on a countable state space. Their macroscopic behavior, including phenomena such as large fluctuations, can be analyzed through the WKB reformulation. For the backward equation of the rescaled process, the WKB reformulation coincides with Varadhan's discrete nonlinear semigroup, which serves as a monotone scheme approximating the limiting first-order Hamilton–Jacobi equation (HJE).

In practical applications, physical considerations often motivate formulating the backward equations and the associated discrete HJE on a bounded domain with a state-constraint boundary condition. For instance, in a closed system, the conservation of mass naturally implies that the states remain bounded. In open systems, where materials can be exchanged with the environment, mass is not strictly conserved. Nevertheless, such systems can still limit the amount of a chemostat species that is  transferred into or out  of the reaction system.   Moreover, certain chemical processes require strict compositional constraints to ensure the stability and precision in manufacturing.

In this paper, we assume that
the discrete backward HJE is defined on the grid point
\[
\Omegah: =\left\{\vec{i}h : \vec{i} \in \mathbb{N}^N \right\} \cap\overline{\Omega},
\]
where $\Omega \subset \mathbb{R}^N_{+}:=   \{x\in \bR^N: \, x_i > 0 \text{ for all } i\}$ is an open, bounded domain with $\overline{\Omega} \subset \mathbb{R}^N_{+}$. Our goal is to establish the convergence of the monotone scheme given by the discrete backward HJE with a state constraint, as well as to derive the corresponding large deviation principle (LDP) on $\overline{\Omega}$.

We consider a system of $M$ chemical reactions involving $N$ molecular species, occurring in a container of size $\frac{1}{h}$, where $\frac{1}{h} \gg 1$. A stochastic description of these reactions is given by random time-changed Poisson processes $X(t)$ \cite{Kurtz15}. Here, each component of $X(t)$, denoted as
$
X_\ell(t), \, \ell=1, \cdots, N,
$
represents the number of molecules of the $\ell$-th species at time $t$. The process $X(t)$ is a continuous-time Markov chain defined on a discrete state space. We assume that reactions are independent of the spatial positions of the molecules and that the total number of molecules is proportional to the container size. The limit $h \to 0$ is known as the thermodynamic limit or macroscopic limit. For the $j$-th reaction, let $\vec{\nu}_j \in \mathbb{Z}^N$ denote the net change vector in the molecular counts. Consider an open, bounded domain $\Omega \subset \mathbb{R}^N_{+}$ with $\overline{\Omega} \subset \mathbb{R}^N_{+}$. Define the rescaled process $\cv(t)=hX(t)$ and assume that $\cv(t) \in \Omegah=\left\{\vec{i}h: \vec{i} \in \mathbb{N}^N \right\} \cap\overline{\Omega}$. Then, $\cv(t)$ evolves according to
\begin{equation}\label{CRCR}
\begin{aligned}
\cv(t) =& \cv(0) + \sum_{j=1}^M  \vec{\nu}_{j} h \Bigg(  \mathbbm{1}_{\left\{ \cv(t^-) + \vec{\nu}_jh \in \Omegah \right\}} Y^+_j  \left(\frac{1}{h}\int_0^t {\Phi}^+_j(\cv(s))\ud s\right)\\
&\qquad \qquad \qquad \qquad \qquad\qquad- \mathbbm{1}_{\left\{ \cv(t^-) - \vec{\nu}_jh \in \Omegah \right\}}Y^-_j  \left(\frac{1}{h}\int_0^t {\Phi}^-_j(\cv(s))\ud s \right)\Bigg),
\end{aligned}
\end{equation}
 where $Y_j^\pm(t)$ are independent and identically distributed (i.i.d.) unit-rate Poisson processes, and $\mathbbm{1}$ is the indicating function to enforce the state constraint, ensuring that the process remains in $\Omegah$. $\Phi^\pm_j ( \vx)$ is the intensity  of this Poisson process, indicating the encounter of species in one reaction. The WKB reformulation of the backward equation for $\cv$ is given by
\begin{equation}\label{upwind0}
\begin{aligned}
    \pt_t \uh(\vxv,t) =& \sum_{j=1, \vxv+ \vec{\nu}_{j} h\in \Omegah }^M     \Phi^+_j(\vxv)\left( e^{\frac{\uh\left(\vxv+ \vec{\nu}_{j} h,t\right)-\uh\left(\vxv, t\right)}{h}} - 1\right) \\
&+ \sum_{j=1,  \vxv- \vec{\nu}_{j}h\in\Omegah}^M \Phi_j^-(\vxv)\left( e^{  \frac{\uh\left(\vxv- \vec{\nu}_{j}h,t\right)-\uh\left(\vxv,t\right)}{h}  - 1}\right), \quad \left(\vxv, t\right) \in \Omegah \times (0, \infty),
\end{aligned}
\end{equation}
subject to the initial condition
\[
\uh(\vxv, 0) = u_0(\vxv), \quad \vxv \in \Omegah,
\]
where $u_0:\overline{\Omega} \to \mathbb{R}$ is a bounded and continuous function. In Proposition \ref{prop:variformuh}, we provide a variational formula for $\uh$ in terms of  an optimal control problem subject to a state constraint, which is new in the literature. As $h \to 0$, we expect that the solution of \eqref{upwind0} converges to a limit $u$, which solves the corresponding continuous HJE
\begin{equation}\label{eqn:cHJE}
\left\{\begin{aligned}
&\partial_t u(\vx,t) =  \sum_{j=1}^M \left[ \Phi^+_j(\vec{x})\left(e^{\vec{\nu}_j\cdot \nabla u(\vx, t)}-1\right)     +   \Phi_j^-(\vec{x})\left( e^{-\vec{\nu}_j\cdot \nabla u(\vx, t)}-1\right)  \right], \qquad (\vx, t)\in\Omega \times (0, \infty),\\
&u(\vx, 0) = u_0(\vx), \qquad \vx \in \overline{\Omega},
\end{aligned}
\right.
\end{equation}
with an appropriate boundary condition for $t>0$.

Surprisingly, even though the discrete HJE \eqref{upwind0} is subject to a state-constraint boundary condition, the limiting solution $u$ does not   satisfy the corresponding continuous HJE \eqref{eqn:cHJE} with a state constraint but rather with a Nuemann boundary condition. An explicit example demonstrating this phenomenon is provided in Section \ref{sec:example}. The main challenge lies in identifying the correct boundary condition for \eqref{eqn:cHJE} when $t>0$. We resolve this issue in the case of a single chemical reaction.

\subsection{Relevant Literature}
Our work is motivated by the recent study of Gao and Liu \cite{GaoLiu2023}, which established the convergence of the discrete HJE via the discrete nonlinear semigroup in the unbounded domain $\mathbb{R}^N_+$, together with the large deviation principle for chemical reaction processes at any fixed time under no-reaction boundary conditions. The stochastic optimal control problem of diffusion processes with state constraints was first studied in \cite{Lasry1989}, where the admissible controls are those that keep the diffusion process within a bounded domain with probability one at all times. In the zero-noise limit of diffusion processes, the rate of convergence of the corresponding second-order stationary HJE to the first-order state-constraint HJE was analyzed in \cite{HanTu2022_vanishVisc_stateConst, DuttaNguyenTu2025}. In contrast, we consider a stochastic control problem for a jump process, where the controls ensure that the process remains within a bounded domain with probability one at all times. In our setting, the zero-noise limit corresponds to the thermodynamic or large-volume limit, in which the grid size of the jump process tends to zero. To the best of our knowledge, the convergence of the discrete backward HJE in a bounded domain with state-constraint boundary conditions, as well as its limiting equation, has not been previously investigated.

The use of viscosity solutions to HJE as limits of nonlinear semigroups of Markov processes was developed by Feng and Kurtz \cite{feng2006large}, providing a general and systematic framework for studying LDP. See also the more recent developments in \cite{Kraaij_2016, Kraaij_2020}. The idea of employing nonlinear semigroups and variational principles of Markov processes to analyze LDP dates back to the seminal works of Varadhan and Fleming \cite{Varadhan_1966, fleming1983optimal}, while the idea of using viscosity solutions to HJE originates from \cite{Ishii_Evans_1985, fleming1986pde, Barles_Perthame_1987}. A comprehensive reference for HJE is \cite{tran_hj_2021}.

For chemical reaction in whole space, the sample path LDP was first established in \cite{AgazziDemboEckmann2017} using a time-linear interpolated process combined with the inverse contraction principle. More recently, the boundary issue has been addressed by introducing a uniform vanishing rate at the boundary \cite{Agazzi_Andreis_Patterson_Renger_2021}. Furthermore, a dynamic LDP for the joint evolution of concentrations and fluxes in chemical reaction jump processes was obtained in \cite{patterson2019large}.

For Hamilton–Jacobi equations with state constraints, we refer to \cite{Soner1986, Soner1986b} in the framework of optimal control theory, as well as subsequent works such as \cite{Capuzzo-Dolcetta1990, Ishii1996, ishii_class_2002, Mitake2008}. For Hamilton–Jacobi equations with Neumann boundary conditions, we refer to \cite{lions1985neumann, BARLES1991, DUPUIS19901123, Ishii2011}. For other related topics on state-constraint problems, large deviation principles, discrete schemes, and more, see also \cite{Cagnetti2013, Giga2019Discrete, Sandholm2022, JangTran2025}.


\subsection{Main result}

For the case of a single chemical reaction, i.e., $M=1$, with $N=2$, we establish that the discrete backward HJE with a state constraint converges to the limiting continuous HJE with a Neumann boundary condition. Furthermore, we prove the large deviation principle on the closure $\overline{\Omega}$. The proof extends to all $N\geq 2$ with $M=1$.

When $M=1$, let $\vec{\nu}$ be the reaction vector representing the net change in molecular counts. For $h>0$, let $\vxzh \in \Omegah$ be the initial position of the rescaled process $\cv(t)$. The state space of $\cv(t)$ is then restricted to the line passing through $\vxzh$ in the direction of $\vec{\nu}$. Under this setting, the discrete backward HJE \eqref{upwind0} simplifies to the following form:
\begin{equation}\label{eqn:1duh}
\begin{aligned}
\pt_t \uh(\vxv,t) = & \mathbbm{1}_{\left\{\vxv +h\vec{\nu} \,\in \,\Omegah\right\}}  \Phi^+(\vxv) \left( e^{\frac{\uh\left(\vxv+ \vec{\nu} h,t\right)-\uh\left(\vxv,t\right)}{h}} - 1\right)\\ &+ \mathbbm{1}_{\left\{\vxv -h\vec{\nu} \, \in \,\Omegah \right\}}\Phi^-(\vxv)\left( e^{  \frac{\uh\left(\vxv- \vec{\nu}h,t\right)-\uh\left(\vxv,t\right)}{h} } - 1\right), \quad \left(\vxv, t\right) \in \Omegah \times (0, \infty),\\
\uh(\vxv, 0)=&u_0(\vxv), \quad \vxv \in \Omegah,
\end{aligned}
\end{equation}
where $u_0$ is the prescribed initial condition.
Since the values of $\uh$ at grid points outside the line passing through $\vxv$ in the direction of $\vec{\nu}$ have no influence on $\uh(\vxv)$ for any $\vxv \in \Omegah$, it is natural to reparametrize the problem and reduce \eqref{upwind0} to an equation restricted to this line. This results in a one-dimensional discrete HJE, which significantly simplifies the subsequent analysis.

Let $\vx_0 \in \overline{\Omega}$ be the limiting starting position of the rescaled reaction process, that is, $\lim_{h \to 0+} \vxzh =\vx_0$. Define the maximal closed line segment in the direction $\vec{\nu}$ that contains $\vx_0$ by
\[
S_0=  \left\{\vec{x}_0 + \alpha \vec{\nu}: a\leq \alpha \leq b\right\} \subset \left\{\vec{x}_0 + \alpha \vec{\nu}: \alpha \in \mathbb{R} \right\}\cap \overline{\Omega},
\]
for some constants $a, b \in \mathbb{R}$.

We first consider the case $\vxzh \in S_0$ for sufficently small $h>0$. In this case, the state space of the rescaled chemical reaction process $\cv(t)$ is restricted to the line segment $S_0$. Suppose $\cv(0)=\vxzh=\vx_0+(r+kh)\vec{\nu} \in S_0$ for some $r \in \mathbb{R}$ small and $k \in \mathbb{N}$. Then the process evolves within the discrete set
\[
\horh:=\left\{\vec{x}_0 +(r+\alpha) \vec{\nu} \in S_0: \alpha =kh, k  \in \mathbb{Z}\right\}.
\]
Define the projected state space
\[
\torh :=\{r+\aldh:\aldh = kh, k \in \mathbb{Z}, \vx_0 +(r+\aldh) \vec{\nu} \in S_0\},
\]
and introduce the reduced function $\wrh:\torh \times [0,\infty) \to \mathbb{R}$ by
\[\wrh \left(r+\aldh, t\right): = \uh(\vec{x}_0+(r+\aldh) \vec{\nu} , t).\]
We show that $\wrh$ converges to a function $w$, which is the unique solution of the following continuous Hamilton--Jacobi equation with a Neumann boundary condition:

\begin{equation}\label{HJE_NM}
\left\{
    \begin{aligned}
    \partial_t w(\alpha, t) &= \tilde{H}\left(\alpha, \partial_\alpha w(\alpha, t)\right),
    && \text{in } (a, b) \times (0, \infty), \\[6pt]
    -\partial_\alpha w(a, t) &= \partial_\alpha w(b, t) = 0,
    && \text{for } t > 0, \\[6pt]
    w(\alpha, 0) &= w_0(\alpha),
    && \text{for } \alpha \in [a, b],
    \end{aligned}\tag{{\sf HJE}}
\right.
\end{equation}
where the Hamiltonian is given by
\[
\tilde{H}(\alpha, p)=\tilde{\Phi}^+(\alpha) \left( e^p - 1\right)+ \tilde{\Phi}^-(\alpha)\left( e^{-p} - 1\right),
\]
and
\[
\tilde{\Phi}^+(\alpha) = \Phi^+(\vec{x}_0 + \alpha \vec{\nu}),
\quad
\tilde{\Phi}^-(\alpha) = \Phi^-(\vec{x}_0 + \alpha \vec{\nu}),
\quad
w_0(\alpha) = u_0(\vec{x}_0 + \alpha \vec{\nu}),
\quad \text{for } \alpha \in [a, b].
\]
The proof proceeds by verifying the conditions for viscosity solutions of \eqref{HJE_NM}. Specifically, the upper semicontinuous (USC) envelope $\overline{w}$ of $\wrh$ is shown to be a subsolution to \eqref{HJE_NM}, while the lower semicontinuous (LSC) envelope $\underline{w}$ of $\wrh$ is a supersolution to \eqref{HJE_NM}. The comparison principle of \eqref{HJE_NM} then yields $\overline{w}=\underline{w}=w$, which gives the desired convergence.

We then address the more general case where the starting point $\vxzh$ of $\cv(t)$ lies in $\overline{\Omega}$, under the additional assumption that $\Omega$ is convex. In this case, $\vxzh$ is not necessarily on $S_0$, but it remains close to $\vx_0$ since $\lim_{h \to 0+} \vxzh =\vx_0$. In particular, we assume $\vxzh = \vec{x}_0 + \beta\vec{\nu}_\perp + (r+kh)\vec{\nu}$ for some $r, \beta$ small and $k \in \mathbb{N}$. We consider the maximal closed line segment passing through $\vx_0 + \beta \vec{\nu}_{\perp}$ in the direction of $\vec{\nu}$ and denote by
\[
S_{\sbe}=\left\{\vec{x}_0 +\beta\vec{\nu}_{\perp}+ \alpha \vec{\nu}: a_{\sbe}\leq \alpha \leq b_{\sbe}\right\}
\]
for some $a_{\sbe}, b_{\sbe} \in \mathbb{R}$. Define the projected space corresponding to $S_{\sbe}$ by
\[
\tobrh=\{r+\aldh: \aldh=kh, \vec{x}_0 +\beta \vec{\nu}_\perp+(r+\aldh) \vec{\nu} \in S_{\sbe}\},
\]
and define $\wrhb:\tobrh \times [0,\infty) \to \mathbb{R}$ by
\[
\wrhb \left(r+\aldh, t\right): = \uh(\vec{x}_0+\beta \vec{\nu}_\perp +(r+\aldh) \vec{\nu} , t).
\]
For a fixed $\beta$, the previous convergence result holds for $\wrhb$. Moreover, we show that as $\beta, r, h \to 0$, $\wrhb$ converges to $w$.

\begin{thm}[Convergence on $\overline{\Omega}$, informal summary of Theorem \ref{thm:convergelbeta}]
Let $\vx_0\in \overline{\Omega}$ and assume the setting above. Let $\alpha \in [a, b]$ and $t \geq 0$. Let $h, \rh> 0$ and $\beh$ be sufficiently small, and suppose that $\rh + \alpha_{\V} \in \tobhrhh$ with
\[
 \lim_{h \to 0^+} \rh = 0,
\quad \text{and} \quad
\lim_{h \to 0^+} \alpha_{\V} = \alpha.
\]
Then
\[
\lim_{h\to 0^+}\wrhhbh(\rh+ \alpha_{\V}, t)= w(\alpha, t),
\]
where $w$ is the unique solution of \eqref{HJE_NM}.


\end{thm}


The proof idea is as follows. Using the nonexpansive property of the resolvent problem (see Propositions \ref{prop_nonexp_d} and \ref{prop:uhproperty}), we obtain estimates for the difference of $\wrhb$ and $\wrh$. In particular, this difference is small when the numbers of grid points in $\tobrh$ and $\torh$ are close, and $r$, $\beta$, and $h$ are sufficiently small. Combining this with the earlier convergence result of $\wrh$ to $w$, we establish the convergence of $\wrhb$ to $w$.

Once the convergence of $\wrhb$ to $w$ is established, we use the variational representation of $w$ and apply Varadhan's inverse lemma to derive the large deviation principle for the rescaled process at a fixed time, as summarized below.

\begin{thm}[LDP on $\overline{\Omega}$, informal summary of Theorem \ref{thm:LDPOmega}]
Let $\vx_0 \in \overline{\Omega}$ and $h>0$. Suppose $\vxzh \to \vx_0$. Then the rescaled process $\cv(t)$ with initial position $\vxzh \in \overline{\Omega}$ satisfies LDP with the associated good rate function  given by
\[
\tilde{I}(\vec{y}; \vx_0, t) := I\!\left(\frac{(\vec{y}-\vx_0)\cdot \vec{\nu}}{|\vec{\nu}|^2}; 0, t\right),
\]
where $I$ is defined as
\[
I(y; \alpha, t) := \inf \left\{\left. \int_0^t \tilde{L}\big(\eta(s), v(s)\big) \, ds \, \right| \, \eta(t) = y,\; (\eta, v, l) \in \mathrm{SP}(\alpha) \right\}.
\]
Here, $\tilde{L}$ is the Legendre transform of $\tilde{H}$, and $\mathrm{SP}(\alpha)$ denotes the family of solutions to the Skorokhod problem (see \eqref{eqn:1dpath}).
\end{thm}

\subsection{Organization of this paper}
In Section \ref{sec:WKBdistHJE}, we present the WKB reformulation of the backward equation for the rescaled process $\cv$, yielding a discrete state-constraint HJE, and provide a variational representation for its solution. Section \ref{sec:restrict1d} focuses on the case of a single chemical reaction, where the starting point $\vxzh$ of $\cv$ lies in the maximal closed line segment $S_0$ containing the limiting starting point $\vx_0$, oriented along the reaction vector $\vec{\nu}$ representing the net molecular number change. We show that the discrete backward HJE with state constraints on $S_0$ converges to the continuous HJE on $S_0$ with a Neumann boundary condition. In Section \ref{sec:1dLDP}, we establish the large deviation principle for $\cv$ with $\vxzh \in S_0$ at any fixed time. Finally, in Section \ref{sec:ubetaconvergence}, we extend both the convergence result and the large deviation principle to the rescaled process $\cv$ for any single chemical reaction with starting point $\vxzh \in \overline{\Omega}$.

\subsection{Notations}

We provide the following table for the notations used throughout this paper.

\begin{table}[H]
\centering
\begin{tabular}{ll}
\toprule
\textbf{Symbol} & \textbf{Description} \\
\midrule
$\Omega$ & Domain of the problem \\
$h$ & Grid size (inverse of container size) \\
$\Omega_h = \left\{\vec{i}h : \vec{i} \in \mathbb{N}^N \right\} \cap\overline{\Omega}$ & Discrete domain of grid points \\
$\vec{\nu}_j \in \mathbb{Z}^N$ & Net change vector of the molecular counts for the $j$-th reaction \\
$\Phi_j^\pm(\vec{x})$ & Reaction intensity function of the $j$-th reaction \\
$\uh$ & WKB-transformed solution to the discrete backward HJE \\
$\Hh$& Hamiltonian in the equation of $\uh$\\
$u$ & Limiting continuous solution as $h \to 0$ \\
$\vx_0 \in \mathbb{R}^N$ & Limiting initial position vector \\
$[a, b] \subset\mathbb{R}$ & The largest interval for which $\left\{\vec{x}_0 + \alpha \vec{\nu}: a\leq \alpha \leq b\right\} \subset\overline{\Omega}$\\
$S_0$ & $S_0 = \{\vec{x}_0 + \alpha \vec{\nu} : a \leq \alpha \leq b \}$ \\
$\horh$ & $\horh=\left\{\vec{x}_0 +(r+\alpha) \vec{\nu} \in S_0: \alpha =kh, k  \in \mathbb{Z}\right\}$\\
$\torh$ & $\torh =\{r+\aldh:\aldh = kh, k \in \mathbb{Z}, \vx_0 +(r+\aldh) \vec{\nu} \in S_0\}$\\
$\wrh:\torh \times [0,\infty) \to \mathbb{R}$ & $\wrh \left(r+\aldh, t\right)= \uh(\vec{x}_0+(r+\aldh) \vec{\nu} , t)$\\
$\tHrh$ & Hamiltonian in the equation of $\wrh$\\
$\overline{w}$ & Upper semicontinuous envelope of $\wrh$\\
$\underline{w}$ & Lower semicontinuous envelope of $\wrh$\\
$w :[a, b] \times [0, \infty) \to \mathbb{R}$& Limit of $\wrh$\\
$\tilde{H}$& Hamiltonian in the equation of $w$\\
$[a_{\sbe}, b_{\sbe}] \subset\mathbb{R}$ & The largest interval for which $\left\{\vec{x}_0 +\beta\vec{\nu}_{\perp}+ \alpha \vec{\nu}: a_{\sbe}\leq \alpha \leq b_{\sbe}\right\} \subset\overline{\Omega}$\\
$S_{\sbe}$ & $S_{\sbe}=\left\{\vec{x}_0 +\beta\vec{\nu}_{\perp}+ \alpha \vec{\nu}: a_{\sbe}\leq \alpha \leq b_{\sbe}\right\}$\\
$\hobrh$ & $\hobrh=\left\{\vec{x}_0 +\beta \vec{\nu}_\perp+(r+\alpha) \vec{\nu} \in S_{\sbe}: \alpha =kh, k  \in \mathbb{Z}\right\}$\\
$\tobrh$&$\tobrh=\{r+\aldh: \aldh=kh, \vec{x}_0 +\beta \vec{\nu}_\perp+(r+\aldh) \vec{\nu} \in S_{\sbe}\}$\\
$\wrhb:\tobrh \times [0,\infty) \to \mathbb{R}$ & $\wrhb \left(r+\aldh, t\right) = \uh(\vec{x}_0+\beta \vec{\nu}_\perp +(r+\aldh) \vec{\nu} , t)$\\
$\tHrhb$ & Hamiltonian in the equation of $\wrhb$\\
$\overline{w}^{\sbe}$ & Upper semicontinuous envelope of $\wrhb$ \\
$\underline{w}^{\sbe}$ & Lower semicontinuous envelope of $\wrhb$ \\
$w^{\sbe} : \left[a_{\sbe}, b_{\sbe}\right] \times [0, \infty) \to \mathbb{R}$ & Limit of $\wrhb$ for fixed $\beta>0$\\
$\tilde{H}_{\sbe}$ & Hamiltonian in the equation of $w^{\sbe}$\\
\bottomrule
\end{tabular}
\caption{Table of notations used throughout the paper.}
\label{tab:notation}
\end{table}

\section{WKB reformulation and discrete HJE with state constraints}\label{sec:WKBdistHJE}

Chemical reactions involving $\ell=1,\cdots,N$ species $X_\ell$ and $j=1,\cdots,M$ reactions can be kinematically described as
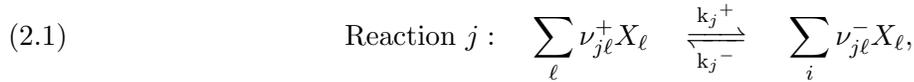
\begin{equation}\label{CRCR}
\text{Reaction }j: \quad \sum_{\ell} \nu_{j\ell}^+ X_{\ell} \quad \ce{<=>[k_j^+][k_j^-]} \quad   \sum_i \nu_{j\ell}^- X_\ell,
\end{equation}
where the nonnegative integers $\nu_{j\ell}^\pm \geq 0$ are stoichiometric coefficients and $k_j^\pm\geq 0$ are the reaction rates for the $j$-th forward and backward reactions. The column vector $\vec{\nu}_j:= \vec{\nu}_j^- - \vec{\nu}_j^+ := \bbs{\nu_{j\ell}^- - \nu_{j\ell}^+}_{\ell=1:N}\in \bZ^N$ is called the reaction vector for the $j$-th reaction, counting the net change in molecular numbers for species $X_\ell$. Let $\mathbb{N}$ be the set of natural numbers including zero, which serves as the state space for the counting process $X_\ell(t)$, representing the number of each species $\ell=1,\cdots,N$ in the biochemical reactions. In this paper, all vectors $\vec{X}= \bbs{X_i}_{i=1:N} \in \mathbb{N}^N$ and $ \bbs{k_j}_{j=1:M}\in \bR^M$ are column vectors.

With the reaction container size $1/h\gg 1$, the process $\cv_\ell(t)=hX_\ell(t)$ satisfies the random time-changed Poisson representation for chemical reactions \eqref{Csde} (see \cite{kurtz1980representations, Kurtz15}):
\begin{equation} \label{Csde}
\begin{aligned}
\cv(t) = \cv(0) + \sum_{j=1}^M  \vec{\nu}_{j} h \Bigg(  Y^+_j  \left(\frac{1}{h}\int_0^t {\Phi}^+_j(\cv(s))\ud s\right)
- Y^-_j  \left(\frac{1}{h}\int_0^t {\Phi}^-_j(\cv(s))\ud s\right)\Bigg).
\end{aligned}
\end{equation}
Here, for the $j$-th reaction channel, $Y^\pm_{j}(t)$ are i.i.d. unit-rate Poisson processes, and $\Phi_j^\pm(\cv)$ is the  intensity function. One example of $\Phi_j^\pm(\cv)$ is given by the macroscopic law of mass action (LMA):
\begin{equation}\label{lma}
\Phi^\pm_j(\vx) = k^\pm_j  \prod_{\ell=1}^{N} \bbs{x_\ell}^{\nu^\pm_{j\ell}}.
\end{equation}

\subsection{Physical domain and state constrained backward equation}
Due to the physical meaning of the counting process, the natural ambient domain for a chemical reaction is the lattice with non-negative constraint
\begin{equation}
\cv \in   \{ \vxi = \vec{i} h; \vec{i} \in \mathbb{N}^N \}.
\end{equation}

Additional physical constraints may arise from conservation laws. In a closed system, for example, consider the conservation of mass in a chemical reaction. Denote $\vec{m}=(m_\ell)_{\ell=1:N}$, where $m_\ell$ represents the molecular weight for the $\ell$-th species. Then the conservation of mass for the $j$-th reaction in \eqref{CRCR} implies
\begin{equation}\label{mb_j}
\vec{\nu}_j \cdot \vec{m} = 0, \quad j=1,\cdots, M.
\end{equation}
Thus we have the conservation law for the states
$\sum_\ell m_\ell \cv_\ell(t) = \text{const},$
and  consider domain
\begin{equation}
      \{ \vxi = \vec{i} h; \,\,  \vec{i} \in \mathbb{N}^N,\,  \vxv \cdot \vec{m} = \text{const}\}.
\end{equation}

In an open system, where materials exchange with the environment (denoted $\emptyset$), the exchange reaction $X_{\ell}\ce{<=>} \emptyset$ does not conserve mass. But we still restrict the states in a bounded domain due to the capacity of the environment. For example, in  open systems, one
can constrain the amount of a chemostat species that is  transferred into or out  of the reaction system.

Moreover, certain chemical reactions must satisfy strict compositional constraints to ensure both stability and precision in manufacturing. For example, in the epitaxial growth process, maintaining an exact ratio of the constituent elements is essential to achieve the desired crystal structure and material quality.

These physical considerations naturally motivate the study of the forward and backward equations, as well as the associated discrete HJE, within a bounded domain $\Omega_{\V}$ with state-constraint conditions; details are provided in the next section.

\subsubsection{Working domain}
From now on, we consider a generic bounded open  domain
$\Omega \subset \bR^N_+:=   \{x\in \bR^N:\, x_i > 0 \text{ for all } i\}$ such that $\overline{\Omega} \subset \bR^N_+$. We also choose the discrete domain for the rescaled chemical reaction process as
\begin{equation}\label{eqn:defomegah}
\Omegah := \left\{ \vxi = \vec{i} h; \,\,  \vec{i} \in \mathbb{N}^N,\, \vxi \in \overline{\Omega}\right\}.
\end{equation}

We point out that  $\Omegah \to \Omega$ as $h\to 0^+$ in the sense that $\forall \vx\in \Omega$, there exists $\vxv\in \Omegah$ such that $\vxv\to \vx.$
Moreover,  from LMA \eqref{lma}, there exists $c_0>0$ such that
\begin{equation}\label{eq:lower}
   \Phi_j^\pm(x) \geq c_0>0, \quad \forall x\in \overline{\Omega}.
\end{equation}

\subsubsection{Forward and backward equation in $\Omegah$}
In $\Omegah$, we construct the state-constraint process \eqref{CRCR} such that
the chemical reactions only happen within domain $\Omegah$. For $\vxi\in \Omegah$, let $\zh(\vxv,t)$  satisfy the corresponding backward equation (see {\cite[Appendix]{GL22}} for the derivation of  the generator $Q$)
\begin{equation}
\begin{aligned}
\partial_t \zh(\vxi,t) =&(Q \zh)(\vxi,t)\\
:=& \frac{1}{h}\sum_{j=1, \vxv+ \vec{\nu}_{j}h\in\Omegah}^M \Phi^+_j(\vxv)\left[ \zh(\vxv+ \vec{\nu}_{j}h)  - \zh(\vxv) \right] \\
 &+ \frac{1}{h}\sum_{j=1, \vxv- \vec{\nu}_{j}h\in \Omegah}^M \Phi_j^-(\vxv)\left[ \zh(\vxv- \vec{\nu}_{j}h )-\zh(\vxv) \right], \quad \vxv \in \Omegah.
\end{aligned}
\end{equation}

For any $f\in C_b(\Omegah)$, the solution to the backward equation, with the initial condition $\zh(\vxv, 0)= f(\vxv)$ for all $\vxv \in \Omegah$, can be expressed as
\begin{equation}
\zh(\vxv, t) = \mathbb{E}^{\vxv}[f(\cv_t)].
\end{equation}

Denote the time marginal law of $\cv(t)$ as $p_{\V}(\vxv,t) = \bE(\mathbbm{1}_{\vxv}(\cv(t)))$, where $\mathbbm{1}_{\vxv}$ is the indicator function. Then $p_{\V}(\vxv,t)$, for $\vxi\in \Omegah$, satisfies  the forward   equation   \cite{Kurtz15, GL22}:
\begin{equation}\label{rp_eq}
\begin{aligned}
\frac{\ud}{\ud t} p_{\V}(\vxi, t) =& \frac{1}{h}\sum_{j=1, \vxi- \vec{\nu}_{j} h\in \Omegah}^M  \left( \Phi^+_j(\vxi- \vec{\nu}_{j} h) p_{\V}(\vxi- \vec{\nu}_{j} h,t) - \Phi^-_j(\vxi) p_{\V}(\vxi,t) \right) \\
& + \frac{1}{h}\sum_{j=1, \vxi+\vec{\nu}_{j}h\in \Omegah}^M \left( \Phi^-_j(\vxi+\vec{\nu}_{j}h) p_{\V}(\vxi+\vec{\nu}_{j}h,t) - \Phi_j^+(\vxi) p_{\V}(\vxi,t) \right), \quad \text{for } \vxi\in \Omegah.
\end{aligned}
\end{equation}
It is easy to check that we have the total mass conservation in $\Omegah$.
\begin{lem}\label{lem_mass}
We have the conservation of total probability in $\Omegah$:
$$
    \frac{\ud}{ \ud t} \sum_{\vxv \in \Omegah} p_{\V}(\vxi, t) = 0.
    $$
\end{lem}
\begin{proof}
    From \eqref{rp_eq}, we have
    \[
\begin{aligned}
\frac{\ud}{\ud t} \sum_{\vxv \in \Omegah} p_{\V}(\vxi, t) =& \frac{1}{h}\sum_{\vxv \in \Omegah}\sum_{j=1, \vxi- \vec{\nu}_{j} h\in \Omegah}^M  \left( \Phi^+_j(\vxi- \vec{\nu}_{j} h) p_{\V}(\vxi- \vec{\nu}_{j} h,t) - \Phi^-_j(\vxi) p_{\V}(\vxi,t) \right) \\
& + \frac{1}{h}\sum_{\vxv \in \Omegah}\sum_{j=1, \vxi+\vec{\nu}_{j}h\in \Omegah}^M \left( \Phi^-_j(\vxi+\vec{\nu}_{j}h) p_{\V}(\vxi+\vec{\nu}_{j}h,t) - \Phi_j^+(\vxi) p_{\V}(\vxi,t) \right).
\end{aligned}
    \]
By letting $\vec{y}_i=\vxv +\vec{\nu}_jh$ in the second line above, we obtain
\[
\begin{aligned}
\frac{\ud}{\ud t} \sum_{\vxv \in \Omegah} p_{\V}(\vxi, t) =& \frac{1}{h}\sum_{\vxv \in \Omegah}\sum_{j=1, \vxi- \vec{\nu}_{j} h\in \Omegah}^M  \left( \Phi^+_j(\vxi- \vec{\nu}_{j} h) p_{\V}(\vxi- \vec{\nu}_{j} h,t) - \Phi^-_j(\vxi) p_{\V}(\vxi,t) \right) \\
& + \frac{1}{h}\sum_{\vec{y}_i\in \Omegah}\sum_{j=1, \vec{y}_i-\vec{\nu}_jh \in \Omegah}^M  \left( \Phi^-_j(\vec{y}_i) p_{\V}(\vec{y}_i,t) - \Phi_j^+(\vec{y}_i-\vec{\nu}_jh) p_{\V}(\vec{y}_i-\vec{\nu}_jh,t) \right)\\
&=0.
\end{aligned}
\]
\end{proof}

We remark that a direct consequence of Lemma \ref{lem_mass} is the exponential tightness of the process $\cv(t)$ for each $t\in[0,T]$, i.e.,
for any $0<\ell<\8$, there exists a $R_\ell$ and $h_0$ such that,
\begin{equation}\label{eq:exp_tight}
\sup_{t\in[0,T]}\sum_{|\vxi|\geq R_{\ell}, \vxi\in\Omegah} p_{\V}(\vxi,t) \leq e^{-\frac{\ell}{h}}, \quad \forall h\leq h_0.
\end{equation}

\subsection{WKB reformulation in backward equation and discrete Hamiltonian $\Hh$}
WKB expansion is widely used to
characterize the exponential asymptotic behavior. To prove the  LDP rigorously, we make a change of variable in the backward equation with
  $$\uh(\vxi,t):= h\log \zh(\vxi,t),  \quad \vxi\in\Omega_{\V}.$$ Then $\uh$ solves the discrete backward HJE \eqref{upwind0}:
  \[
  \begin{aligned}
\pt_t \uh(\vxv,t) =& \sum_{j=1, \vxv+ \vec{\nu}_{j} h\in \Omegah }^M     \Phi^+_j(\vxv)\bbs{ e^{\frac{\uh\left(\vxv+ \vec{\nu}_{j} h,t\right)-\uh\left(\vxv, t\right)}{h}} - 1} \\
&+ \sum_{j=1,  \vxv- \vec{\nu}_{j}h\in\Omegah}^M \Phi_j^-(\vxv)\bbs{ e^{  \frac{\uh\left(\vxv- \vec{\nu}_{j}h,t\right)-\uh\left(\vxv,t\right)}{h} } - 1}, \quad \vxv\in \Omegah.
\end{aligned}
  \]

Formally, $\uh \to u$ where $u$ is a viscosity solution of the continuous HJE \eqref{eqn:cHJE}, i.e.,
\[
\partial_t u(\vx,t) =  \sum_{j=1}^M \left[ \Phi^+_j(\vec{x})\left(e^{\vec{\nu}_j\cdot \nabla u(\vx, t)}-1\right)     +   \Phi_j^-(\vec{x})\left( e^{-\vec{\nu}_j\cdot \nabla u(\vx, t)}-1\right)  \right] = H(\nabla u(\vx), \vx), \,\,\vx\in\Omega.
\]

\subsubsection{Monotone scheme resulted from the backward equation}

Define the discrete Hamiltonian operator corresponding to the backward equation
\begin{equation}\label{def_H}
\begin{aligned}
\Hh: \ell^{\8}(\Omegah) \to \ell^{\8}(\Omegah); \qquad  ( u(\vxi) )_{\vxi\in\Omegah} \mapsto ( \Hh(\vxi, u(\vxi), u(\cdot)) )_{\vxi\in\Omegah}
\end{aligned}
\end{equation}
\begin{equation}\label{def:distH}
\begin{aligned}
\Hh(\vxi, \uh(\vxi), \varphi):=& \sum_{j=1, \vxv+ \vec{\nu}_{j} h\in \Omegah }^M     \Phi^+_j(\vxv)\bbs{ e^{\frac{\varphi\left(\vxv+ \vec{\nu}_{j} h\right)-\uh\left(\vxv\right)}{h}} - 1} \\
&+ \sum_{j=1,  \vxv- \vec{\nu}_{j}h\in\Omegah}^M \Phi_j^-(\vxv)\bbs{ e^{  \frac{\varphi\left(\vxv- \vec{\nu}_{j}h\right)-\uh\left(\vxv\right)}{h} } - 1}.
\end{aligned}
\end{equation}

It is straightforward to verify that $\Hh$ satiefies the monotone scheme condition, i.e.,
\begin{equation}\label{mono}
\text{if }\,\,  \varphi_1\leq \varphi_2\quad \text{ then }\,\, \Hh(\vx, \uh(\vx), \varphi_1) \leq \Hh(\vx,\uh(\vx), \varphi_2).
\end{equation}

Denote the time step as $\Delta t$. Then the backward Euler time discretization gives
\begin{equation}\label{bEuler}
\frac{\uh^{\sn} -\uh^{\sn-1}}{\Delta t} - \Hh(\vxi, \uh^{\sn}(\vxi), \uh^{\sn})=0.
\end{equation}
Then $\uh^{\sn}=(I-\Delta t \Hh)^{-1} \uh^{\sn-1} =: J_{\Delta t, h} \uh^{\sn-1}$ can be solved as the solution $\uh$ to the discrete resolvent problem with $f=\uh^{\sn-1}$
\begin{equation}\label{dis_resolvent}
\uh(\vxi) - \Delta t \Hh(\vxi, \uh(\vxi), \uh)=f(\vxi).
\end{equation}

We state the existence of a solution $\uh$ to the resolvent problem \eqref{dis_resolvent}. The proof follows essentially the same argument based on Perron's method as in \cite[Lemma 3.2]{GaoLiu2023} and is therefore omitted.

\begin{prop}[{Existence of the resolvent problem \eqref{dis_resolvent}}]\label{prop:exresolvent}
Let $\Delta t >0$ and $f \in \ell^{\8}(\Omegah)$. Then there exists a solution $\uh \in \ell^{\8}(\Omegah)$ to \eqref{dis_resolvent}, and $\uh$ satisfies
\[
\inf_{\vx_j\in \Omegah}f(\vx_j) \leq \uh(\vxv) \leq \sup_{\vx_j\in \Omegah}f(\vx_j),
\]
for all $\vxv \in \Omegah$.
\end{prop}

Next, we state the nonexpansive property of the solutions to the discrete resolvent problems, which guarantees the uniqueness of the solution to \eqref{dis_resolvent}. The proof is omitted as it follows the same argument as in \cite[Lemma 3.1]{GaoLiu2023}.

\begin{prop}[{Uniqueness of the resolvent problem  \eqref{dis_resolvent}}]\label{prop_nonexp_d}
 Let $f, g \in \ell^{\8}(\Omegah)$. Let $\uh$ and $\vh$ be two solutions satisfying $\uh= J_{\Delta t, h} f$ and $\vh=J_{\Delta t, h} g$. Then we have
\begin{enumerate}[(i)]
\item Monotonicity: if $f\leq g$, then $\uh\leq \vh$;
\item Nonexpansive:
\begin{equation}\label{nonexp_d}
\|\uh-\vh\|_{\8} \leq \|f-g\|_{\8}.
\end{equation}
\end{enumerate}
\end{prop}


The property (ii) shows the resolvent $J_{\Delta t, h} = (I - \Delta t \Hh)^{-1}$ is nonexpansive. Thus we also obtain $-\Hh$ is an accretive operator on $\ell^\8.$

The following is an elementary property of the solution $\uh$, which follows directly from \cite[Lemma 1.2]{CrandallLiggett}. We include a direct proof below, as it is straightforward in our setting.

\begin{prop}\label{prop:uhproperty}
 Let $f \in \ell^{\8}(\Omegah)$ and let $\uh$ be the solution satisfying $\uh= J_{\Delta t, h} f$. Then
\[
\left\|\Hh \uh\right\|_{\ell^\8 \left(\Omegah\right)}= \frac{\left\|\uh-f\right\|_{\ell^\8 \left(\Omegah\right)}}{\Delta t} \leq \left\|\Hh f\right\|_{\ell^\8 \left(\Omegah\right)}.
\]
\end{prop}

\begin{proof}
From Proposition \ref{prop_nonexp_d}, we have
\[
\begin{aligned}
 \left\|\uh-f\right\|_{\ell^\8 \left(\Omegah\right)}  & = \left\| (I - \Delta t \Hh)^{-1}f- (I - \Delta t \Hh)^{-1}(I - \Delta t \Hh)f\right\|_{\ell^\8 \left(\Omegah\right)} \\
 & \leq \left\|f-(I - \Delta t \Hh)f\right\|_{\ell^\8 \left(\Omegah\right)}\\
 &=\Delta t \left\|\Hh f\right\|_{\ell^\8 \left(\Omegah\right)}.
\end{aligned}
\]
Since $\displaystyle \Hh \uh=\frac{\uh-f}{\Delta t}$, hence the conclusion follows.
\end{proof}

\subsection{Existence and uniqueness of discrete HJE with state constraints}

We establish the well-posedness of the solutions $\uh$ to the discrete Hamilton--Jacobi equations \eqref{upwind0}, as well as some of their properties, in the following proposition. The proof proceeds similarly to that of \cite[Theorem 3.4]{GaoLiu2023}.

\begin{prop}[Well-posedness of the discrete HJE \eqref{upwind0}]\label{prop:propdis}
Let $h>0$ and $u^0 \in \ell^\infty(\Omegah)$. Then there exists a unique solution $\uh \in C\left([0, \infty);\ell^\infty \left(\Omegah\right)\right)$ to \eqref{upwind0} with the initial condition $\uh(\vxv,0) = u^0(\vxv)$ for all $\vxv \in \Omegah$, given by
\[
\uh(\cdot, t)=\lim_{\Delta t \to 0} (I-\Delta t \Hh)^{-\lfloor\frac{t}{\Delta t} \rfloor} u^0,
\]
where $\lfloor\frac{t}{\Delta t} \rfloor$ is the greatest integer less than or equal to $\frac{t}{\Delta t}$. Moreover, $\uh$ satisfies the following properties.
\begin{itemize}
    \item[(i)]  Suppose $\tuh \in C\left([0, \infty);\ell^\infty \left(\Omegah\right)\right)$ is the solution to \eqref{upwind0} with the initial value $\tuh(\vxv, 0) =\tilde{u}^0(\vxv)$. Then
    \[
    \left\|\uh(\cdot, t) -\tuh(\cdot, t)\right\|_{ l^\infty(\Omegah)} \leq \left\|u^0-\tilde{u}^0\right\|_{ l^\infty(\Omegah)}.
    \]
    \item[(ii)] Contraction:
    \begin{equation} \label{eqn:discontra}
        \inf_{\vxv \in \Omegah}u^0(\vxv) \leq \inf_{\vxv \in \Omegah} \uh(\vxv, t) \leq \sup_{\vxv \in \Omegah} \uh(\vxv, t) \leq \sup_{\vxv \in \Omegah}u^0(\vxv).
    \end{equation}
    \item [(iii)]$\uh\in C^1\left([0, \infty);\ell^\infty \left(\Omegah\right)\right)$.
\end{itemize}
\end{prop}

\begin{proof}
From Proposition \ref{prop:exresolvent} and Proposition \ref{prop_nonexp_d}, we have that $-\Hh$ is a maximal accretive operator on $l^\infty(\Omegah)$. Then, by Crandall and Liggett \cite{CrandallLiggett}, there exists a unique global contraction $C^0$-semigroup solution $\uh (t, \vxv) \in C\left([0,\infty);l^\infty(\Omegah)\right)$ to \eqref{upwind0}, given by
\[
\uh(\cdot,t)=\lim_{\Delta t \to 0} (I-\Delta t \Hh)^{-\lfloor\frac{t}{\Delta t} \rfloor} u^0.
\]
For any $x\in \vxv$ and $t, s \in [0, \infty)$, we have
\begin{equation}\label{eqn:disintergral}
\uh(\vxv,t)-\uh(\vxv,s)=\int_s^t \Hh \left(\vxv,\uh( \vxv,\tau),\uh(\tau)\right)\,d\tau.
\end{equation}
Hence, $\uh(\vxv, \cdot) \in C^1[0, \infty)$.
 \begin{itemize}
     \item[(i)] From the nonexpansive property of $J_{\Delta t,h}$ in Proposition \ref{prop_nonexp_d}, we have
\[
\left\| (I - \Delta t \Hh)^{-1}u^0- (I - \Delta t \Hh)^{-1}\tilde{u}^0\right\|_{l^\infty(\Omegah)} \leq \|u^0-\tilde{u}^0\|_{l^\infty(\Omegah)},
\]
which implies
\[
\begin{aligned}
\left\|\uh(\cdot,t) -\tuh(\cdot,t)\right\|_{l^\infty(\Omegah)}&=\left\|\lim_{\Delta t\to 0} (I-\Delta t \Hh)^{-\lfloor\frac{t}{\Delta t} \rfloor} u^0 -\lim_{\Delta t\to 0} (I-\Delta t \Hh)^{-\lfloor\frac{t}{\Delta t} \rfloor} \tilde{u}^0 \right\|_{l^\infty(\Omegah)}\\
&\leq \|u^0-\tilde{u}^0\|_{l^\infty(\Omegah)}.
\end{aligned}
\]
\item[(ii)] The contraction \eqref{eqn:discontra} follows directly from (i).
\item[(iii)] From (ii) and \eqref{eqn:disintergral},
we deduce that for any $x\in \vxv$ and $t, s \in [0, \infty)$
\begin{equation} \label{eqn:distbd_h}
\left\|\uh(\cdot,t) -\uh(\cdot,s)\right\|_{l^\infty(\Omegah)}\leq |t-s| C,
\end{equation}
where $C=C(h,\|u^0\|_{l^\infty(\Omegah)}, \Omega, \Phi^+,\Phi^-)>0$.
Moreover, we have the estimate
\begin{equation}\label{eqn:Hhdifest}
\begin{aligned}
&\left|\Hh\left(\vxv,\uh(\vxv,s), \uh(s)\right)-\Hh\left(\vxv,\uh(\vxv,t), \uh(t)\right)\right|\\
\leq &\sum_{j=1, \vxv+ \vec{\nu}_{j} h\in \Omegah }^M     \Phi^+_j(\vxv) \left| e^{\frac{\uh\left(\vxv+ \vec{\nu}_{j} h,s\right)-\uh\left(\vxv, s\right)}{h}} - e^{\frac{\uh\left(\vxv+ \vec{\nu}_{j} h,t\right)-\uh\left(\vxv, t\right)}{h}}\right| \\&\qquad \qquad \qquad + \sum_{j=1,  \vxv- \vec{\nu}_{j}h\in\Omegah}^M \Phi_j^-(\vxv)\left| e^{\frac{\uh\left(\vxv- \vec{\nu}_{j} h,s\right)-\uh\left(\vxv, s\right)}{h}} - e^{\frac{\uh\left(\vxv- \vec{\nu}_{j} h,t\right)-\uh\left(\vxv, t\right)}{h}}\right|\\
\leq&C \left\|\uh(\cdot, s)-\uh(\cdot, t)\right\|_{l^\infty(\Omegah)}\\
\leq &C|t-s|,
\end{aligned}
\end{equation}
for some constant $C=C(h,\|u^0\|_{l^\infty(\Omegah)}, M, \Omega, \Phi^+,\Phi^-)>0$. The second inequality follows from (ii) and the continuity of $\Phi^+, \Phi^-$, while the last inequality follows from \eqref{eqn:distbd_h}.

Now, fix $t > 0$ and let $\tau \in \mathbb{R}$ be sufficiently small. For any $\vxv \in \Omegah$, we have
\[
\begin{aligned}
&\left|\frac{\uh\left(\vxv, t+\tau\right)-\uh\left(\vxv, t\right)}{\tau}-\Hh\left(\vxv,\uh(\vxv,t), \uh(t)\right)\right|\\
\leq &\frac{1}{|\tau|}\left|\int_t^{t+\tau}\left(\Hh\left(\vxv,\uh(\vxv,s), \uh(s)\right)-\Hh\left(\vxv,\uh(\vxv,t), \uh(t)\right)\right)\, ds \right|\\
\leq &C|\tau|,
\end{aligned}
\]
where the last inequality follows from \eqref{eqn:Hhdifest}, and $C = C(h, \left|u^0\right|_{l^\infty(\Omegah)}, M, \Omega, \Phi^+, \Phi^-) > 0$.

Therefore, $\uh \in C^1\left((0, \infty); \ell^\infty(\Omegah)\right)$. The case $t = 0$ follows similarly by taking $\tau > 0$, completing the proof.
 \end{itemize}
\end{proof}

The following proposition provides a uniform bound on the time derivative independent of $h$ when the initial condition for all $h$ is given by a function $u_0\in W^{1,\infty} (\Omega)\cap C(\overline{\Omega})$. This result is a direct application of \cite[Lemma 2.2]{brezis_pazy_1970} in our setting.
\begin{prop}\label{prop:unibduhpt}
Assume either $\partial \Omega$ is Lipschitz continuous or $\Omega$ is convex.
Let $u_0 \in W^{1,\infty} (\Omega)\cap C(\overline{\Omega})$. For $h>0$, let $\uh \in C^1\left([0, \infty);\ell^\infty \left(\Omegah\right)\right)$ be the unique solution to \eqref{upwind0} with the initial condition $\uh(\vxv,0) = u_0(\vxv)$ for all $\vxv \in \Omegah$. Then there exists a constant $C_0=C_0 \left(\left\| Du_0\right\|_{L^\infty(\Omega)}, M, \Omega, \Phi^+, \Phi^-\right)>0$ independent of $\vxv$, $t$, and $h$, such that
\[\left|\pt_t \uh\left(\vxv,t\right)\right| \leq C_0
\]
for all $h>0$, $\vxv \in \Omegah$, and $t\geq 0$.
\end{prop}
\begin{proof}
Let $h>0$, $\vxv \in \Omegah$, and $t>0$. Since $u_0 \in W^{1,\infty} (\Omega)\cap C(\overline{\Omega})$, and either $\partial \Omega$ is Lipschitz continuous or $\Omega$ is convex, it follows that for any $j \in \left\{1, 2, \cdots, M\right\}$ such that $\vxv
+\vec{\nu}_{j} h\in \Omegah$, we have
\begin{equation}\label{eqn:ubdu0j+}
\left|\frac{u_0\left(\vxv+ \vec{\nu}_{j} h\right)-u_0\left(\vxv\right)}{h}\right| \leq C,
\end{equation}
for some constant $C=C\left(\left\|Du\right\|_{L^\infty(\Omega)}, \partial \Omega\right)>0$. Similarly, for any $j \in \left\{1, 2, \cdots, M\right\}$ such that $\vxv
-\vec{\nu}_{j} h\in \Omegah$, we also have
\begin{equation}\label{eqn:ubdu0j-}
  \left|\frac{u_0\left(\vxv-\vec{\nu}_{j} h\right)-u_0\left(\vxv\right)}{h}\right| \leq C,
\end{equation}
with the same constant $C$ as in \eqref{eqn:ubdu0j+}.
By \cite[Lemma 2.2]{brezis_pazy_1970}, we obtain the following estimate
\[
\begin{aligned}
\left\|\pt_t \uh\left(\cdot,t\right)\right\|_{\ell^\8 \left(\Omegah\right)} &\leq \left|\Hh u_0\right|\\
&=\left| \sum_{j=1, \vxv+ \vec{\nu}_{j} h\in \Omegah }^M     \Phi^+_j(\vxv)\bbs{ e^{\frac{u_0\left(\vxv+ \vec{\nu}_{j} h\right)-u_0\left(\vxv\right)}{h}} - 1} \right.\\
&\left. \qquad \qquad \qquad\quad + \sum_{j=1,  \vxv- \vec{\nu}_{j}h\in\Omegah}^M \Phi_j^-(\vxv)\bbs{ e^{  \frac{u_0\left(\vxv- \vec{\nu}_{j}h\right)-u_0\left(\vxv\right)}{h} } - 1}\right|\\
&\leq C_0,
\end{aligned}
\]
for some constant $C_0=C_0(\left\| Du_0\right\|_{L^\infty(\Omega)}, M, \Omega, \Phi^+, \Phi^-) >0$, where the last equality follows  from the bounds in \eqref{eqn:ubdu0j+}–\eqref{eqn:ubdu0j-} and the continuity of $\Phi^+, \Phi^-$. Thus, the results holds for any $h>0$, $\vxv \in \Omegah$, and $t>0$. The case $t=0$ follows from the fact that $\uh(\vxv, \cdot) \in C^1[0, \infty)$.
\end{proof}

\subsection{Variational solution representation for the discrete HJE}
In this section, we present a solution representation for the discrete HJE \eqref{upwind0}, formulated as an optimal control problem subject to a state constraint.

\subsubsection{Variational lemmas}
Let $h>0$ and $\vxv \in \Omegah$, where $i \in \Gamma_{\V}$ for some finite index set $\Gamma_{\V}$.
Define $\Hh:\Omegah \times \mathbb{R}^{\left|\Gamma_{\V}\right|} \to \mathbb{R}$ by
\begin{equation}
\begin{aligned}
    \Hh \left(\vec{x}_i, \vec{\phi}\right):&= \sum_{j=1, \,\vxv+ \vec{\nu}_{j} h\in \Omegah }^M     \Phi^+_j(\vxv)\bbs{ e^{\frac{\phi(\vxv+ \vec{\nu}_{j} h)-\phi(\vxv)}{h}} - 1} + \sum_{j=1,\,  \vxv- \vec{\nu}_{j}h\in\Omegah}^M \Phi_j^-(\vxv)\bbs{ e^{  \frac{\phi(\vxv- \vec{\nu}_{j}h)-\phi(\vxv)}{h} } - 1}\\
    &=\sum_{k\in \Gamma_{\V}} \Phi\left(\vxv, \vec{x}_k\right)\left(e^\frac{\phi\left(\vec{x}_k\right)-\phi\left(\vxv\right)}{h}-1\right),
\end{aligned}
\end{equation}
where $\Phi:\Omegah \times \Omegah \to \mathbb{R}$ is defined by
\begin{equation}\label{Phi_ik}
    \Phi\left(\vxv, \vec{x}_k\right):= \left\{\begin{aligned}&\Phi^+_{j(k)}\left(\vxv\right), \qquad \qquad \qquad \text{ if } \vec{x}_k=\vxv+\vec{\nu}_jh \text{ for some } j \in \{1, \cdots, M\},\\
    &\Phi^-_{j(k)}\left(\vxv\right), \qquad \qquad \qquad \text{ if } \vec{x}_k=\vxv-\vec{\nu}_jh \text{ for some } j \in \{1, \cdots, M\},\\
    &0, \qquad \qquad \qquad \qquad \quad \, \, \, \,\text{ if } k \neq i \text{ and } \vec{x}_k \notin \left\{\vxv \pm \vec{\nu}_jh: j= 1, \cdots, M \right\},\\
    & -\sum_{l\in\Gamma_{\V}, \vec{x}_l \in \left\{\vxv \pm \vec{\nu}_jh: j= 1, \cdots, M \right\}} \Phi(\vxv, \vec{x}_l), \qquad \text{ if } k = i.
    \end{aligned}
    \right.
\end{equation}

We then calculate the associated Lagrangian $\Lh: \Omegah \times \mathbb{R}^{|\Gamma_{\V}|} \to \mathbb{R}$ through
\begin{equation}\label{Legen}
    \Lh \left(\vxv, \vec{s} \right):= \sup_{\vec{\phi}} \left\{\vec{s} \cdot \vec{\phi} - \Hh \left(\vxv, \vec{\phi}\right)\right\}.
\end{equation}

If $\sum_{k\in \Gamma_{\V}} s_k \neq 0$, it can be seen that $\Lh(\vxv,\vec{s})=+\8$ due to $\Hh(\vxv,  \text{constant vector})=0.$ If $\sum_{k\in \Gamma_{\V}} s_k=0$ but there exists $s_{k_0}<0$ for some $k_0 \neq i$, we can demonstrate $\Lh(\vxv,\vec{s})=+\8$ by taking $\phi_{k}=0$ for $k \neq k_0$ and $\phi_{k_0}=-\8$.
As a consequence, for finite $\Lh$, we must have
$\sum_{k\in \Gamma_{\V}} s_k = 0$ and $s_k\geq 0$ for any $k\neq i$.

 Note that $s_i = -\sum_{k\in\Gamma_{\V}, k\neq i } s_k$. Let $\tilde{\phi} \in \mathbb{R}^{|\Gamma_{\V}|}$ and we differentiate with respect to $\eps$ in $g(\eps) :=\vec{s} \cdot(\vec{\phi}+\eps \tilde{\phi})-H(x, \vec{\phi}+\eps \tilde{\phi})$ and set $g^\prime(0)=0$, resulting in
\begin{equation}\label{tm3.6}
    \sum_{k\in\Gamma_{\V} }  s_k  \left(\tilde{\phi}_k-\tilde {\phi}_i\right) = \sum_{k\in\Gamma_{\V} }    \frac{1}{h}\Phi(\vxv, \vec{x}_k) e^{\frac{\phi_k-\phi_i}{h}} \left(\tilde{\phi}_k-\tilde{\phi}_i \right).
\end{equation}
Then
\[
\sum_{k\in\Gamma_{\V} }     \left(\frac{1}{h}\Phi(\vxv, \vec{x}_k) e^{\frac{\phi_k-\phi_i}{h}}  - s_k \right) \left(\tilde{\phi}_k-\tilde{\phi}_i \right)=0.
\]

So the optimal $\vec{\phi}$ which achieves the Legendre transformation \eqref{Legen} satisfies
\begin{equation}\label{ss}
s\left(\vxv, \vec{x}_k\right) := s_k=\left\{\begin{aligned}
   &\frac{1}{h}\Phi(\vxv, \vec{x}_k) e^{\frac{\phi_k-\phi_i}{h}}, \qquad \qquad \qquad \, \text{ if }\, \vec{x}_k \in \left\{\vxv \pm \vec{\nu}_jh: j= 1, \cdots, M \right\}, \\
&0, \qquad \qquad \qquad  \qquad \qquad \qquad \quad \, \text{ if } k \neq i \text{ and } \vec{x}_k \notin \left\{\vxv \pm \vec{\nu}_jh: j= 1, \cdots, M \right\}, \\
& -\sum_{l\in\Gamma_{\V}, \vec{x}_l \in \left\{\vxv \pm \vec{\nu}_jh: j= 1, \cdots, M \right\}} s(\vxv, \vec{x}_l), \quad \text{ if } k=i.\\
\end{aligned}
\right.
\end{equation}

Upon substituting this optimal $\vec{\phi}$ into the Legendre transformation at $\vxv \in\Omegah$, and for $\sum_j s_j=0$, we have
\begin{equation}
\begin{aligned}
    \Lh(\vxv, \vec{s}) = &  \sum_{k\in\Gamma_{\V}, \, \vec{x}_k \in \left\{\vxv \pm \vec{\nu}_jh: j= 1, \cdots, M \right\}} \left(\frac{1}{h}\Phi(\vxv, \vec{x}_k) e^{\frac{\phi_k-\phi_i}{h}} \left(\phi_k-\phi_i \right) -\Phi(\vxv, \vec{x}_k)\left(e^{\frac{\phi_k-\phi_i}{h}} -1\right) \right) \\
    =&\sum_{k\in\Gamma_{\V}, \, \vec{x}_k \in \left\{\vxv \pm \vec{\nu}_jh: j= 1, \cdots, M \right\}} \Phi(\vxv, \vec{x}_k)\left(h\frac{s\left(\vxv, \vec{x}_k\right)}{\Phi(\vxv, \vec{x}_k)} \log\left(h\frac{s\left(\vxv, \vec{x}_k\right)}{\Phi(\vxv, \vec{x}_k)}\right) -h\frac{s\left(\vxv, \vec{x}_k\right)}{\Phi(\vxv, \vec{x}_k)} +1 \right)\\
    =& \sum_{j=1,\, \vxv+ \vec{\nu}_{j} h\in \Omegah }^M \Phi^+_j(\vxv) \left(h\frac{s\left(\vxv, \vxv+\vec{\nu}_jh\right)}{\Phi^+_j(\vxv)} \log\left(h\frac{s\left(\vxv, \vxv+\vec{\nu}_jh\right)}{\Phi^+_j(\vxv)}\right) -h\frac{s\left(\vxv, \vxv+\vec{\nu}_jh\right)}{\Phi^+_j(\vxv)}+1\right)  \\
    &+\sum_{j=1,\, \vxv- \vec{\nu}_j h\in \Omegah }^M \Phi^-_j(\vxv) \left(h\frac{s\left(\vxv,\vxv- \vec{\nu}_j h\right)}{\Phi^-_j(\vxv)} \log\left(h\frac{s\left(\vxv, \vxv- \vec{\nu}_j h\right)}{\Phi^-_j(\vxv)}\right) -h\frac{s\left(\vxv, \vxv- \vec{\nu}_j h\right)}{\Phi^-_j(\vxv)}+1\right)\\
    \geq & 0.
\end{aligned}
\end{equation}
Redefine $\Lh$ in terms of the control velocity $\vec{v} \in \mathbb{R}^{|\Gamma_{\V}|}$ where
\begin{equation}\label{eqn:v_tphi}
v(\vxv, \vec{x}_k) = \left\{\begin{aligned}
&\frac{s\left(\vxv, \vec{x}_k\right)}{\frac{1}{h}\Phi(\vxv, \vec{x}_k)}= e^\frac{\phi_k-\phi_i}{h}, \quad \qquad \qquad \qquad \qquad \qquad \qquad \qquad \quad \,  \text{ if }\vec{x}_{k} \in \left\{\vxv \pm \vec{\nu}_jh: j= 1, \cdots, M \right\}, \\
& \frac{s(\vxv, \vxv)}{\frac{1}{h}\Phi(\vxv, \vec{x}_i)}=\frac{\sum_{l\in\Gamma_{\V}, \vec{x}_l \in \left\{\vxv \pm \vec{\nu}_jh: j= 1, \cdots, M \right\}}\Phi\left(\vxv, \vec{x}_l\right)e^\frac{\phi_l-\phi_i}{h}}{\sum_{l\in\Gamma_{\V}, \vec{x}_l \in \left\{\vxv \pm \vec{\nu}_jh: j= 1, \cdots, M \right\}} \Phi(\vxv, \vec{x}_l)}, \quad   \text{ if }  k=i,\\
& 0, \qquad \qquad \qquad \qquad  \qquad \qquad \qquad \qquad \qquad \qquad \qquad \qquad \, \,\,\text{ otherwise},
\end{aligned}
\right.
\end{equation}
that is, $s(\vxv, \vec{x}_k)=\frac{1}{h}v\left(\vxv, \vec{x}_k\right) \Phi(\vxv, \vec{x}_k)$. Denote the new function by $\tilde{L}_{\V}:\Omegah \times \mathbb{R}^{|\Gamma_{\V}|}\to \mathbb{R}$ and we have
\begin{equation}
\begin{aligned}
    \tilde{L}_{\V}(\vxv, \vec{v}) =& \Lh(\vxv, \vec{s}) \\
    = &\sum_{k\in\Gamma_{\V}, \, \vec{x}_k \in \left\{\vxv \pm \vec{\nu}_jh: j= 1, \cdots, M \right\}} \Phi(\vxv, \vec{x}_k)\left(v\left(\vxv, \vec{x}_k\right)\log\left(v\left(\vxv, \vec{x}_k\right)\right) -v\left(\vxv, \vec{x}_k\right) +1 \right)\\
    =& \sum_{j=1,\, \vxv+ \vec{\nu}_{j} h\in \Omegah }^M \Phi^+_j(\vxv) \left(v\left(\vxv, \vxv+\vec{\nu}_jh\right) \log\left(v\left(\vxv,  \vxv+\vec{\nu}_jh\right)\right) -v\left(\vxv,  \vxv+\vec{\nu}_jh\right)+1\right)  \\
    &+\sum_{j=1,\, \vxv- \vec{\nu}_{j} h\in \Omegah }^M \Phi^+_j(\vxv) \left(v\left(\vxv, \vxv-\vec{\nu}_jh\right) \log\left(v\left(\vxv,  \vxv-\vec{\nu}_jh\right)\right) -v\left(\vxv,  \vxv-\vec{\nu}_jh\right)+1\right)\\
    \geq &0.
\end{aligned}
\end{equation}

\begin{prop}\label{prop:variformuh}
    Let $\uh$ be a classical solution to \eqref{upwind0} with initial data $\uh(\vec{x}_l, 0)=f(\vec{x}_l)$ for any $\vec{x}_l \in \Omegah$ with $l \in \Gamma_{\V}$. Then $\uh(\vec{x}_l, t)$ has the following representation formula
    \begin{equation}\label{oc}
    \begin{aligned}
   \uh\left(\vec{x}_l, t\right) = \max_{\vec{v},\vec{p} } & \left\{\sum_{i \in \Gamma_{\V}}f(\vxv) p_t(\vxv) - \int_0^t \sum_{i\in \Gamma_{\V}} \tilde{L}_{\V}(\vxv, \vec{v}_{t-s}) p_{t-s}(\vxv)  \ud s\right\},\\
    \text{s.t.} \,\,   \dot{p}_s(\vxv) =& \sum_{k\in \Gamma_{\V}} \left(\frac{1}{h}v_{s}(\vec{x}_k, \vxv) \Phi \left(\vec{x}_k,\vxv\right) p_s(\vec{x}_k) - \frac{1}{h} v_{s}(\vxv, \vec{x}_k) \Phi \left(\vxv,\vec{x}_k\right) p_s(\vec{x}_i)\right), \vxv \in \Omegah,\\
   & p_{s=0}(\vxv) = \delta_{il}, \quad  \sum_{\vxv\in \Omegah^c} p(\vxv,t)\equiv 0,
   \\ &v_s(\vxv, \vec{x}_k) \geq 0 \text{ for } k\neq i, \quad \sum_{k\in \Gamma_{\V}}v_s(\vxv, \vec{x}_k) \Phi(\vxv, \vec{x}_k)=0 \text{ for any } s \in [0, t].
    \end{aligned}
\end{equation}
Here $\Phi(\cdot,\cdot)$ is defined in \eqref{Phi_ik}.
\end{prop}
\begin{rem}
    By the definitions of $\vec{v}_s$ and $\Phi$, we write the forward equation in detail
    \begin{equation}
    \begin{aligned}
     \dot{p}_s(\vxv)=&\sum_{k\in \Gamma_{\V}, \, \vec{x}_k \in \left\{\vxv \pm \vec{\nu}_jh: j= 1, \cdots, M \right\} }\left(\frac{1}{h}v_{s}(\vec{x}_k, \vxv) \Phi \left(\vec{x}_k,\vxv\right) p_s(\vec{x}_k) - \frac{1}{h}v_{s}(\vxv, \vec{x}_k) \Phi \left(\vxv,\vec{x}_k\right) p_s(\vec{x}_i)\right)\\
     =&\sum_{j=1}^M \left(\frac{1}{h}v_{s}(\vxv+\vec{\nu}_jh, \vxv) \Phi^+_i\left(\vxv+\vec{\nu}_jh\right) p_s(\vxv+\vec{\nu}_jh) - \frac{1}{h}v_s(\vxv, \vxv+\vec{\nu}_jh)\Phi^+_j\left(\vxv\right)p_s(\vxv)\right)\\
   &+ \sum_{j=1}^{M} \left(\frac{1}{h}v_{s}(\vxv-\vec{\nu}_jh, \vxv) \Phi^-_i\left(\vxv-\vec{\nu}_jh\right) p_s(\vxv-\vec{\nu}_jh) - \frac{1}{h}v_s(\vxv, \vxv-\vec{\nu}_jh)\Phi^-_j\left(\vxv\right)p_s(\vxv)\right).
    \end{aligned}
    \end{equation}
\end{rem}
\begin{proof}
Let $\gamma:\Omegah \times [0, \infty) \to \mathbb{R}$ be defined by the representation formula given on the right hand side of \eqref{oc}. We will first show $\uh \geq \gamma$ and then argue the quality is actually achieved.

Let $\vxv \in \Omegah$ and $\tau \in [0, t]$. For any control velocity $\vec{v}_\tau$ and probability density $p_\tau$ that satisfy the conditions in \eqref{oc}, define a vector $\vec{s}_\tau(\vxv) \in \mathbb{R}^{|\Gamma_{\V}|}$ where the $k$th component of $\vec{s}_\tau(\vxv)$ is $s_\tau(\vxv, \vec{x}_k):=\frac{1}{h}v_\tau\left(\vxv, \vec{x}_k\right) \Phi(\vxv, \vec{x}_k)$. Note that $\sum_{k\in \Gamma_{\V}} s_\tau( \vxv, \vec{x}_k)=0$ and $s_\tau (\vxv, \vec{x}_k) \geq 0$ for $k\neq i$ from \eqref{oc} and \eqref{Phi_ik}. Hence, from the definition of $\Lh$, we have
\begin{equation}\label{eqn:runningcost}
\begin{aligned}
&\int_0^t \sum_{i\in \Gamma_{\V}} \tilde{L}_{\V}(\vxv, \vec{v}_{t-s}) p_{t-s}(\vxv)  \ud s
\\ =&\int_0^t \sum_{i\in \Gamma_{\V}} \Lh(\vxv, \vec{s}_{t-\tau}) p_{t-\tau}(\vxv)  \ud \tau\\
\geq &\int_0^t \sum_{i\in \Gamma_{\V}} \left( \vec{s}_{t-\tau} \cdot \vec{\phi} - \Hh \left(\vxv, \vec{\phi}\right)\right)p_{t-\tau}(\vxv)  \ud \tau\\
=& \int_0^t \sum_{i\in \Gamma_{\V}}\left( \sum_{k \in \Gamma_{\V}}\frac{1}{h} v_{t-s}(\vxv, \vec{x}_k) \Phi(\vxv, \vec{x}_k) \phi_k - \Hh \left(\vxv, \vec{\phi}\right)\right)p_{t-s}(\vxv)  \ud s
\end{aligned}
\end{equation}
for any $\vec{\phi} \in \mathbb{R}^{|\Gamma_{\V}|}$.

On the other end,
\begin{equation}
\begin{aligned}
 &\sum_{i\in \Gamma_{\V}} \uh(\vxv, t) p_0(\vxv) -\sum_{i \in \Gamma_{\V}} \uh(\vxv, 0)p_t(\vxv) \\
 =& \sum_{i \in \Gamma_{\V}} \int_0^t \frac{d}{ds} \left(\uh(\vxv, s)p_{t-s}\left(\vxv\right)\right)ds\\
= &\sum_{i \in \Gamma_{\V}}\int_0^t \partial_t \uh(\vxv, s)p_{t-s} \left(\vxv\right) - \uh(\vxv, s) \dot{p}_{t-s}(\vxv) ds\\
=&  \int_0^t \sum_{i \in \Gamma_{\V}}\sum_{k\in \Gamma_{\V}} \Phi\left(\vxv, \vec{x}_k\right)\left(e^\frac{\uh\left(\vec{x}_k, s\right)-\uh\left(\vxv, s\right)}{h}-1\right)p_{t-s} \left(\vxv\right)\\
&-\sum_{i \in \Gamma_{\V}}\sum_{k\in \Gamma_{\V}} \frac{1}{h} \uh(\vxv,s)\left(v_{t-s}(\vec{x}_k, \vxv) \Phi \left(\vec{x}_k,\vxv\right) p_{t-s}(\vec{x}_k) - v_{t-s}(\vxv, \vec{x}_k) \Phi \left(\vxv,\vec{x}_k\right) p_{t-s}(\vec{x}_i)\right) ds\\
=&\int_0^t \sum_{i \in \Gamma_{\V}}\sum_{k\in \Gamma_{\V}} \Phi\left(\vxv, \vec{x}_k\right)\left(e^\frac{\uh\left(\vec{x}_k, s\right)-\uh\left(\vxv, s\right)}{h}-1\right)p_{t-s} \left(\vxv\right)\\
&-\sum_{i \in \Gamma_{\V}}\sum_{k\in \Gamma_{\V}} \frac{1}{h} \uh(\vxv,s)v_{t-s}(\vec{x}_k, \vxv) \Phi \left(\vec{x}_k,\vxv\right) p_{t-s}(\vec{x}_k) - \sum_{i \in \Gamma_{\V}}\sum_{k\in \Gamma_{\V}} \frac{1}{h} \uh(\vxv,s) v_{t-s}(\vxv, \vec{x}_k) \Phi \left(\vxv,\vec{x}_k\right) p_{t-s}(\vec{x}_i) ds\\
=&\int_0^t \sum_{i \in \Gamma_{\V}}\sum_{k\in \Gamma_{\V}} \Phi\left(\vxv, \vec{x}_k\right)\left(e^\frac{\uh\left(\vec{x}_k, s\right)-\uh\left(\vxv, s\right)}{h}-1\right)p_{t-s} \left(\vxv\right)\\
&-\sum_{k \in \Gamma_{\V}}\sum_{i\in \Gamma_{\V}} \frac{1}{h} \uh(\vec{x}_k,s)v_{t-s}(\vxv, \vec{x}_k) \Phi \left(\vxv,\vec{x}_k\right) p_{t-s}(\vxv) - \sum_{i \in \Gamma_{\V}}\sum_{k\in \Gamma_{\V}}\frac{1}{h}  \uh(\vxv,s) v_{t-s}(\vxv, \vec{x}_k) \Phi \left(\vxv,\vec{x}_k\right) p_{t-s}(\vec{x}_i) ds\\
\end{aligned}
\end{equation}
where the last equality comes from exchanging the index name $i$ and $k$ for the second term. Moreover, since $\sum_{k\in \Gamma_{\V}} v_{t-s}(\vxv, \vec{x}_k) \Phi \left(\vxv,\vec{x}_k\right) =0$, we have
\begin{equation}\label{eqn:uhd}
    \begin{aligned}
&\sum_{i\in \Gamma_{\V}} \uh(\vxv, t) p_0(\vxv) -\sum_{i \in \Gamma_{\V}} \uh(\vxv, 0)p_t(\vxv)\\
=&\int_0^t \sum_{i \in \Gamma_{\V}}\sum_{k\in \Gamma_{\V}} \Phi\left(\vxv, \vec{x}_k\right)\left(e^\frac{\uh\left(\vec{x}_k, s\right)-\uh\left(\vxv, s\right)}{h}-1\right)p_{t-s} \left(\vxv\right)-\sum_{k \in \Gamma_{\V}}\sum_{i\in \Gamma_{\V}} \frac{1}{h}\uh(\vec{x}_k,s)v_{t-s}(\vxv, \vec{x}_k) \Phi \left(\vxv,\vec{x}_k\right) p_{t-s}(\vxv)ds,\\
=&\int_0^t \sum_{i \in \Gamma_{\V}}  \left( \sum_{k\in \Gamma_{\V}}\Phi\left(\vxv, \vec{x}_k\right)\left(e^\frac{\uh\left(\vec{x}_k, s\right)-\uh\left(\vxv, s\right)}{h}-1\right)-\sum_{k \in \Gamma_{\V}} \frac{1}{h}v_{t-s}(\vxv, \vec{x}_k) \Phi \left(\vxv,\vec{x}_k\right) \uh(\vec{x}_k,s)\right)p_{t-s}(\vxv)ds\\
\geq & - \int_0^t \sum_{i\in \Gamma_{\V}} \tilde{L}_{\V}(\vxv, \vec{v}_{t-s}) p_{t-s}(\vxv)  \ud s,
    \end{aligned}
\end{equation}
where the last inequality comes from \eqref{eqn:runningcost}. Rearranging the inequality and we obtain
\[
\uh(\vec{x}_l, t) \geq \sum_{i \in \Gamma_{\V}}f(\vxv) p_t(\vxv)- \int_0^t \sum_{i\in \Gamma_{\V}} \tilde{L}_{\V}(\vxv, \vec{v}_{t-s}) p_{t-s}(\vxv)  \ud s,
\]
which implies
\[
\uh(\vec{x}_l, t) \geq \gamma(\vec{x}_l, t).
\]

Moreover, by \eqref{eqn:v_tphi}, the equality in the last inequality of \eqref{eqn:uhd} is achieved when
\begin{equation}
    v_{t-s}(\vxv, \vec{x}_k) = \left\{\begin{aligned}
& e^\frac{\uh(\vec{x}_k,s)-\uh(\vxv, s)}{h}, \quad \qquad \qquad \qquad \qquad \qquad \qquad \qquad \qquad \,  \text{ if }\vec{x}_{k} \in \left\{\vxv \pm \vec{\nu}_jh: j= 1, \cdots, M \right\}, \\
&\frac{\sum_{l\in\Gamma_{\V}, \vec{x}_m \in \left\{\vxv \pm \vec{\nu}_jh: j= 1, \cdots, M \right\}}\Phi\left(\vxv, \vec{x}_m\right)e^\frac{\uh(\vec{x}_m,s)-\uh(\vxv, s)}{h}}{\sum_{l\in\Gamma_{\V}, \vec{x}_m \in \left\{\vxv \pm \vec{\nu}_jh: j= 1, \cdots, M \right\}} \Phi(\vxv, \vec{x}_m)}, \,\,\,\, \text{ if }  k=i,\\
& 0, \qquad \qquad \qquad \qquad  \qquad \qquad \qquad \qquad \qquad \qquad \qquad \quad \, \,\text{ otherwise}.
\end{aligned}
\right.
\end{equation}
Therefore, $\uh(\vec{x}_l, t) = \gamma(\vec{x}_l, t)$.
\end{proof}

\section{Convergence results for a single chemical reaction}\label{sec:restrict1d}
In this section, we examine the case of two species and a single chemical reaction, i.e., $N=2$ and $M=1$. In general, the same results also hold for $N \geq 2$ and $M=1$. Remarkably, after a suitable reparametrization, we show that the discrete state-constrained solution to the WKB reformulation of the backward equation converges to the solution of a continuous Hamilton--Jacobi equation with a Neumann boundary condition.

Consider the following chemical reaction with $N=2$ species $X_1$ and $X_2$, that is,
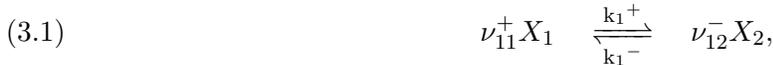
\begin{equation}
\nu^+_{11}X_{1} \quad \ce{<=>[k_1^+][k_1^-]} \quad \nu^{-}_{12}  X_2,
\end{equation}
where $k_1^\pm\geq 0$ are the reaction rates for the forward and backward reactions. In this case, $\vec{\nu}= (-\nu^+_{11}, \nu^-_{12})$. Let $\Omega \subset \mathbb{R}_+^2$ be an open bounded domain with Lipschitz boundary and $\overline{\Omega} \subset \mathbb{R}_+^2$. The chemical reaction process is restricted in $\overline{\Omega}$. Let $\vx_0 \in \overline{\Omega}$, which will serve as the limiting starting position for the rescaled reaction process.

We consider the intersection of $\overline{\Omega}$ with the line passing through $\vx_0$ in the direction of $\vec{\nu}$ and denote the maximal closed line segment in the direction $\vec{\nu}$ that contains $\vx_0$ by
\begin{equation}\label{eqn:defs0}
S_0:=  \left\{\vec{x}_0 + \alpha \vec{\nu}: a\leq \alpha \leq b\right\} \subset \left\{\vec{x}_0 + \alpha \vec{\nu}: \alpha \in \mathbb{R} \right\}\cap \overline{\Omega},
\end{equation}
for some constants $a, b \in \mathbb{R}$. Equivalently, we choose the largest interval $[a, b] \subset\mathbb{R}$ for which
\begin{equation}
   \left\{\vec{x}_0 + \alpha \vec{\nu}: a\leq \alpha \leq b\right\} \subset\overline{\Omega}.
\end{equation}
In particular, the endpoints $\vx_0+a\vec{\nu}$ and $\vx_0+b\vec{\nu}$ lie on the boundary $\partial \Omega$. Note that if the domain $\Omega$ is not convex, it is possible that $S_0 \subsetneq \left\{\vec{x}_0 + \alpha \vec{\nu}: \alpha \in \mathbb{R} \right\}\cap \overline{\Omega}$.

Recall from \eqref{eqn:defomegah} that the rescaled chemical reaction process $\cv(t)$ starts at
\[\vxzh \in \Omegah = \left\{ \vxi = \vec{i} h : \,\,  \vec{i} \in \mathbb{N}^N,\, \vxi \in \overline{\Omega}\right\}
\] and $\vxzh$ may or may not lie in $S_0$. We first consider the case $\vxzh \in S_0$, focusing on the rescaled chemical reaction process $\cv(t)$ whose state space is restricted to the line segment $S_0$. A more general setting, allowing arbitrary initial positions in $\overline{\Omega}$, will be addressed later in Section \ref{sec:ubetaconvergence}. Specifically, assume $\cv(0)=\vxzh=\vx_0+(r+kh)\vec{\nu} \in S_0$ for some $r \in \mathbb{R}$ small and $k\in \mathbb{Z}$. Then the process evolves within the discrete set
\begin{equation} \label{eqn:ohtilde}
 \horh:=\left\{\vec{x}_0 +(r+\alpha) \vec{\nu} \in S_0: \alpha =kh, k  \in \mathbb{Z}\right\}.
\end{equation}
By introducing a suitable reparametrization, we can reformulate the problem as a one-dimensional system, as detailed below.

Let $\vxv \in \horh$. By the previous discussion, the discrete backward Hamilton–Jacobi equation obtained via the WKB reformulation is given by \eqref{eqn:1duh}, namely,
\[
\begin{aligned}
\pt_t \uh(\vxv,t) = & \mathbbm{1}_{\left\{\vxv +h\vec{\nu} \,\in \,\horh\right\}}  \Phi^+(\vxv) \left( e^{\frac{\uh\left(\vxv+ \vec{\nu} h,t\right)-\uh\left(\vxv,t\right)}{h}} - 1\right)\\ &+ \mathbbm{1}_{\left\{\vxv -h\vec{\nu} \, \in \,\horh \right\}}\Phi^-(\vxv)\left( e^{  \frac{\uh\left(\vxv- \vec{\nu}h,t\right)-\uh\left(\vxv,t\right)}{h} } - 1\right),\\
\uh(\vxv, 0)=&u_0(\vxv),
\end{aligned}
\]
where $u_0$ is a given bounded and continuous function.

From \eqref{eqn:ohtilde}, $\vxv = \vec{x}_0 +(r+\alpha) \vec{\nu} $ for some $\alpha =kh $ with $k \in \mathbb{Z}$. Since $\horh \subset S_0$,
\[
\horh=\left\{\vec{x}_0 +(r+\alpha) \vec{\nu}: \alpha =kh, k  \in \mathbb{Z}, k_{\sa, \sr, \V} \leq k \leq k_{\scb, \sr, \V}\right\}
\]
for some integers $k_{\sa, \sr, \V} \leq k_{\scb, \sr, \V}$ that depend on $a, b, r$, $\vec{x}_0$ and $h$, and we only indicate the dependence on $a, b, r, h$ for simplicity. Note that for any fixed $r\in \mathbb{R}$, we have
\begin{equation}\label{eqn:endpts}
\lim_{h\to 0^+} k_{\sa, \sr, \V}h=a, \qquad \, \lim_{h\to 0^+} k_{\scb, \sr, \V}h=b .
\end{equation}
Denote
\[
\torh :=\{r+\aldh: \aldh=kh, k\in \mathbb{Z},  k_{\sa, \sr, \V} \leq k \leq k_{\scb, \sr, \V}\}
\]
and define $\wrh:\torh \times [0,\infty) \to \mathbb{R}$ by
\[\wrh \left(r+\aldh, t\right): = \uh(\vec{x}_0+(r+\aldh) \vec{\nu} , t).\]
Moreover, we define $\tilde{\Phi}^+(r+\aldh):= \Phi^+(\vec{x}_0+(r+\aldh) \vec{\nu})$, $\tilde{\Phi}^-(r+\aldh):= \Phi^-(\vec{x}_0+(r+\aldh) \vec{\nu})$, $w_0(r+\aldh):=u_0(\vec{x}_0 +(r+\aldh) \vec{\nu})$ for any $r+\aldh \in \torh$.
Then from \eqref{eqn:1duh}, $\wrh$ solves
\begin{equation}\label{eqn:1dv_h}
    \left\{\begin{aligned}
\pt_t \wrh(r+\aldh,t) = & \mathbbm{1}_{\{\aldh
\leq (k_{b,r, h}-1)h\}}  \tilde{\Phi}^+(r+\aldh) \left( e^{\frac{\wrh(r+\aldh+ h, t)-\wrh(r+\aldh, t)}{h}} - 1\right)\\ &+ \mathbbm{1}_{\{\aldh
\geq (k_{a,r, h}+1)h\}}\tilde{\Phi}^-(r+\aldh)\left( e^{  \frac{\wrh(r+\aldh- h, t)-\wrh(r+\aldh, t)}{h} } - 1\right), \\  &\qquad\qquad\qquad\qquad\qquad\qquad\qquad\text{ for }  (r+\aldh, t) \in \torh \times (0, \infty),\\
\wrh(r+\aldh, 0)=&w_0(r+\aldh),
    \end{aligned}
    \right.
\end{equation}
with the discrete Hamiltonian
\begin{equation}
\begin{aligned}
\tHrh: \ell^{\8}(\torh) \to \ell^{\8}(\torh); \qquad  ( w(r+\aldh) )_{r+\aldh \in\torh} \mapsto ( \tHrh(r+\aldh, w(r+\aldh), w(\cdot)) )_{r+\aldh \in\torh}
\end{aligned}
\end{equation}
defined by
\begin{equation}\label{eqn:defHtildeh}
\begin{aligned}
\tHrh\left(r+\aldh , w(r+\aldh),w\right) :=&\mathbbm{1}_{\{\aldh \leq (k_{\scb, \sr, \V}-1)h\}}  \tilde{\Phi}^+(r+\aldh) \left( e^{\frac{w(r+\aldh+ h)-w(r+\aldh)}{h}} - 1\right)\\ &+ \mathbbm{1}_{\{\aldh
\geq (k_{\sa, \sr, \V}+1)h\}}\tilde{\Phi}^-(r+\aldh)\left( e^{  \frac{w(r+\aldh- h)-w(r+\aldh)}{h} } - 1\right).
\end{aligned}
\end{equation}
We will show that $\wrh$ converges to $w$ as $r, h \to 0^+$, where $w$ solves the continuous Hamilton--Jacobi equation with a Neumann boundary condition given in \eqref{HJE_NM}:
\begin{equation}
    \left\{\begin{aligned}
     \pt_t w(\alpha, t) =&  \tilde{\Phi}^+(\alpha) \left( e^{\partial_\alpha w(\alpha,t)} - 1\right)+ \tilde{\Phi}^-(\alpha)\left( e^{-\partial_\alpha w(\alpha,t)} - 1\right) \quad \text{ in } (a ,b) \times (0, \infty),\\ -\partial_\alpha w(a, t) =& \partial_\alpha w(b, t)=0 \qquad \qquad \qquad\qquad\qquad \qquad \qquad \qquad \,\text{ for } t \in (0, \infty),\\w(\alpha, 0)=&w_0(\alpha) \qquad \qquad \qquad\qquad\qquad \qquad \qquad \qquad \qquad \quad \,\text{ on } [a, b],
    \end{aligned}
    \right.
\end{equation}
where $\tilde{\Phi}^+(\alpha)= \Phi^+(\vec{x}_0+ \alpha \vec{\nu})$, $\tilde{\Phi}^-(\alpha)= \Phi^-(\vec{x}_0+ \alpha\vec{\nu})$, $w_0(\alpha)=u_0(\vec{x}_0 +\alpha \vec{\nu})$ for $\alpha \in [a, b]$. This is a Cauchy problem with a Neumann boundary condition for $t >0$. Note that the Hamiltonian
\begin{equation}\label{eqn:Htilde}
\tilde{H}(\alpha, p):=\tilde{\Phi}^+(\alpha) \left( e^p - 1\right)+ \tilde{\Phi}^-(\alpha)\left( e^{-p} - 1\right)
\end{equation}
is strictly convex in $p$ and is coercive, i.e., $\lim_{|p|\to +\infty} \inf_{\alpha \in [a, b]}\tilde{H} (\alpha, p) =\infty$.

\subsection{Comparison Principle for continuous HJE with Neumann Boundary condition}
The comparison principle for Hamilton-Jacobi equations with Neumann Boundary condition is well established in the literature (see for instance  \cite{lions1985neumann, BARLES1991, DUPUIS19901123, Ishii2011}). We include it here for completeness.

\begin{prop}[{Restatement of \cite[Theorem 3.4]{Ishii2011}}] \label{prop:cpn} Let $T>0$. Suppose $v_1 \in {\rm USC} \left([a, b] \times [0, T)\right)$ and $v_2 \in {\rm LSC} \left([a, b] \times [0, T)\right)$, where ${\rm USC}$ and ${\rm LSC}$ denote the spaces of upper and lower semincontinuous functions, respectively. Assume that $v_1$ is a viscosity subsolution and $v_2$ is a viscosity supersolution of \eqref{HJE_NM} in $[a,b] \times [0, T)$. Then $v_1\leq v_2$ in $[a, b] \times [0, T)$.
\end{prop}

\begin{cor}\label{cor:1dcontcontract}
 Let $T>0$. Suppose $w \in C \left([a, b] \times [0, T)\right)$ is a viscosity solution to \eqref{HJE_NM} in $[a,b] \times [0, T)$ with initial value $w_0\in C([a, b])$, and $\tilde{w} \in C \left([a, b] \times [0, T)\right)$ is a viscosity solution to \eqref{HJE_NM} in $[a,b] \times [0, T)$ with initial value $\tilde{w}_0\in C([a, b])$. Then
 \[
 \left\|w-\tilde{w}\right\|_{L^\infty\left([a, b]\times[0,T)\right)} \leq \left\|w_0-\tilde{w}_0\right\|_{L^\infty\left([a,b]\right)}.
 \]
\end{cor}

\begin{proof}
Define $v: [a, b] \times [0, \infty) \to \mathbb{R}$ by
\[
v(x, t):= \tilde{w}(x, t)+  \left\|w_0-\tilde{w}_0\right\|_{L^\infty\left([a,b]\right)}.
\]
Then $v$ is a supersolution of \eqref{HJE_NM} with initial condition $w_0$. By Proposition \ref{prop:cpn}, for any $x \in [a, b]$ and $t \geq 0$,
\[
w(x ,t) \leq v(x,t),
\]
which implies
\[
w(x, t) -\tilde{w}(x, t)\leq \left\|w_0-\tilde{w}_0\right\|_{L^\infty\left([a,b]\right)}.
\]
The proof of the other inequality,
\[
\tilde{w}(x, t)-w(x, t) \leq \left\|w_0-\tilde{w}_0\right\|_{L^\infty\left([a,b]\right)}
\]
follows by a similar argument.
\end{proof}

\subsection{Existence and uniqueness of continuous HJE with Neumann Boundary condition}
The uniqueness of \eqref{HJE_NM} follows directly from the comparison principle. As for existence, a representation formula derived from an optimal control problem provides a solution to \eqref{HJE_NM}. This is also well known in the literature (see \cite{Ishii2011}).

Define $w :[a, b] \times [0, \infty) \to \mathbb{R}$ by
\begin{equation}\label{eqn:1dvf}
    w(\alpha, t) := \sup\left\{\left.w_0(\eta(t))-\int_0^t \tilde{L}\left(\eta(s), v(s) \right) ds \right|\left(\eta, v, l\right) \in {\rm SP}(\alpha)\right\}
\end{equation}
where $\tilde{L} :  [a, b] \times \mathbb{R}  \to \mathbb{R}$ is the Legendre transform of $ \tilde{H}$ in \eqref{eqn:Htilde}, and for $ \alpha \in [a, b]$, ${\rm SP}(\alpha)$ denotes the family of solutions to the Skorokhod problem: for given $\alpha \in [a, b]$ and $v\in L^1_{\rm loc} \left([0, \infty), \mathbb{R}\right)$, find a pair $(\eta, l)\in {\rm AC}_{\rm loc} \left([0, \infty), \mathbb{R}\right) \times L^1_{\rm loc} \left([0, \infty), \mathbb{R}\right)$ such that
\begin{equation} \label{eqn:1dpath}
    \left\{
    \begin{aligned}
     \eta(0) = &\alpha,\\
     \eta(s) \in & [a, b] \qquad\qquad\qquad \text{ for all } s\in [0, \infty),\\
     \dot{\eta}(s)+l(s) n(\eta(s))=&v(s) \qquad \qquad\qquad\, \text{ for a.e. } s \in (0, \infty),\\
     l(s) \geq& 0 \qquad \qquad \qquad \quad \text{ for a.e. } s \in (0, \infty),\\
     l(s)= &0 \text{ if } \eta(s) \in (a, b) \quad\text{ for a.e. } s \in (0, \infty),
    \end{aligned}
    \right.
\end{equation}
where $n:\{a, b\} \to \{-1, 1\}$ is defined by
$n(a)=-1$ and $n(b)=1$, which can be viewed as the outward normal vector of the domain $[a, b]$.
\begin{prop}[{Restatement of \cite[Theorem 5.1]{Ishii2011}}] \label{thm:rf}
    The value funcion $w$ defined by \eqref{eqn:1dvf} is continuous on $[a, b] \times [0, \infty)$ and it is a viscosity solution to \eqref{HJE_NM}. Moreover, $\lim_{t\to 0^+}w(\alpha, t)=w_0(\alpha)$ uniformly for $\alpha \in [a,b]$.
\end{prop}

\subsection{Convergence of $\wrh$ to $w$ as $h\to 0^+$}
Let $h>0 $ and $r \in \mathbb{R}$ be small. Suppose $\wrh : \torh  \times [0, \infty) \to \mathbb{R}$ is the solution to  \eqref{eqn:1dv_h} with initial data $\wrh(r+\aldh, 0)=w_0(r+\aldh)$ for any $r+\aldh \in \torh$. Define the upper semicontinuous (USC) envelope of $\wrh$ by
\begin{equation}\label{eqn:defubar}
\overline{w}(\alpha, t): = \limsup_{h \to 0^+, \, r \to 0, \, \alhi \to \alpha, \, r+\alhi \in \torh , \,\ti\to t} \wrh(r+\alhi, \ti )
\end{equation}
for $(\alpha, t) \in [a, b] \times [0, \infty)$, and the lower semicontinuous (LSC) envelope of $\wrh$ by
\begin{equation}\label{eqn:defuunder}
\underline{w}(\alpha, t): = \liminf_{h \to 0^+, \, r \to 0, \, \alhi \to \alpha, \, r+\alhi \in \torh , \,\ti\to t} \wrh(r+\alhi, \ti)
\end{equation}
for $(\alpha, t) \in [a, b] \times [0, \infty)$.

\begin{thm}\label{thm:wbarsubsln}
Let $u_0 \in  C(\overline{\Omega})$, and let $\wrh : \torh  \times [0,\infty) \to \mathbb{R}$ be the solution to \eqref{eqn:1dv_h}. Then, the USC envelope $\overline{w}:  [a, b] \times [0, \infty) \to \mathbb{R}$ is a subsolution to \eqref{HJE_NM}, and the LSC envelope $\underline{w}: [a, b] \times [0, \infty) \to \mathbb{R}$ is a supersolution to \eqref{HJE_NM}.
\end{thm}
\begin{proof}
We prove that the USC envelope $\overline{w}$ is a subsolution to \eqref{HJE_NM} in $[a, b] \times[0, \infty)$. The proof of the fact that $\underline{w}$ is a supersolution to \eqref{HJE_NM} in $[a, b]\times[0, \infty)$ is similar.

    Step 1: For $t=0$, we claim that $\overline{w}(\alpha, 0)=w_0(\alpha)$ for any $\alpha \in [a, b]$. We first prove the claim under the assumption that $u_0 \in W^{1,\infty} (\Omega)\cap C(\overline{\Omega})$, and later extend the result to the case $u_0 \in  C(\overline{\Omega})$.

Suppose $u_0 \in W^{1,\infty} (\Omega)\cap C(\overline{\Omega})$. Then $w_0 \in \mathrm{Lip}([a, b])$ with the Lipschitz constant $\left\|Du_0\right\|_{L^\infty(\Omega)}$. Fix $\alpha \in [a, b]$, and let $r, h >0$. By an argument similar to the proof of Proposition \eqref{prop:unibduhpt}, for any $r+\aldh \in \torh$ and $t\geq 0$,
\begin{equation}\label{eqn:ptbound}
\begin{aligned}
\left|\pt_t \wrh(r+\aldh,t)\right|&\leq \left|\tHrh\left(r+\aldh, w_0(\vxi),w_0\right)\right|\\
&\leq C_0
\end{aligned}
\end{equation}
where the second inequality follows from $w_0\in \mathrm{Lip}([a, b])$, with a constant
\[
C_0=C_0(\left\|Du_0\right\|_{L^\infty(\Omega)}, \Omega, \Phi^+, \Phi^-)>0
\]
independent of $h$. Hence, for any $r+\alhi \in \torh$ and $\ti\geq 0$,
\begin{equation}
\begin{aligned}\label{eqn:wlipt0}
    \left|\wrh(r+\alhi, \ti) -w_0(\alpha)\right|\leq &\left|\wrh(r+\alhi, \ti)  - w_0(r+\alhi)\right| \\
    &+\left|w_0(r+\alhi)-w_0(\alpha)\right|\\
    \leq &C_0|\ti|+\left\|Du_0\right\|_{L^\infty(\Omega)}|r+\alhi-\alpha|,
\end{aligned}
\end{equation}
where $C_0$ comes from \eqref{eqn:ptbound}. Letting $h \to 0^+, r \to 0, \alhi\to \alpha$, and $\ti \to 0$, we obtain
\[
\begin{aligned}
\overline{w}(\alpha,0)&= \limsup_{h \to 0^+, \, r \to 0, \, \alhi \to \alpha, \, r+\alhi \in \torh , \,\ti\to 0} \wrh(r+\alhi, \ti)\\
&= \lim_{h \to 0^+, \, r \to 0, \, \alhi \to \alpha, \, r+\alhi \in \torh , \,\ti\to 0} \wrh(r+\alhi, \ti)\\
&=w_0(\alpha).
\end{aligned}
\]

Now suppose $u_0 \in C(\overline{\Omega})$. Then there exists a sequence $\left\{\uzk\right\}_{\sk\in \mathbb{N}} \subset W^{1,\infty}\cap C(\overline{\Omega})$ such that $|\uzk(x)-u_0(x)|\leq \frac{1}{k}$ for all $x \in \overline{\Omega}$. For $\alpha \in [a, b]$, define $\wzk(\alpha):=\uzk(\vec{x}_0 +\alpha\vec{\nu})$. Let $\wrhk:\torh \to \mathbb{R}$ be the solution to \eqref{eqn:1dv_h} with initial data $\wzk\big|_{\torh}$, and $w_{\sk}:[a, b] \to \mathbb{R}$ be the solution to \eqref{HJE_NM} with initial data $\wzk$.

Since $\left|\uzk(x)-u_0(x)\right|\leq \frac{1}{k}$ for all $x \in \overline{\Omega}$, we have $|\wzk(r+\aldh)-w_0(r+\aldh)|\leq\frac{1}{k}$ for any $r+\aldh \in \torh$. Then, by Proposition \ref{prop:propdis}, it follows that for any $h>0$,
\begin{equation}\label{eqn:pdistbd}
    \left| \wrhk(r+\alhi, \ti)-\wrh(r+\alhi, \ti) \right| \leq \left\|\wzk -w_0\right\|_{L^\infty\left(\torh\right)} \leq \frac{1}{k}.
\end{equation}

Therefore, for any $r+\alhi \in \torh$ and $\ti\geq 0$,
\begin{equation}
\begin{aligned}
    &\left|\wrh(r+\alhi, \ti)-w_0(\alpha)\right| \\
    \leq & \left|\wrh(r+\alhi, \ti) - \wrhk (r+\alhi, \ti)\right|+\left|\wrhk(r+\alhi, \ti)-\wzk(\alpha)\right|+\left|\wzk(\alpha)-w_0(\alpha)\right|\\
    \leq & \frac{2}{k}+ C_0|\ti|+\left\|D\uzk\right\|_{L^\infty(\Omega)}|r+\alhi-\alpha|,
\end{aligned}
\end{equation}
where the last inequality uses \eqref{eqn:wlipt0} and \eqref{eqn:pdistbd}. Taking the limsup as $h \to 0^+, r \to 0, \alhi\to \alpha$, and $\ti \to 0$, we find
\[
\limsup_{h \to 0^+, \, r \to 0, \, \alhi \to \alpha, \, r+\alhi \in \torh , \,\ti\to 0}\left|\wrh(r+\alhi, \ti)-w_0(\alpha)\right| \leq\frac{2}{k}
\]
for any $k\in \mathbb{N}$. Hence, the claim follows.

    Step 2: For $t>0$, let $\varphi \in C^1\left([a, b] \times [0, \infty)\right)$ such that $\overline{w}-\varphi$ attains a strict local max at $(\alpha_0, t_0) \in [a, b] \times (0, \infty)$. We claim that there exists a sequence $\{h_{\sn}\}_{\sn=1}^{\infty}$ with $\lim_{n \to \infty} h_{\sn} = 0$ and a sequence $\{(r_{\sn}+\alpha_{\sn}, t_{\sn}) \}_{\sn=1}^{\infty} \subset \tornhn \times (0, \infty)$ with $\lim_{n \to \infty} r_{\sn} =0$, $\lim_{n \to \infty} \alpha_{\sn} = \alpha_0$ and $\lim_{n \to \infty} t_{\sn} = t_0$ such that
    \[
    \wrnhn\left(r_{\sn}+\alpha_{\sn}, t_{\sn}\right) \to \overline{w} (\alpha_0, t_0)
    \]
 and $(r_{\sn}+\alpha_{\sn},t_{\sn})$ is a local maximum of $\wrnhn-\varphi$. A proof is provided in \ref{sec:localmax} for the reader's convenience.
Hence, we have
\[
\partial_t \varphi \left(r_{\sn}+\alpha_{\sn}, t_{\sn}\right)=\partial_t \wrnhn \left(r_{\sn}+\alpha_{\sn}, t_{\sn}\right)
\]
and
\[
\wrnhn(r_{\sn}+\alpha_{\sn} \pm h_{\sn}, t_{\sn})-\varphi(r_{\sn}+\alpha_{\sn}\pm h_{\sn}, t_{\sn}) \leq \wrnhn(r_{\sn}+\alpha_{\sn}, t_{\sn}) -\varphi(r_{\sn}+\alpha_{\sn}, t_{\sn})
\]
if $r_{\sn}+\alpha_{\sn} \pm h_{\sn} \in \tornhn$. This implies
\begin{equation}\label{eqn:1dphi}
\begin{aligned}
&\partial_t \varphi \left(r_{\sn}+\alpha_{\sn}, t_{\sn}\right)-\mathbbm{1}_{\{\alpha_{\sn}
\leq (k_{\scb,{\sr}_{\sn}, {\V}_{\sn}}-1)h_{\sn}\}}  \tilde{\Phi}^+(r_{\sn}+\alpha_{\sn}) \left( e^{\frac{\varphi(r_{\sn}+\alpha_{\sn}+ h_{\sn}, t_{\sn})-\varphi(r_{\sn}+\alpha_{\sn}, t_{\sn})}{h_{\sn}}} - 1\right)\\ &\qquad  \qquad\qquad \qquad-\mathbbm{1}_{\{\alpha_{\sn}
\geq (k_{\sa, {\sr}_{\sn}, {\V}_{\sn}}+1)h_{\sn}\}}\tilde{\Phi}^-(r_{\sn}+\alpha_{\sn})\left( e^{  \frac{\varphi(r_{\sn}+\alpha_{\sn}- h_{\sn}, t_{\sn})-\varphi(r_{\sn}+\alpha_{\sn}, t_{\sn})}{h_{\sn}} } - 1\right)\\
\leq&\partial_t \wrnhn \left(r_{\sn}+\alpha_{\sn}, t_{\sn}\right)-\mathbbm{1}_{\{\alpha_{\sn}
\leq (k_{\scb, {\sr}_{\sn}, {\V}_{\sn}}-1)h_{\sn}\}}  \tilde{\Phi}^+(r_{\sn}+\alpha_{\sn}) \left( e^{\frac{\wrnhn(r_{\sn}+\alpha_{\sn}+ h_{\sn}, t_{\sn})-\wrnhn(r_{\sn}+\alpha_{\sn}, t_{\sn})}{h_{\sn}}} - 1\right)\\ &\qquad \qquad \qquad\qquad\quad -\mathbbm{1}_{\{\alpha_{\sn}
\geq (k_{\sa, {\sr}_{\sn},{\V}_{\sn}}+1)h_{\sn}\}}\tilde{\Phi}^-(r_{\sn}+\alpha_{\sn})\left( e^{  \frac{\wrnhn(r_{\sn}+\alpha_{\sn}- h_{\sn}, t_{\sn})-\wrnhn(r_{\sn}+\alpha_{\sn}, t_{\sn})}{h_{\sn}} } - 1\right)\\
=&0.
\end{aligned}
\end{equation}

There are two cases to be considered.
\begin{enumerate}
    \item Suppose $\alpha_0 \in (a, b)$. Then it follows from \eqref{eqn:endpts} that for $n$ large enough, we have  $ (k_{\sa, \sr_{\sn}, \V_{\sn}} +1)h_{\sn}\leq \alpha_{\sn} \leq (k_{\scb, \sr_{\sn}, \V_{\sn}}-1)h_{\sn}$, and \eqref{eqn:1dphi} becomes
\[
\begin{aligned}
\partial_t \varphi \left(r_{\sn}+\alpha_{\sn}, t_{\sn}\right)- &\tilde{\Phi}^+(r_{\sn}+\alpha_{\sn}) \left( e^{\frac{\varphi(r_{\sn}+\alpha_{\sn}+ h_{\sn}, t_{\sn})
-\varphi(r_{\sn}+\alpha_{\sn}, t_{\sn})}{h_{\sn}}} - 1\right)\\
& \qquad \qquad -\tilde{\Phi}^-(r_{\sn}+\alpha_{\sn})\left( e^{  \frac{\varphi(r_{\sn}+\alpha_{\sn}- h_{\sn}, t_{\sn})-\varphi(r_{\sn}+\alpha_{\sn}, t_{\sn})}{h_{\sn}} } - 1\right) \leq 0.
\end{aligned}
\]
    Therefore, by sending $n \to \infty$, we obtain

\[
\partial_t \varphi \left(\alpha_0, t_0\right) - \tilde{\Phi}^+(\alpha_0)\left(e^{ \partial_\alpha\varphi(\alpha_0,t_0)}- 1 \right)- \tilde{\Phi}^-(\alpha_0)\left(e^{- \partial_\alpha\varphi(\alpha_0,t_0)}- 1 \right) \leq 0,
\]
which shows that $\overline{w}$ is a subsolution to \eqref{HJE_NM} in $(a, b) \times (0, \infty)$.
\item Suppose $\alpha_0 = a$ or $\alpha_0 = b$. Without loss of generality, we can assume $\alpha_0=a$ and the proof for the case where $\alpha_0=b$ is similar. By the definition of viscosity subsolution to \eqref{HJE_NM}, we need to show
\begin{equation}\label{eqn:1dsubsln}
\min \left\{\partial_t \varphi \left(a, t_0\right) - \tilde{\Phi}^+(a)\left(e^{ \partial_\alpha\varphi(a,t_0)}- 1 \right)- \tilde{\Phi}^-(a)\left(e^{- \partial_\alpha\varphi(a,t_0)}- 1 \right), -\partial_\alpha \varphi(a, t_0)\right\} \leq 0.
\end{equation}
Taking a subsequence if needed, there are two subcases to be considered.
\begin{enumerate}
    \item If  $(k_{\sa, \sr_{\sn}, \V_{\sn}}+1)h_{\sn} \leq \alpha_{\sn} \leq (k_{\scb, \sr_{\sn}, \V_{\sn}}-1)h_{\sn}$ for $n$ large enough, then like in Case (1), we have
\[
\partial_t \varphi \left(a, t_0\right) - \tilde{\Phi}^+(a)\left(e^{ \partial_\alpha\varphi(a,t_0)}- 1 \right)- \tilde{\Phi}^-(a)\left(e^{- \partial_\alpha\varphi(a,t_0)}- 1 \right) \leq 0,
\]
which implies \eqref{eqn:1dsubsln} holds.
     \item If $\alpha_{\sn} = k_{\sa, \sr_{\sn}, \V_{\sn}} h_{\sn}$ for $n$ large enough, then \eqref{eqn:1dphi} becomes
\[
\partial_t \varphi \left(r_{\sn}+\alpha_{\sn}, t_{\sn}\right)- \tilde{\Phi}^+(r_{\sn}+\alpha_{\sn}) \left( e^{\frac{\varphi(r_{\sn}+\alpha_{\sn}+ h_{\sn}, t_{\sn})-\varphi(r_{\sn}+\alpha_{\sn}, t_{\sn})}{h_{\sn}}} - 1\right) \leq 0.
\]
Send $n \to \infty$ to obtain
\begin{equation} \label{eqn:1d1side}
    \begin{aligned}
    \partial_t \varphi \left(a, t_0\right) - \tilde{\Phi}^+(a)\left(e^{ \partial_\alpha\varphi(a,t_0)}- 1 \right) \leq 0.
    \end{aligned}
\end{equation}
If $\partial_t \varphi \left(a, t_0\right) - \tilde{\Phi}^+(a)\left(e^{ \partial_\alpha\varphi(a,t_0)}- 1 \right)- \tilde{\Phi}^-(a)\left(e^{- \partial_\alpha\varphi(a,t_0)}- 1 \right) \leq 0$, then we are done. Suppose not, that is, we have
\begin{equation}
\partial_t \varphi \left(a, t_0\right) - \tilde{\Phi}^+(a)\left(e^{ \partial_\alpha\varphi(a,t_0)}- 1 \right)- \tilde{\Phi}^-(a)\left(e^{- \partial_\alpha\varphi(a,t_0)}- 1 \right) > 0.
\end{equation}
Then by \eqref{eqn:1d1side}, we have $- \tilde{\Phi}^-(a)\left(e^{- \partial_\alpha\varphi(a,t_0)}- 1 \right) > 0$, which implies $-\partial_\alpha \varphi(a, t_0)<0$ and hence \eqref{eqn:1dsubsln} holds.
\end{enumerate}
\end{enumerate}

\end{proof}

By applying the comparison principle and the representation formula for solutions to \eqref{HJE_NM}, as stated in Proposition \ref{prop:cpn} and Proposition \ref{thm:rf}, we establish the convergence of the discrete HJE solutions to the continuous HJE solution. This result is summarized in the following corollary.
\begin{cor}\label{cor:1dconverge}
Let $\overline{w}$ be as defined in \eqref{eqn:defubar} and $\underline{w}$ as defined in \eqref{eqn:defuunder}. Let $w$ be the solution to \eqref{HJE_NM}. Then
\[
\overline{w}=\underline{w}=w.
\]
\end{cor}

\bigskip

\section{Large deviation principle and macroscopic mean-field limit}\label{sec:1dLDP}

In this section, we continue with the same setting as in the previous one, namely $N = 2$ and $M = 1$. We establish a large deviation principle at a single time, assuming the rescaled reaction processes start at positions restricted to the line segment $S_0$. The same result remains valid for general $N \geq 2$ with $M = 1$.

From \eqref{eqn:1dvf}, \eqref{eqn:1dpath} and Proposition \ref{thm:rf}, the viscosity solution $w:[a,b] \times [0, \infty) \to \mathbb{R}$ to \eqref{HJE_NM} can be represented by
\begin{equation}
w(\alpha, t) = \sup_{y\in[a, b]}\left\{w_0(y)-I\left(y;\alpha,t\right) \right\},
\end{equation}
where
\begin{equation}\label{eqn:defIratef}
    \quad I\left(y;\alpha,t\right):=\inf \left\{\left.\int_0^t  \tilde{L}\left(\eta(s), v(s) \right)ds\right| \eta(t)=y, (\eta, v, l) \in {\rm SP}(\alpha) \right\},
\end{equation}
and $\eta, v, l$ solves \eqref{eqn:1dpath}. Recall here $\tilde{L}$ is the Legendre transform of $ \tilde{H}$ in \eqref{eqn:Htilde}.

\begin{thm}\label{thm:1dLDP}
Let $\vx_0 \in \overline{\Omega}$ and $h>0$. Suppose $\vxzh = \vx_0+(\rh+\alzh)\vec{\nu}
\in \overline{\Omega}$ where $\rh \in \mathbb{R}$ satisfies $|\rh| < h$, $\alzh = \kzh h$ for some $\kzh \in \mathbb{Z}$, and
\[
\lim_{h \to 0^+}\rh=0, \quad \lim_{h\to 0^+}\alzh = 0.
\] Then the chemical reaction process $\cv(t)$ with $\cv(0)=\vxzh$ at each time $t$ satisfies the large deviation principle with the good rate function $\tilde{I}\left(\vec{y}; \vx_0, t\right):= I\left(\frac{(\vec{y}-\vx_0)\cdot \vec{\nu}}{|\vec{\nu}|^2};0,t\right)$, where $I$ is defined in \eqref{eqn:defIratef}.
\end{thm}

\begin{proof}
    For $h>0$ small enough, the state space for $\cv(t)$ is
    \[
    \hat{\Omega}_{\sr_{\V}, \V}=\left\{\vec{x}_0 +(\rh+\alpha) \vec{\nu}: \alpha =kh, k  \in \mathbb{Z}, k_{\sa, {\sr}_{\V}, \V} \leq k \leq k_{\scb, {\sr}_{\V}, \V}\right\} \subset S_0,
    \]
    where $k_{\sa, {\sr}_{\V}, \V}$ and $k_{\scb, {\sr}_{\V}, \V}$ are some integers and $S_0 \subset \overline{\Omega}$ is the maximal closed line segment in the direction $\vec{\nu}$ that contains $\vx_0$ defined in \eqref{eqn:defs0}. Hence $\cv(t)=\vx_0+(\rh+\aluh(t))\vec{\nu}$ for some $\aluh(t) = k^{\V}(t)h$ where $k^{\V}(t) \in \mathbb{Z}$ and $\aluh(0)=\alzh=\kzh h$.

We now prove the large deviation principle for process $\cv(t)$ at a single time through Varadhan's inverse lemma \cite{Bryc_1990}. Note that the exponential tightness of $\cv(t)$ for each $t\in[0,T]$ is given in \eqref{eq:exp_tight}. Then we only need to verify the convergence of
the discrete Varadhan's nonlinear semigroup to the variational representation of the continuous viscosity solution.


From Varadhan's lemma \cite{Varadhan_1966} on the necesssary condition of the large deviation principle and Brysc's theorem \cite{Bryc_1990} on the sufficient condition, we know $\cv$ satisfies the large deviation principle with the good rate function $\tilde{I}(y)$ if and only if for any bounded continuous function $u_0$,

\begin{equation}\label{eqn:LDP1dw0}
\lim_{h \to 0^+}h \log \mathbb{E}^{\vxzh}\left(e^\frac{u_0\left(\cv(t)\right)}{h}\right) =\sup_{\vec{y}\in S_0}\left\{u_0(\vec{y})-\tilde{I}\left(\vec{y}; \vx_0, t\right) \right\}.
\end{equation}
Hence, it suffices to prove \eqref{eqn:LDP1dw0}.

    From the WKB reformulation for the backward equation, we know
    \[
    \uh(\vxzh, t) = h \log \mathbb{E}^{\vxzh}\left(e^\frac{u_0\left(\cv(t)\right)}{h}\right),
    \]
    and hence by the definition of $\wrh$, we have
    \[
    \wrh(\rh+\alzh, t)=h\log \mathbb{E}^{\rh+\alzh}\left(e^\frac{w_0\left(\rh+\aluh(t)\right)}{h}\right).
    \]
    Since $\lim_{h \to 0^+}\rh = 0$ and $\lim_{h\to 0^+}\alzh = 0$, from Corollary \ref{cor:1dconverge},
    \[
    \lim_{h\to 0^+}{\wrh(\rh+\alzh,t)}=w(0,t),
    \]
    that is,
    \[
    \begin{aligned}
\lim_{h \to 0^+}h \log \mathbb{E}^{\vxzh}\left(e^\frac{u_0\left(\cv(t)\right)}{h}\right)&= \lim_{h \to 0^+} h\log \mathbb{E}^{\rh+\alzh}\left(e^\frac{w_0\left(\rh+\aluh(t)\right)}{h}\right)\\&=w(0, t)\\& = \sup_{\alpha\in[a, b]}\left\{w_0(\alpha)-I\left(\alpha;0,t\right) \right\}\\&=\sup_{\alpha\in[a, b]}\left\{u_0(\vx_0+\alpha\vec{\nu})-I\left(\alpha;0,t\right) \right\}\\
&=\sup_{\vec{y}\in S_0}\left\{u_0(\vec{y})-I\left(\frac{(\vec{y}-\vx_0)\cdot \vec{\nu}}{|\vec{\nu}|^2};0,t\right) \right\},
    \end{aligned}
    \]
where $S_0 =  \left\{\vec{x}_0 + \alpha \vec{\nu}: a\leq \alpha \leq b\right\}$ is the maximal closed line segment in the direction $\vec{\nu}$ that contains $\vx_0$.
\end{proof}

Next, we show the uniqueness of the mean-field limit.

\begin{lem}\label{lem:lzero}
    For fixed $\alpha \in [a, b]$, $\tilde{L}(\alpha, s)=0$ if and only if $s=\partial_p\tilde{H}(\alpha, 0) = \tilde{\Phi}^+(\alpha)-\tilde{\Phi}^-(\alpha).$
\end{lem}
\begin{proof}
    We know $\tilde{L}(\alpha, s) = p^\ast \cdot s-\tilde{H}(\alpha, p^\ast)$ for some $p^\ast \in \mathbb{R}$ if and only if $s=\partial_p\tilde{H}(\alpha, p^\ast)$. Moreover, this $p^\ast$ is unique. Indeed, suppose $\tilde{L}(\alpha, s) = p^\ast \cdot s-\tilde{H}(\alpha, p^\ast)= \tilde{p}^\ast \cdot s-\tilde{H}(\alpha, \tilde{p}^\ast)$ for some $\tilde{p}^\ast, p^\ast \in \mathbb{R}$. Then $s= \partial_p\tilde{H}(\alpha, p^\ast) = \partial_p\tilde{H}(\alpha, \tilde{p}^\ast)$. Since $H$ is strictly convex, it follows that $ p^\ast= \tilde{p}^\ast$.

    If $s=\partial_p\tilde{H}(\alpha, 0)$, then $\tilde{L}(\alpha, s) = 0 \cdot s - \tilde{H}(\alpha, 0) = 0$. Conversely, if $\tilde{L}(\alpha, s) =0$, note that $\tilde{L}(\alpha, s) = p^\ast \cdot s-\tilde{H}(\alpha, p^\ast)$ for $p^\ast =0$. Therefore, $s=\partial_p\tilde{H}(\alpha, 0)$.
\end{proof}
\begin{lem}
Let $t>0$ and $\alpha \in [a, b]$. Then there exists a unique path $\eta:[0,t]\to [a,b]$ such that $I(\eta(t);\alpha,t)=0$.
\end{lem}
\begin{proof}
Let $(\eta, v, l)$ be a solution to ${\rm SP}(\alpha)$ such that $I(\eta(t);\alpha,t)=0$. Note in \eqref{eqn:1dpath}, we have $l(s)=0$ and $\dot{\eta}(s)=v(s)$ whenever $\eta(s)\in (a, b)$.
\begin{itemize}
    \item[(a)] Suppose $\alpha \in (a, b)$. Define $t_0$ as the first time $\eta$ reaches the boundary of the domain, that is,
\[
t_0 :=\inf \left\{s \in [0, \infty):\eta(s)\in \{a, b\}\right\}.
\]

If $t_0 \geq t$, then $\eta(s) \in (a, b)$ for all $s \in [0, t]$. Since $\tilde{L}\geq 0$ and $I(\eta(t);\alpha,t)=0$, it follows that $\tilde{L}(\eta(s), \dot{\eta}(s)) = 0$ on $[0,t]$. By Lemma \ref{lem:lzero}, we know
\begin{equation}\label{eqn:meanode}
\dot{\eta}(s) = \partial_p\tilde{H}(\eta(s), 0) = \tilde{\Phi}^+(\eta(s))-\tilde{\Phi}^-(\eta(s))
\end{equation}
for $s \in [0, t]$. Since $\tilde{\Phi}^+$ and $\tilde{\Phi}^-$ are locally Lipschitz, and $\eta(0)=\alpha$, hence the solution $\eta$ to \eqref{eqn:meanode} is unique in $AC([0,t], \mathbb{R})$.

If $t_0 < t$, then $\eta (s) \in (a , b)$ for $s \in [0, t_0)$ and $\eta$ again solves \eqref{eqn:meanode}, hence is unique up to time $t_0$, where it reaches a boundary point, say $\eta(t_0) = b$. We claim that $\eta(s) \equiv \eta(t_0)$ for $s\in (t_0, t]$. Suppose on the contrary that for some $t^\ast \in (t_0, t]$, $\eta(t^\ast) \in (a, b)$. Define the time $\eta$ leaves the endpoint $b$ by
\[
t_1 := \sup \left\{s \in [t_0, t^\ast] : \eta(s) \in \{a, b\} \right\}.
\]
Then $\eta(s) \in (a, b)$ for $s \in (t_1, t^\ast)$, with $\eta(t_1) = b$. By Lemma \ref{lem:lzero}, $\eta$ solves \eqref{eqn:meanode} on $(t_1, t^\ast)$ with initial condition $\eta(t_1) = b$. This contradicts the uniqueness of solutions to \eqref{eqn:meanode} near the endpoint $b$, as $\eta$ solves \eqref{eqn:meanode} on $(0, t_0)$ with $\eta(t_0)=b$. Therefore, it must be that $\eta(s) \equiv \eta(t_0)$ for $s \in (t_0, t]$ and thus $\dot{\eta}(s)=0$ for $s \in (t_0, t] $.

Furthermore, since $I(\eta(t);\alpha,t)=0$, Lemma \ref{lem:lzero} implies that $v$ satisfies
\[
v(s)= \tilde{\Phi}^+(\eta(t_0))-\tilde{\Phi}^-(\eta(t_0))
\]
for almost everwhere $s \in (t_0, t] $, and thus $l$ is uniquely defined via $\dot{\eta}(s)=0=v(s)-l(s)$.  Thus, in this case as well, the solution $\eta$ is unique in $AC([0,t], \mathbb{R})$.

\item[(b)] Suppose $\alpha = a$ or $\alpha = b$. Without loss of generality, assume $\alpha = b$.
If $\tilde{\Phi}^+(b) - \tilde{\Phi}^-(b) \geq 0$, then to ensure $I(\eta(t); \alpha, t) = 0$, the path must remain constant: $\eta(s) \equiv b$ for all $s \in [0, t]$. In this case, we have $v(s) = l(s) = \tilde{\Phi}^+(b) - \tilde{\Phi}^-(b)$, so that $\dot{\eta}(s) = v(s) - l(s) = 0$, as required.

On the other hand, if $\tilde{\Phi}^+(b) - \tilde{\Phi}^-(b) < 0$, then the dynamics push $\eta$ away from the boundary into the interior $(a, b)$. While $\eta(s) \in (a, b)$, the path solves
\[
\dot{\eta}(s) = \partial_p \tilde{H}(\eta(s), 0) = \tilde{\Phi}^+(\eta(s)) - \tilde{\Phi}^-(\eta(s)).
\]
If $\eta$ reaches the opposite endpoint $a$ at some time $t_0 < t$, the situation can be analyzed similarly to case (a) above.
\end{itemize}
\end{proof}

Finally, we state the law of large numbers below as a consequence of the large deviation principle.
\begin{cor}
Let $\vx_0 \in \overline{\Omega}$, and for each small $h>0$, define $\vxzh = \vx_0+(\rh+\alzh)\vec{\nu}
\in \overline{\Omega}$ where $\rh \in \mathbb{R}$ satisfies $|\rh| < h$, $\alzh = \kzh h$ for some $\kzh \in \mathbb{Z}$, and
\[
\lim_{h \to 0^+}\rh=0, \quad \lim_{h\to 0^+}\alzh = 0.
    \] Let $\eta^\ast \in AC([0,t];\mathbb{R})$ be the unique path such that $I(\eta^\ast(t);0,t)=0$. Then $\vx^\ast(t):=\vx_0+\eta^\ast(t)\vec{\nu}$ is the mean-field limit path in the sense that the weak law of large numbers for process $\cv(t)$ holds at any time $t$. Moreover, for any $t>0$ and $\varepsilon>0$, there exists $h_0>0$ such that if $h<h_0$, then
    \[
    \mathbb{P}\left\{\left|\cv(t)-\vx^\ast(t)\right|\geq\varepsilon\right\} \leq e^{-\frac{\beta(\varepsilon, t)}{2h}},
    \]
    where $\beta(\varepsilon, t): = \inf_{\left|\vy-\vx^\ast(t)\right|\geq \varepsilon, \, \vy \in S_0} I\left(\frac{(\vec{y}-\vx_0)\cdot \vec{\nu}}{|\vec{\nu}|^2};0,t\right) >0$ with $I$ denoting the rate function defined in \eqref{eqn:defIratef} and $S_0 =  \left\{\vec{x}_0 + \alpha \vec{\nu}: a\leq \alpha \leq b\right\}$ is the maximal closed line segment in the direction $\vec{\nu}$ that contains $\vx_0$.
\end{cor}

\section{Convergence results on the whole domain}\label{sec:ubetaconvergence}

Previously, we established that for sufficiently small $h>0$, if the initial position
\[
\vxzh \in \Omegah = \left\{ \vxi = \vec{i} h; \,\,  \vec{i} \in \mathbb{N}^N,\, \vxi \in \overline{\Omega}\right\}
\]
of the rescaled stochastic process $\cv(t)$ lies on $S_0$—the maximal closed line segment in the direction $\vec{\nu}$ that contains $\vx_0$—then the desired large deviation principle holds. We now turn to the more general case where the starting point $\vxzh$ of $\cv(t)$ lies in $\overline{\Omega}$, but not necessarily on $S_0$, and satisfies $\lim_{h \to 0^+} \vxzh = \vx_0$. Throughout this discussion, we continue to assume $N = 2$ and $M = 1$. In general, the same results also hold for $N \geq 2$ and $M=1$. Additionally, we require that $\overline{\Omega}$ be convex; in general, the convergence result may fail to hold in nonconvex domains. (See Remark \ref{rem:convex})

Let $\vx_0 \in \overline{\Omega}$ and $\beta \in \mathbb{R}$. Define $\vec{\nu}_\perp := (\nu_{11}^+, -\nu_{12}^-)$, and note that $\vec{\nu} \cdot \vec{\nu}_{\perp} = 0$. Consider the line passing through $\vx_0 + \beta \vec{\nu}_{\perp}$ in the direction of $\vec{\nu}$, and denote the intersection of the line and $\overline{\Omega}$ as
\[
S_{\sbe}:=\{\vx_0+\beta\vec{\nu}_{\perp}+\alpha \vec{\nu}:\alpha\in \mathbb{R}\} \cap \overline{\Omega}=\left\{\vec{x}_0 +\beta\vec{\nu}_{\perp}+ \alpha \vec{\nu}: a_{\sbe}\leq \alpha \leq b_{\sbe}\right\},
\]
for some $a_{\sbe}, b_{\sbe} \in \mathbb{R}$. Note that it could be the case that $\vec{x}_0 +\beta\vec{\nu}_{\perp} \notin \overline{\Omega}$.

For a fixed $\beta$, we can deduce the same result as in the previous section. More specifically, assume $\cv(0)=\vxzh=\vx_0+\beta \vec{\nu}_\perp + (r+kh)\vec{\nu} \in \overline{\Omega}$ for some $r, \beta \in \mathbb{R}$ small, and $k \in \mathbb{Z}$. Then the process evolves within the discrete set
\begin{equation}
\begin{aligned}
 \hobrh:&=\left\{\vec{x}_0 +\beta \vec{\nu}_\perp+(r+\alpha) \vec{\nu} \in \overline{\Omega}: \alpha =kh, k  \in \mathbb{Z}\right\} \\ &=\left\{\vec{x}_0 +\beta \vec{\nu}_\perp+(r+\alpha) \vec{\nu}: \alpha =kh, k\in \mathbb{Z}, k_{\sa_{\sbe}, \sr, \V}\leq k\leq k_{\scb_{\sbe}, \sr, \V} \right\}
\end{aligned}
\end{equation}
for some integers $k_{\sa_{\sbe}, \sr, \V} \leq k_{\scb_{\sbe}, \sr, \V}$ that depend on $a_{\sbe}, b_{\sbe}, \beta, r$, $\vec{x}_0$ and $h$. For fixed $\beta, r$, we have $\lim_{h\to 0^+} k_{\sa_{\sbe}, \sr, \V}h=a_{\sbe}$ and $ \lim_{h\to 0^+} k_{\scb_{\sbe},\sr, \V}h=b_{\sbe}$. Denote
\[
\tobrh :=\{r+\aldh: \aldh=kh, k\in \mathbb{Z},  k_{\sa_{\sbe}, \sr, \V} \leq k \leq k_{\scb_{\sbe}, \sr, \V}\},
\]
and define $\wrhb:\tobrh \times [0,\infty) \to \mathbb{R}$ by
\[\wrhb \left(r+\aldh, t\right): = \uh(\vec{x}_0+\beta \vec{\nu}_\perp +(r+\aldh) \vec{\nu} , t).\]

Also define $\tilde{\Phi}^+_{\sbe}(r+\aldh):= \Phi^+(\vec{x}_0+\beta \vec{\nu}_\perp+(r+\aldh) \vec{\nu})$, $\tilde{\Phi}^-_{\sbe}(r+\aldh):= \Phi^-(\vec{x}_0+\beta \vec{\nu}_\perp+(r+\aldh) \vec{\nu})$, $w_{\zero}^{\sbe}(r+\aldh):=u_0\left(\vec{x}_0 +\beta \vec{\nu}_\perp+(r+\aldh) \vec{\nu}\right)$ for any $r+\aldh \in \tobrh$.

Then $\wrhb$ solves
\begin{equation}\label{eqn:1ddistbeta}
    \left\{\begin{aligned}
\pt_t \wrhb(r+\aldh,t) = & \mathbbm{1}_{\{\aldh
\leq (k_{\scb_{\sbe}, \sr, \V}-1)h\}}  \tilde{\Phi}^+_{\sbe}
(r+\aldh) \left( e^{\frac{\wrhb(r+\aldh+ h, t)-\wrhb(r+\aldh, t)}{h}} - 1\right)\\ &+ \mathbbm{1}_{\{\aldh
\geq (k_{\sa_{\sbe}, \sr, \V}+1)h\}}\tilde{\Phi}^-_{\sbe}(r+\aldh)\left( e^{  \frac{\wrhb(r+\aldh- h, t)-\wrhb(r+\aldh, t)}{h} } - 1\right), \\  &\qquad\qquad\qquad\qquad\qquad\qquad\qquad\text{ for }  (r+\aldh, t) \in \tobrh \times (0, \infty),\\
\wrhb(r+\aldh, 0)=&w_{\zero}^{\sbe}(r+\aldh),
    \end{aligned}
    \right.
\end{equation}
with the discrete Hamiltonian
\begin{equation}
\begin{aligned}
\tHrhb: \ell^{\8}(\tobrh) \to \ell^{\8}(\tobrh); \quad  ( w(r+\aldh) )_{r+\aldh \in\tobrh} \mapsto ( \tHrhb(r+\aldh, w(r+\aldh), w(\cdot)) )_{r+\aldh \in\tobrh}
\end{aligned}
\end{equation}
defined by
\begin{equation}\label{eqn:defHtildehbeta}
\begin{aligned}
\tHrhb\left(r+\aldh , w(r+\aldh),w\right) :=&\mathbbm{1}_{\{\aldh \leq (k_{\scb_{\sbe},\sr, \V}-1)h\}}  \tilde{\Phi}_{\sbe}^+(r+\aldh) \left( e^{\frac{w(r+\aldh+ h)-w(r+\aldh)}{h}} - 1\right)\\ &+ \mathbbm{1}_{\{\aldh
\geq (k_{\sa_{\sbe}, \sr, \V}+1)h\}}\tilde{\Phi}_{\sbe}^-(r+\aldh)\left( e^{  \frac{w(r+\aldh- h)-w(r+\aldh)}{h} } - 1\right).
\end{aligned}
\end{equation}

Define the upper semicontinuous envelope of $\wrhb$ by
\[
\overline{w}^{\sbe}(\alpha, t): = \limsup_{h \to 0^+, \, r \to 0, \, \alhi \to \alpha, \, r+\alhi \in \tobrh , \,\ti\to t} \wrhb(r+\alhi, \ti)
\]
for $(\alpha, t) \in [a_{\sbe}, b_{\sbe}] \times [0, \infty)$, and the lower semicontinuous envelope of $\wrhb$ by

\[
\underline{w}^{\sbe}(\alpha, t): = \liminf_{h \to 0^+, \, r \to 0, \, \alhi \to \alpha, \, r+\alhi \in \tobrh , \,\ti\to t} \wrhb(r+\alhi, \ti)
\]
for $(\alpha, t) \in [a_{\sbe}, b_{\sbe}] \times [0, \infty)$.


Analogous to Theorem \ref{thm:wbarsubsln} and Corollary \ref{cor:1dconverge}, $\wrhb$ converges to $w^{\sbe}$ as $r,h \to 0^+$, where $w^{\sbe}$ is the solution of a Cauchy problem with a Neumann boundary condition. This is summarized in the following corollary.

\begin{cor}
For fixed $\beta \in \mathbb{R}$,  $\overline{w}^{\sbe} = \underline{w}^{\sbe} =w ^{\sbe}$ where $w^{\sbe} : \left[a_{\sbe}, b_{\sbe}\right] \times [0, \infty) \to \mathbb{R}$ solves the following Cauchy problem with the Neumann boundary condition:
\begin{equation} \label{eqn:limitbeta}
    \left\{\begin{aligned}
     \pt_t w^{\sbe}(\alpha, t) =&  \tilde{\Phi}^+_{\sbe}(\alpha) \left( e^{\partial_\alpha w^{\sbe}(\alpha,t)} - 1\right)+ \tilde{\Phi}^-_{\sbe}(\alpha)\left( e^{-\partial_\alpha w^{\sbe}(\alpha,t)} - 1\right) \quad \text{ in } \left(a_{\sbe} ,b_{\sbe} \right) \times (0, \infty),\\ -\partial_\alpha w^{\sbe}(a, t) =& \partial_\alpha w^{\sbe}(b, t)=0 \qquad \qquad \qquad\qquad\qquad \qquad \qquad \qquad \,\text{ for } t \in (0, \infty),\\w^{\sbe}(\alpha, 0)=&w_{\zero}^{\sbe} (\alpha) \qquad \qquad \qquad\qquad\qquad \qquad \qquad \qquad \qquad \quad \,\text{ on } \left[a_{\sbe}, b_{\sbe}\right],
    \end{aligned}
    \right.
\end{equation}
where $\tilde{\Phi}^+_{\sbe}(\alpha)= \Phi^+(\vec{x}_0+ \beta \vec{\nu}_\perp+\alpha \vec{\nu})$, $\tilde{\Phi}^-_{\sbe}(\alpha)= \Phi^-(\vec{x}_0+ \beta \vec{\nu}_\perp+\alpha\vec{\nu})$, $w_{\zero}^{\sbe}(\alpha)=u_0 \left(\vec{x}_0+\beta\vec{\nu}_\perp +\alpha \vec{\nu}\right)$ for $\alpha \in \left[a_{\sbe}, b_{\sbe}\right]$.
\end{cor}

Note that the Hamiltonian
\begin{equation} \label{eqn:Hbeta}
\tilde{H}_{\sbe}(\alpha, p):=\tilde{\Phi}^+_{\sbe} (\alpha) \left( e^p - 1\right)+ \tilde{\Phi}^-_{\sbe} (\alpha)\left( e^{-p} - 1\right)
\end{equation}
is strictly convex in $p$ and is coercive, i.e., $\lim_{|p|\to +\infty} \inf_{\alpha \in \left[a_{\sbe}, b_{\sbe}\right]}\tilde{H}_{\sbe} (\alpha, p) =\infty$.

From \eqref{eqn:1dvf}, \eqref{eqn:1dpath},and Proposition \ref{thm:rf}, the viscosity solution $w^{\sbe}:\left[a_{\sbe},b_{\sbe}\right] \times [0, \infty) \to \mathbb{R}$ to \eqref{eqn:limitbeta} can be represented by
\begin{equation}\label{eqn:wbetaocf}
w^{\sbe}(\alpha, t) = \sup_{y\in[a_{\sbe}, b_{\sbe}]}\left\{w_{\zero}^{\sbe}(y)-I_{\sbe}\left(y;\alpha,t\right) \right\},
\end{equation}
where
\begin{equation}
    \quad I_{\sbe} \left(y;\alpha,t\right):=\inf \left\{\left.\int_0^t  \tilde{L}_{\sbe} \left(\eta(s), v(s) \right)ds\right| \eta(t)=y, (\eta, v, l) \in {\rm SP}_{\sbe}(\alpha) \right\},
\end{equation}
$\tilde{L}_{\sbe}$ is the Legendre transform of $\tilde{H}_{\sbe}$ defined in \eqref{eqn:Hbeta}, and ${\rm SP}_{\sbe}(\alpha)$ denotes the family of solutions to the Skorokhod problem: for given $\alpha \in \left[a_{\sbe}, b_{\sbe}\right]$ and $v\in L^1_{\rm loc} \left([0, \infty), \mathbb{R}\right)$, find a pair $(\eta, l)\in {\rm AC}_{\rm loc} \left([0, \infty), \mathbb{R}\right) \times L^1_{\rm loc} \left([0, \infty), \mathbb{R}\right)$ such that
\begin{equation}
    \left\{
    \begin{aligned}
     \eta(0) = &\alpha,\\
     \eta(s) \in & \left[a_{\sbe}, b_{\sbe}\right] \qquad\qquad\qquad \text{ for all } s\in [0, \infty),\\
     \dot{\eta}(s)+l(s) n_{\sbe}(\eta(s))=&v(s) \qquad \qquad\qquad\quad\,\,\, \text{ for a.e. } s \in (0, \infty),\\
     l(s) \geq& 0 \qquad \qquad \qquad \qquad\,\,\,\, \text{ for a.e. } s \in (0, \infty),\\
     l(s)= &0 \text{ if } \eta(s) \in (a_{\sbe}, b_{\sbe}) \quad \,\, \text{ for a.e. } s \in (0, \infty),
    \end{aligned}
    \right.
\end{equation}
where $n_{\sbe}:\{a_{\sbe}, b_{\sbe}\} \to \{-1, 1\}$ is defined by
$n(a_{\sbe})=-1$ and $n(b_{\sbe})=1$.

The solution $w^{\sbe}$ to the limiting equation \eqref{eqn:limitbeta} is actually close to the solution $w$ to the limiting equation \eqref{HJE_NM} pointwise in the interior of $[a,  b]$ when $\beta$ is small, as summarized in the following proposition.


\begin{prop}
   Assume that $\Omega$ is convex, and let $u_0 \in C(\overline{\Omega})$. For $\alpha \in [a, b]$, define
   \[
   w_0(\alpha):=u_0\left(\vx_0+\alpha\vec{\nu}\right).
   \]
   Let $\beta>0$. For $\alpha \in [a_{\sbe}, b_{\sbe}]$, define
   \[
   w_{\zero}^{\sbe}(\alpha):=u_0\left(\vx_0+\alpha \vec{\nu}+\beta \vec{\nu}_\perp\right).
   \]
   Let $w$ be the solution to \eqref{HJE_NM} with the initial condition $w_0$, and let $w^{\sbe}$ be the solution to \eqref{eqn:limitbeta} with the initial condition $w_{\zero}^{\sbe}$. Then, for any $\alpha \in (a, b)$ and $t \in [0, \infty)$,
   \[
   \lim_{\beta \to 0} w^{\sbe}(\alpha,t)=w(\alpha,t).
   \]
\end{prop}

\begin{proof}
We prove the result in two steps.

Step 1: We first establish the result under the assumption that $u_0 \in W^{1, \infty}(\Omega) \cap C(\overline{\Omega})$. Later, we will only require $u_0 \in C(\overline{\Omega})$. When $u_0 \in W^{1, \infty}(\Omega) \cap C(\overline{\Omega})$, it follows that
    \begin{equation}\label{eqn:w0lip}
    w_0 \in \mathrm{Lip}([a, b]), \quad w_{\zero}^{\sbe} \in \mathrm{Lip}([a_{\sbe}, b_{\sbe}]),
    \end{equation}
    with the same Lipschitz constant $\left\| Du_0\right\|_{L^\infty(\Omega)}$.

    Since $\overline{\Omega}$ is convex, it follows that
    \begin{equation}\label{eqn:endpointsdif}
    |a_{\sbe}-a|\leq \omega_0(|\beta|), \qquad |b_{\sbe}-b|\leq \omega_0(|\beta|),
    \end{equation}
    for some modulus of continuity $\omega_0$. In particular, for $|\beta|$ sufficiently small, we have $\alpha \in (a_{\sbe}, b_{\sbe})$ since $\alpha \in (a, b)$. For any $\alpha\in [a, b]$ and $\alpha'\in [a_{\sbe}, b_{\sbe}]$, we estimate
    \[
    \left|\tilde{\Phi}^+ (\alpha)-\tilde{\Phi}^+_{\sbe}\left(\alpha'\right)\right| = \left|\Phi^+\left(\vec{x}_0+\alpha \vec{\nu}\right) -\Phi^+\left(\vec{x}_0+ \beta \vec{\nu}_\perp+\alpha' \vec{\nu}\right)\right| \leq \omega^+(|\beta|,  \left|\alpha-\alpha'\right|),
    \]
    where $\omega^+: [0, \infty)\to [0, \infty)$ is a modulus of continuity.
    Similarly,
    \[
    \left|\tilde{\Phi}^- (\alpha)-\tilde{\Phi}^-_{\sbe}\left(\alpha'\right)\right| = \left|\Phi^-\left(\vec{x}_0+\alpha \vec{\nu}\right) -\Phi^-\left(\vec{x}_0+ \beta \vec{\nu}_\perp+\alpha' \vec{\nu}\right)\right| \leq \omega^-(|\beta|,  \left|\alpha-\alpha'\right|),
    \]
    where $\omega^-: [0, \infty)\to [0, \infty)$ is a modulus of continuity. Hence, for any $\alpha\in [a, b]$, $\alpha'\in [a_{\sbe}, b_{\sbe}]$, and any $p \in \mathbb{R}$ with $|p| \leq C$ for some constant $C>0$, we have
    \begin{equation}
    \left|\tilde{H}(\alpha, p) - \tilde{H}_{\sbe}(\alpha', p)\right| \leq \omega_C(|\beta|,  \left|\alpha-\alpha'\right|),
    \end{equation}
    where $\omega_C:[0, \infty) \to [0, \infty)$ is some modulus of continuity. Furthermore, since $\tilde{H}$ and $\tilde{H}_{\sbe}$ are superlinear in $p$, for any $\alpha\in [a, b]$, $\alpha'\in [a_{\sbe}, b_{\sbe}]$, and any $v \in \mathbb{R}$ with $|v| \leq C_1$ for some constant $C_1>0$, we have
    \[
    \begin{aligned}
    \tilde{L}(\alpha, v)&=\sup_{p\in \mathbb{R}}p\cdot v-\tilde{H}(\alpha,p) = \sup_{|p|\leq \tilde{C}} p\cdot v-\tilde{H}(\alpha,p),\\
    \tilde{L}_{\sbe}(\alpha, v)&=\sup_{p\in \mathbb{R}}p\cdot v-\tilde{H}_{\sbe}(\alpha,p) = \sup_{|p|\leq \tilde{C}} p\cdot v-\tilde{H}_{\sbe} (\alpha,p),
    \end{aligned}
    \]
    where $\tilde{C}=\tilde{C}\left(\Phi^+, \Phi^-, C_1\right)>0$ is independent of $\beta$. Therefore, for any $\alpha\in [a, b]$, $\alpha'\in [a_{\sbe}, b_{\sbe}]$, and any $v \in \mathbb{R}$ with $|v| \leq C_1$,
    \begin{equation}\label{eqn:Ldif}
    \left|\tilde{L}(\alpha, v) - \tilde{L}_{\sbe}(\alpha', v)\right| \leq \omega_{\tilde{C}}(|\beta|,  \left|\alpha-\alpha'\right|).
    \end{equation}
    Moreover, for any $\alpha\in [a, b]$, $\alpha'\in [a_{\sbe}, b_{\sbe}]$,
    \begin{equation} \label{eqn:w0dif}
    \left|w_0(\alpha)-w_{\zero}^{\sbe}(\alpha')\right| =\left|u_0(\vec{x}_0 +\alpha \vec{\nu})-u_0 \left(\vec{x}_0+\beta\vec{\nu}_\perp +\alpha' \vec{\nu}\right)\right| \leq \omega \left(|\beta|, |\alpha-\alpha'|\right),
    \end{equation}
    where $\omega$ is some modulus of continuity.

    By \cite[Theorem 7.2]{Ishii2011}, for any $t \geq 0$, there exists a triple $(\eta, v, l) \in \text{SP}(0)$ such that
    \begin{equation}\label{eqn:wpath}
    w(\alpha, t)=w_0(\eta(t)) - \int_0^t  \tilde{L} \left(\eta(s), v(s) \right)ds.
    \end{equation}

    In particular, $(\eta, v, l) \in L^1_{\rm loc} \left([0, \infty), \mathbb{R}\right) \times {\rm AC}_{\rm loc} \left([0, \infty), \mathbb{R}\right) \times L^1_{\rm loc} \left([0, \infty), \mathbb{R}\right)$ solves
    \begin{equation}
    \left\{
    \begin{aligned}
     \eta(0) = &\alpha,\\
     \eta(s) \in & [a, b] \qquad\qquad\qquad \text{ for all } s\in [0, \infty),\\
     \dot{\eta}(s)+l(s) n(\eta(s))=&v(s) \qquad \qquad\qquad\, \text{ for a.e. } s \in (0, \infty),\\
     l(s) \geq& 0 \qquad \qquad \qquad \quad \text{ for a.e. } s \in (0, \infty),\\
     l(s)= &0 \text{ if } \eta(s) \in (a, b) \quad\text{ for a.e. } s \in (0, \infty),
    \end{aligned}
    \right.
    \end{equation}
    where $n:\{a, b\} \to \{-1, 1\}$ is defined by $n(a)=-1$ and $n(b)=1$. From \eqref{eqn:w0lip}, it follows that $v\in  L^\infty \left([0, \infty) \right)$ and
    \begin{equation}\label{eqn:velbd}
    \|v\|_\infty \leq C_1,
    \end{equation}
    for some constant $C_1=C_1\left(\Omega, H, \left\|Du_0\right\|_{L^\infty(\Omega)}\right)>0$. The proof is standard in the literature and is included in the Appendix for the reader's convenience (see Lemma \ref{lem:velbd}).

    Next, we show that
    \[
    w(\alpha, t) - w^{\sbe}(\alpha, t) \leq \tilde{w}\left(|\beta|\right),
    \]
    for some modulus of continuity $\tilde{\omega}$.

    \begin{enumerate}
        \item If $a=a_{\sbe}$ and $b=b_{\sbe}$, then $\left(\eta, v, l\right)$ is an admissible triple for $w^\beta(\alpha, t)$, and
    \begin{equation}
    \begin{aligned}
     w^{\sbe}(\alpha, t) &\geq w_{\zero}^{\sbe}(\eta(t)) - \int_0^t  \tilde{L}_{\sbe} \left(\eta(s), v(s) \right)ds\\
     &\geq w_0(\eta(t))-\omega\left(|\beta|\right)-\int_0^t  \tilde{L} \left(\eta(s), v(s) \right)ds -t\omega_{\tilde{C}}\left(|\beta|\right)\\
     &\geq w(\alpha, t)-\tilde{w}(|\beta|),
    \end{aligned}
    \end{equation}
    where the second inequality follows from \eqref{eqn:w0dif} and \eqref{eqn:Ldif}, and $\tilde{w}$ is some modulus of continuity. Hence, in this case,
    \[
    w(\alpha, t) - w^{\sbe}(\alpha, t) \leq \tilde{w}\left(|\beta|\right).
    \]
        \item Assume $\max \left\{|a_{\sbe}-a|,|b_{\sbe}-b|\right\} >0$. We modify $(\eta, v, l)$ to construct an admissible triple for $w^{\sbe}(\alpha, t)$. Since $a_{\sbe}$ is close to $a$ and $b_{\sbe}$ is close to $b$ for $|\beta|$ sufficiently small, without loss of generality, we can assume
    \[
    a_{\sbe} \leq a < b_{\sbe} \leq b \quad \text{and} \quad \max\{|a - a_{\sbe}|, |b - b_{\sbe}|\} < \tfrac{1}{16}(b- a).
    \]
    The analysis for other orderings of $a_{\sbe}, b_{\sbe}, a, b$ is similar.

    By \cite[Theorem 4.1]{Ishii2011}, there exists a pair $(\eta_{\sbe}, l_{\sbe})\in {\rm AC}_{\rm loc} \left([0, \infty), \mathbb{R}\right) \times L^1_{\rm loc} \left([0, \infty), \mathbb{R}\right)$ satisfying
    \begin{equation}
    \left\{
    \begin{aligned}
     \eta_{\sbe} (0) = &\alpha,\\
     \eta_{\sbe}(s) \in\, & [a_{\sbe}, b_{\sbe}] \qquad\qquad\qquad \,\,\, \text{ for all } s \in [0, \infty),\\
     \dot{\eta}_{\sbe} (s)+l_{\sbe}(s) n_{\sbe}(\eta_{\sbe}(s))=&v(s) \qquad \qquad\qquad \,\,\quad \text{ for a.e. } s \in (0, \infty),\\
     l_{\sbe}(s) \geq& 0 \qquad \qquad \qquad \qquad \quad \text{ for a.e. } s \in (0, \infty),\\
     l_{\sbe}(s)= &0 \text{ if } \eta_{\sbe}(s) \in (a_{\sbe}, b_
     \beta) \quad\text{ for a.e. } s \in (0, \infty),
    \end{aligned}
    \right.
    \end{equation}
    where $n_{\sbe}:\{a_{\sbe}, b_{\sbe}\} \to \{-1, 1\}$ is defined by $n_{\sbe}(a_{\sbe})=-1$ and $n_{\sbe}(b_{\sbe})=1$, that is, $\left(\eta_{\sbe}, v, l_{\sbe}\right) \in \text{SP}_{\sbe}(\alpha)$, and
    \begin{equation}\label{eqn:wbetapath}
    w^{\sbe}(\alpha, t) \geq w_{\zero}^{\sbe}(\eta_{\sbe}(t)) - \int_0^t  \tilde{L}_{\sbe} \left(\eta_{\sbe}(s), v(s) \right)ds.
    \end{equation}

    We claim $|\eta(s) - \eta_{\sbe}(s)| \leq 2 \max \left\{|a_{\sbe}-a|,|b_{\sbe}-b|\right\}$ for $s \in [0, \infty)$. Define
    \[
    t_0:= \inf\left\{s\geq0: |\eta(s) - \eta_{\sbe}(s)| > 2 \max \left\{|a_{\sbe}-a|,|b_{\sbe}-b|\right\} \right\}
    \]
    Suppose $t_0<\infty$. By the definition of $t_0$ and the continuity of $\eta, \eta_{\sbe}$, we have $|\eta(t_0) - \eta_{\sbe}(t_0)| = 2 \max \left\{|a_{\sbe}-a|,|b_{\sbe}-b|\right\}$.

    There are several cases to be considered.
    \begin{enumerate}
    \item If $\eta(t_0) \in (a, b)$ and $\eta_{\sbe}(t_0) \in (a_{\sbe}, b_{\sbe})$, then there exists a constant $\varepsilon>0$ such that $\left|\dot{\eta}(s) - \dot{\eta}_{\sbe}(s)\right|= |v(s)-v(s)|=0$ for all $s\in[t_0, t_0+\varepsilon]$. Hence,$|\eta(s) - \eta_{\sbe}(s)| \leq 2 \max \left\{|a_{\sbe}-a|,|b_{\sbe}-b|\right\}$ for $ s \in [0, t_0+\varepsilon]$.
    \item If $\eta_{\sbe}(t_0)=a_{\sbe}$, then $\eta(t_0) \in (a, b)$.  There exists a constant $\varepsilon>0$ such that $\eta(s) \in (a, b)$ and $\eta_{\sbe}(s) \in[a_{\sbe}, \eta(s))$ for all $s\in[t_0, t_0+\varepsilon]$. For $s\in[t_0, t_0+\varepsilon]$, we have
    \[
    \begin{aligned}
    |\eta(s)-\eta_{\sbe}(s)|=& \eta(s)-\eta_{\sbe}(s)\\
    =& \eta(t_0) - \eta_{\sbe}(t_0)+ \int_{t_0}^s \dot{\eta} (\tau)- \dot{\eta}_{\sbe}(\tau)d\tau\\
    =&\eta(t_0) - \eta_{\sbe}(t_0) +\int_{t_0}^sv(s)-\left(v(s)-l_{\sbe}(s) n_{\sbe}(\eta_{\sbe}(s)\right)ds \\
    =&\eta(t_0) - \eta_{\sbe}(t_0)+ \int_{t_0}^sl_{\sbe}(s) n_{\sbe}(\eta_{\sbe}(s)) d\\
   \leq & \eta(t_0) - \eta_{\sbe}(t_0).
    \end{aligned}
    \]
    Hence, $|\eta(s) - \eta_{\sbe}(s)| \leq 2 \max \left\{|a_{\sbe}-a|,|b_{\sbe}-b|\right\}$ for $ s \in [0, t_0+\varepsilon]$.
     \item If $\eta_{\sbe}(t_0)=b_{\sbe}$, then $\eta(t_0) \in (a, b)$. Note that we must have    $\eta(t_0)
     < \eta_{\sbe}(t_0)$.  Similarly, there exists a constant $\varepsilon>0$ such that $\eta(s) \in (a, b)$ and $\eta_{\sbe}(s) \in(\eta(s), b_{\sbe}]$ for all $s\in[t_0, t_0+\varepsilon]$. For $s\in[t_0, t_0+\varepsilon]$, we have
    \[
    \begin{aligned}
    |\eta(s)-\eta_{\sbe}(s)|=& \eta_{\sbe}(s)-\eta(s)\\
    =&\eta_{\sbe}(t_0)- \eta(t_0) + \int_{t_0}^s \dot{\eta}_{\sbe} (\tau)- \dot{\eta}(\tau)d\tau\\
    =&\eta_{\sbe}(t_0)- \eta(t_0) +\int_{t_0}^s\left(v(s)-l_{\sbe}(s) n_{\sbe}(\eta_{\sbe}(s)\right)-v(s)ds \\
    =&\eta_{\sbe}(t_0)- \eta(t_0)- \int_{t_0}^sl_{\sbe}(s) n_{\sbe}(\eta_{\sbe}(s)) d\\
   \leq &\eta_{\sbe}(t_0)- \eta(t_0).
    \end{aligned}
    \]
    Hence, $|\eta(s) - \eta_{\sbe}(s)| \leq 2 \max \left\{|a_{\sbe}-a|,|b_{\sbe}-b|\right\}$ for $ s \in [0, t_0+\varepsilon]$.
    \item[(d)]If $\eta(t_0)=a$, then $\eta_{\sbe}(t_0) \in \left(a_{\sbe}, b_{\sbe}\right)$. Similar to case {\rm (c)},  we must have $\eta(t_0)
     < \eta_{\sbe}(t_0)$ and there exists a constant $\varepsilon>0$ such that $|\eta(s) - \eta_{\sbe}(s)| \leq 2 \max \left\{|a_{\sbe}-a|,|b_{\sbe}-b|\right\}$ for $ s \in [0, t_0+\varepsilon]$.
    \item[(e)] If $\eta(t_0)=b$, then $\eta_{\sbe}(t_0) \in (a_{\sbe}, b_{\sbe})$. Similar to case {\rm (b)}, there exists a constant $\varepsilon>0$ such that $|\eta(s) - \eta_{\sbe}(s)| \leq 2 \max \left\{|a_{\sbe}-a|,|b_{\sbe}-b|\right\}$ for $ s \in [0, t_0+\varepsilon]$.
    \end{enumerate}

    In each case, there always exists $\varepsilon>0$ such that $|\eta(s) - \eta_{\sbe}(s)| \leq 2 \max \left\{|a_{\sbe}-a|,|b_{\sbe}-b|\right\}$ for $ s \in [0, t_0+\varepsilon]$, which contradicts the definition of $t_0$. Hence, $|\eta(s) - \eta_{\sbe}(s)| \leq 2 \max \left\{|a_{\sbe}-a|,|b_{\sbe}-b|\right\}$ for $s \in [0, \infty)$.

    From \eqref{eqn:Ldif}, \eqref{eqn:w0dif}, \eqref{eqn:wpath}, \eqref{eqn:wbetapath}, and \eqref{eqn:velbd}, we deduce that
    \[
    \begin{aligned}
    &w(\alpha, t)-w^{\sbe}(\alpha, t)\\
    \leq &w_0(\eta(t)) - w_{\zero}^{\sbe}(\eta_{\sbe}(t)) + \int_0^t  \left(\tilde{L}_{\sbe} \left(\eta_{\sbe}(s), v(s) \right) - \tilde{L} \left(\eta(s), v(s) \right) \right) ds,\\
    \leq &\omega \left(|\beta|, 2\max \left\{|a_{\sbe}-a|,|b_{\sbe}-b|\right\}\right) + t\omega_{\tilde{C}}\left(|\beta|, 2\max \left\{|a_{\sbe}-a|,|b_{\sbe}-b|\right\}\right),\\
    \leq &\tilde{\omega}\left(|\beta|\right),
    \end{aligned}
    \]
    where $\tilde{\omega}$ is a modulus of continuity and the last inequailty follows from \eqref{eqn:endpointsdif}.
    \end{enumerate}

     A symmetric argument yields
    \[
    w^{\sbe}(\alpha, t) - w(\alpha, t) \leq \hat{\omega}\left(|\beta|\right),
    \]
    for some modulus of continuity $\hat{w}$. The key point is for any $\alpha \in [a_{\sbe}, b_{\sbe}]$, the velocity bound $C_1>0$ in \eqref{eqn:velbd} still holds for any $v_{\sbe}$ in triple $\left(\eta_{\sbe}, l_{\sbe}, v_{\sbe}\right)\in \mathrm{SP}_{\sbe} (\alpha)$  such that
    \[
    w^{\sbe} (\alpha, t)=w_{\zero}^{\sbe}(\eta_{\sbe}(t)) - \int_0^t  \tilde{L}_{\sbe} \left(\eta_{\sbe}(s), v_{\sbe}(s) \right)ds.
    \]
    That is, each optimal trajectory satisfying the Skorokhod problem under the $\beta$-perturbation maintains the same uniform velocity bound as in the unperturbed setting (See Lemma \ref{lem:velbd}). Hence, the conclusion holds for $u_0 \in W^{1, \infty}(\Omega) \cap C(\overline{\Omega})$.

    Step 2: If $u_0 \in C(\overline{\Omega})$, we can choose a sequence $\left\{\uzk\right\}_{k\in \mathbb{N}} \subset W^{1,\infty}\cap C(\overline{\Omega})$ such that $|\uzk(x)-u_0(x)|\leq \frac{1}{k}$ for all $x \in \overline{\Omega}$. For $\alpha \in [a, b]$, define $\wzk(\alpha):=\uzk(\vec{x}_0 +\alpha\vec{\nu})$, and for $\alpha'\in [a_{\sbe}, b_{\sbe}]$, define $\wzk^{\sbe}(\alpha'):=\uzk\left(\vec{x}_0+\beta\vec{\nu}_\perp +\alpha' \vec{\nu}\right)$. Let $w_{\sk}:[a, b] \to \mathbb{R}$ be the solution to \eqref{HJE_NM} with initial data $\wzk$, and $w_{\sk}^{\sbe}:[a_{\sbe}, b_{\sbe}] \to \mathbb{R}$ be the solution to \eqref{eqn:limitbeta} with initial data $\wzk^{\sbe}$.

    We know
    \[
    w_{\sk}(\alpha, t) = \sup\left\{\left.\wzk(\eta(t))-\int_0^t \tilde{L}\left(\eta(s), v(s) \right) ds \right|\left(\eta, v, l\right) \in {\rm SP}(\alpha)\right\},
    \]
    By comparison with \eqref{eqn:1dvf}, we deduce that for any $\alpha \in [a, b]$,
    \begin{equation}
    \left| w_{\sk}(\alpha, t) -w(\alpha,t)\right| \leq \max_{\alpha \in [a, b]} \left|\wzk(\alpha)-w_0(\alpha) \right| \leq \frac{1}{k}.
    \end{equation}
    Similarly, for any $\alpha' \in [a_{\sbe}, b_{\sbe}]$,
    \begin{equation}
    \left| w_{\sk}^{\sbe}(\alpha', t) -w^{\sbe}(\alpha',t)\right|\leq \frac{1}{k}.
    \end{equation}
    Now, for each $k\in \mathbb{N}$, since $\uzk \in W^{1,\infty}(\Omega) \cap C(\overline{\Omega})$, it follows from Step 1 that for any $\alpha \in (a, b)$ and $t\geq 0$,
    \begin{equation}
    \lim_{\beta \to 0} w_{\sk}^{\sbe}(\alpha,t) =w_{\sk}(\alpha, t).
    \end{equation}
    Therefore, for sufficiently small $|\beta|$, we have
    \[
    \begin{aligned}
    \left|w^{\sbe} (\alpha, t) -w(\alpha, t)\right|&\leq\left|w^{\sbe} (\alpha, t) -w_{\sk}^{\sbe}(\alpha, t)\right| + \left|w_{\sk}^{\sbe} (\alpha, t) -w_{\sk}(\alpha, t)\right|+\left|w_{\sk} (\alpha, t) -w(\alpha, t)\right|\\
    &\leq \left|w_{\sk}^{\sbe} (\alpha, t) -w_{\sk}(\alpha, t)\right| + \frac{2}{k}
    \end{aligned}
    \]
    Sending $\beta \to 0$ and then $k \to \infty$, we obtain the desired conclusion.
\end{proof}

We have shown that the solutions $w^{\sbe}$ and $w$ to the limiting equations are close when $|\beta|$ is small. Moreover, the discrete solutions $\wrh$ and $\wrhpb$ are close when their domains $\torh, \tobrhp$ have the same number of grid points and the mesh sizes $h'$ and $h$ are comparable. A more precise statement is given in the following lemma.

\begin{lem}\label{lem:endisbeta} Let $u_0 \in W^{1, \infty}(\Omega) \cap C(\overline{\Omega})$, and let the mesh sizes $h, h' > 0$ satisfy $|h' - h| = \mathcal{O}(h^2)$. Suppose $r, \beta \in \mathbb{R}$ are small. Fix $\vx_0 \in \overline{\Omega}$. Let $\wrh: \torh \to \mathbb{R}$ be the solution to \eqref{eqn:1dv_h} with initial condition
\[
w_0(r + \aldh) = u_0\left(\vx_0 + (r + \aldh)\vec{\nu}\right), \quad \text{for all } r + \aldh \in \torh.
\]
Let $\wrhpb: \tobrhp \to \mathbb{R}$ be the solution to \eqref{eqn:1ddistbeta} with initial condition
\[
w_{\zero}^{\sbe}(r + \alpha_{h'}) = u_0\left(\vx_0 + (r + \alpha_{h'})\vec{\nu} + \beta \vec{\nu}_\perp\right), \quad \text{for all } r + \alpha_{h'} \in \tobrhp.
\]
Assume both meshes have the same cardinality, i.e. $\left|\torh\right|=\left|\tobrhp\right|=n$ for some $n\in \mathbb{N}$. Label the points in $\torh$ by increasing order as
        \[
        a\leq r+k_1 h < \cdots <r+k_i h< \cdots < r+k_n h\leq b, \quad i=1, 2, \cdots, n,
        \]
        and likewise label those in  $\tobrhp$ as
        \[
        a_{\sbe}\leq r+k'_1 h' < \cdots <r+k'_i h'< \cdots < r+k'_n h'\leq b_{\sbe}, \quad i=1, 2, \cdots, n,
        \]
where $k_i, k'_i \in \mathbb{Z}$, $i=1, 2, \cdots, n$. Note $k_{i+1}=k_i+1$ and $k'_{i+1}=k_i'+1$ for $i\in\{1,2, \cdots, n-1\}$.

Then, for any $i \in \left\{1, 2, \cdots, n\right\}$ and $t \geq 0$,
\[
\left|\wrh(r+k_ih, t)-\wrhpb(r+k'_ih',t)\right|\leq \omega_t(|\beta|, |k_ih-k'_ih'|, h),
\]
where $\omega_t$ is a modulus of continuity that depends on $t$.
\end{lem}

\begin{proof}
For $t=0$, the result follows from the fact that $u_0 \in W^{1, \infty}(\Omega) \cap C(\overline{\Omega})$.

Fix $t>0$. Let $\Delta t>0$. Consider the resolvents associated with $\wrh$ and $\wrhpb$, denoted respectively by
\[
\tilde{J}_{\Delta t, h} := \left(I-\Delta t \tHrh\right)^{-1}, \qquad \tilde{J}_{\Delta t, h'}^{\sbe} := \left(I-\Delta t\tHrhpb\right)^{-1},
\]
where $\tHrh$ and $\tHrhpb$ are defined in \eqref{eqn:defHtildeh} and \eqref{eqn:defHtildehbeta}, respectively.
Recall that by Proposition \ref{prop:propdis},
\[
\wrh(\cdot, t) = \lim_{\Delta t \to 0}\left(I-\Delta t \tHrh\right)^{-\left[\frac{t}{\Delta t}\right]}w_0, \qquad \wrhpb(\cdot, t) = \lim_{\Delta t \to 0}\left(I-\Delta t\tHrhpb\right)^{-\left[\frac{t}{\Delta t}\right]}w_{\zero}^{\sbe}.
\]

Let $\twrh$ solve $\twrh = \left(I-\Delta t \tHrh \right)^{-1}w_0$, and let $\twrhpb$ solve $\twrhpb = \left(I-\Delta t\tHrhpb\right)^{-1}w_{\zero}^{\sbe}$.

Since $u_0 \in W^{1, \infty}(\Omega) \cap C(\overline{\Omega})$, we have $ w_0 \in \mathrm{Lip}([a, b])$ and $w_{\zero}^{\sbe} \in \mathrm{Lip}([a_{\sbe}, b_{\sbe}])$ with the same Lipschitz constant $\left\| Du_0\right\|_{L^\infty(\Omega)}$. Then for any $i \in \left\{1, 2, \cdots, n\right\}$,
\begin{equation} \label{eqn:Hhw0bound}
    \begin{aligned}
\left|\tHrh\left(r+k_ih , w_0(r+k_ih),w_0\right) \right|=&\left|\mathbbm{1}_{\{i \neq n\}}  \tilde{\Phi}^+(r+k_ih) \left( e^{\frac{w_0(r+k_ih+ h)-w_0(r+k_ih)}{h}} - 1\right)\right.\\ &+ \left. \mathbbm{1}_{\{i \neq 1\}}\tilde{\Phi}^-(r+k_ih)\left( e^{  \frac{w_0(r+k_ih- h)-w_0(r+k_ih)}{h} } - 1\right)\right|\\
\leq C_2&
    \end{aligned}
\end{equation}
for some constant $C_2=C_2(\left\| Du_0\right\|_{L^\infty(\Omega)}, \Omega, \Phi^+, \Phi^-)>0$. From Proposition \ref{prop:uhproperty}, for any $i \in \left\{1, 2, \cdots, n\right\}$,
\begin{equation}\label{eqn:unibdHhwh}
\begin{aligned}
\left|\tHrh\left(r+k_ih , \twrh(r+k_ih),\twrh\right) \right|&\leq
\left\| \tHrh \twrh\right\|_{L^\infty \left(\torh\right)}
\leq \left\| \tHrh w_0\right\|_{L^\infty \left(\torh\right)}
\leq C_2,
\end{aligned}
\end{equation}
where the last inequality comes from \eqref{eqn:Hhw0bound}. Similarly,
\begin{equation}
\left\| \tHrhpb  \twrhpb\right\|_{L^\infty\left(\tobrhp\right)} \leq \left\|\tHrhpb w_{\zero}^{\sbe}\right\|_{L^\infty\left(\tobrhp\right)} \leq C_2.
\end{equation}
From \eqref{eqn:unibdHhwh}  and \eqref{eq:lower}, there exists a constant $C_3=C_3(\left\| Du_0\right\|_{L^\infty(\Omega)}, \Omega, \Phi^+, \Phi^-)>0$ such that for any $i \in \left\{1, 2, \cdots, n-1\right\}$,
\[
\frac{\twrh \left(r+k_ih+ h, t\right)-\twrh \left(r+k_ih, t\right)}{h} \leq C_3,\]
and for any $i \in \left\{2, \cdots,n \right\}$
\[\frac{\twrh \left(r+k_ih- h, t\right)-\twrh \left(r+k_ih, t\right)}{h} \leq C_3,
\]
which implies for any $i \in \left\{1, 2, \cdots,n-1 \right\}$,
\begin{equation}
    \left|\frac{\twrh \left(r+k_i h+h, t\right)-\twrh \left(r+k_ih, t\right)}{h} \right|\leq C_3,
\end{equation}

and
\begin{equation}\label{eqn:unibdewh}
    \left|e^{ \frac{\twrh \left(r+k_ih+ h, t\right)-\twrh \left(r+k_ih, t\right)}{h} } -1 \right|\leq C_4,
\end{equation}
for some constant $C_4=C_4\left(\left\| Du_0\right\|_{L^\infty(\Omega)}, \Omega, \Phi^+, \Phi^-\right)>0$.

Similarly for $\twrhpb$, it holds that for any $i \in \left\{1, 2, \cdots,n-1 \right\}$,
\begin{equation}\label{eqn:ubdwhbdif}
\left|\frac{\twrhpb\left(r+k_ih'+ h', t\right)-\twrhpb\left(r+k_ih', t\right)}{h'} \right|\leq C_3.
\end{equation}

Now, choose ${i^\ast} \in \left\{1,2, \cdots, n\right\}$ such that
\begin{equation}\label{eqn:defistar}
\twrh(r+k_{i^\ast} h, t)-\twrhpb(r+k'_{i^\ast}h',t)=\max\left\{\twrh(r+k_ih, t)-\twrhpb(r+k'_ih',t):i=1, 2, \cdots, n\right\}.
\end{equation}
We assume $i^\ast \neq 1, n$; the boundary cases are analogous.

For $i^* \in \{2, \dots, n-1\}$,
\begin{equation}\label{eqn:whwhbdiff}
    \begin{aligned}
 \twrh(r+k_{i^\ast} h, t)-\twrhpb(r+k'_{i^\ast}h',t)=&\Delta t \tilde{\Phi}^+(r+k_{i^\ast}h) \left( e^{\frac{\twrh \left(r+k_{i^\ast}h+ h, t\right)-\twrh \left(r+k_{i^\ast}h, t\right)}{h}} - 1\right)\\&-\Delta t \tilde{\Phi}^+_{\sbe}\left(r+k'_{i^\ast}h'\right) \left( e^{\frac{\twrhpb \left(r+k'_{i^\ast}h'+ h', t\right)-\twrhpb\left(r+k'_{i^\ast}h', t\right)}{h'}} - 1\right)\\ &+\Delta t \tilde{\Phi}^-(r+k_{i^\ast}h)\left( e^{  \frac{\twrh \left(r+k_{i^\ast}h- h, t\right)-\twrh \left(r+k_{i^\ast}h, t\right)}{h} } - 1\right) \\
 &-\Delta t \tilde{\Phi}^-_{\sbe}(r+k'_{i^\ast}h')\left( e^{  \frac{\twrhpb \left(r+k'_{i^\ast}h'- h', t\right)-\twrhpb\left(r+k'_{i^\ast}h', t\right)}{h'} } - 1\right) \\
 &+w_0(r+k_{i^\ast} h)-w_{\zero}^{\sbe}(r+k'_{i^\ast} h')\\
=& \mathrm{I+II}+w_0(r+k_{i^\ast} h)-w_{\zero}^{\sbe}(r+k'_{i^\ast} h'),
\end{aligned}
\end{equation}
where
\[
    \begin{aligned}
\mathrm{I}:=& \Delta t\tilde{\Phi}^+(r+k_{i^\ast}h) \left( e^{\frac{\twrh \left(r+k_{i^\ast}h+ h, t\right)-\twrh \left(r+k_{i^\ast}h, t\right)}{h}} - 1\right)\\&-\Delta t\tilde{\Phi}^+_{\sbe}\left(r+k'_{i^\ast}h'\right) \left( e^{\frac{\twrhpb \left(r+k'_{i^\ast}h'+ h', t\right)-\twrhpb\left(r+k'_{i^\ast}h', t\right)}{h'}} - 1\right),
    \end{aligned}
\]

\[
    \begin{aligned}
\mathrm{II}:=&\Delta t\tilde{\Phi}^-(r+k_{i^\ast}h)\left( e^{  \frac{\twrh \left(r+k_{i^\ast}h- h, t\right)-\twrh \left(r+k_{i^\ast}h, t\right)}{h} } - 1\right) \\
 &-\Delta t\tilde{\Phi}^-_{\sbe}(r+k'_{i^\ast}h')\left( e^{  \frac{\twrhpb \left(r+k'_{i^\ast}h'- h', t\right)-\twrhpb\left(r+k'_{i^\ast}h', t\right)}{h'} } - 1\right).
    \end{aligned}
\]

We estimate $\mathrm{I}$ first. Write
\begin{equation}\label{eqn:I}
    \mathrm{I} =A+B,
\end{equation}
where
\[
A:=\Delta t\left(\tilde{\Phi}^+(r+k_{i^\ast}h)-\tilde{\Phi}^+_{\sbe}\left(r+k'_{i^\ast}h'\right)\right)\left( e^{\frac{\twrh \left(r+k_{i^\ast}h+ h, t\right)-\twrh \left(r+k_{i^\ast}h, t\right)}{h}} - 1\right)
\]

and
\[
B:=\Delta t\tilde{\Phi}^+_{\sbe}\left(r+k'_{i^\ast}h'\right)\left( e^{\frac{\twrh \left(r+k_{i^\ast}h+ h, t\right)-\twrh \left(r+k_{i^\ast}h, t\right)}{h}} - e^{\frac{\twrhpb \left(r+k'_{i^\ast}h'+ h', t\right)-\twrhpb\left(r+k'_{i^\ast}h', t\right)}{h'}} \right).
\]

From \eqref{eqn:unibdewh}, we know

\[
|A|\leq\Delta t C_4 \omega^+\left(\max_{i\in\{1, 2, \cdots, n\}} |k_ih-k'_ih'|, \left|\beta\right|\right),
\]
where $\omega^+$ is a modulus of continuity of $\Phi^+$ and $C_4$ comes from \eqref{eqn:unibdewh}.

For $B$, by the mean value theorem, there exists a constant $\xi \leq C_3$ such that
\begin{equation}\label{eqn:B}
\begin{aligned}
B=&\Delta t\tilde{\Phi}^+_{\sbe}\left(r+k'_{i^\ast}h'\right) e^\xi \left(\frac{\twrh \left(r+k_{i^\ast}h+ h, t\right)-\twrh \left(r+k_{i^\ast}h, t\right)}{h} \right.\\
&\left.\qquad \qquad \qquad \qquad \qquad \qquad \qquad\qquad\qquad-\frac{\twrhpb \left(r+k'_{i^\ast}h'+ h', t\right)-\twrhpb\left(r+k'_{i^\ast}h', t\right)}{h'}\right), \\
\end{aligned}
\end{equation}
where
\[
\begin{aligned}
&\frac{\twrh \left(r+k_{i^\ast}h+ h, t\right)-\twrh \left(r+k_{i^\ast}h, t\right)}{h}-\frac{\twrhpb \left(r+k'_{i^\ast}h'+ h', t\right)-\twrhpb\left(r+k'_{i^\ast}h', t\right)}{h'}\\
=&\frac{h'\left(\twrh \left(r+k_{i^\ast}h+ h, t\right)-\twrhpb \left(r+k'_{i^\ast}h'+ h', t\right)-\left(\twrh \left(r+k_{i^\ast}h, t\right)-\twrhpb\left(r+k'_{i^\ast}h', t\right)\right)\right)}{hh'}\\
&+\frac{\left(h'-h\right)\left(\twrhpb \left(r+k'_{i^\ast}h'+ h', t\right)-\twrhpb\left(r+k'_{i^\ast}h', t\right)\right)}{hh'}\\
\leq & \frac{\left(h'-h\right)\left(\twrhpb \left(r+k'_{i^\ast}h'+ h', t\right)-\twrhpb\left(r+k'_{i^\ast}h', t\right)\right)}{hh'}\\
\leq &Ch,
\end{aligned}
\]
for some constant $C>0$. The last two inequality come from \eqref{eqn:defistar}, \eqref{eqn:ubdwhbdif} and the assumption that $\left|h'-h\right|\leq O(h^2)$. Hence we have $B \leq \Delta t Ch$ for some constant $C=C\left(\left\| Du_0\right\|_{L^\infty(\Omega)}, \Omega, \Phi^+, \Phi^-\right)>0$. Therefore, from \eqref{eqn:I}, we have
\[
\mathrm{I} \leq \Delta t \omega \left(\max_{i\in\{1, 2, \cdots, n\}} |k_ih-k'_ih'|,\left|\beta\right|, h\right),
\]
for some modulus of continuity $\omega$. By the same argument, we can show
\[
\mathrm{II}\leq \Delta t \omega \left( \max_{i\in\{1, 2, \cdots, n\}}|k_ih-k'_ih'|,\left|\beta\right|, h\right)
\]
for some modulus of continuity $\omega$.

Therefore, from \eqref{eqn:whwhbdiff}, for any $i \in \left\{1, 2, \cdots, n\right\}$,
\begin{equation}
\begin{aligned}
&\max_{i\in\{1, 2, \cdots, n\}} \left(\twrh(r+k_ih, t)-\twrhpb(r+k'_ih',t)\right) \\
\leq &  \max_{i\in\{1, 2, \cdots, n\}}\left(w_0(r+k_ih)-w_{\zero}^{\sbe}(r+k'_ih')\right)+\Delta t \omega \left( \max_{i\in\{1, 2, \cdots, n\}}|k_ih-k'_ih'|,\left|\beta\right|, h\right)
\end{aligned}
\end{equation}
for some modulus of continuity $\omega$.
Similarly, we can also show
\begin{equation}
\begin{aligned}
&\max_{i\in\{1, 2, \cdots, n\}}\left(\twrhpb(r+k'_ih',t)-\twrh(r+k_ih, t)\right)\\  \leq &\max_{i\in\{1, 2, \cdots, n\}}\left(w_{\zero}^{\sbe}(r+k_ih)-w_0(r+k'_ih')\right)+\Delta t \omega \left( \max_{i\in\{1, 2, \cdots, n\}}|k'_ih'-k_ih|,\left|\beta\right|, h\right).
\end{aligned}
\end{equation}
Therefore,
\begin{equation}
\begin{aligned}
    &\max_{i\in\{1, 2, \cdots, n\}}\left|\twrh(r+k_ih, t)-\twrhpb(r+k'_ih',t)\right|\\
    \leq &  \max_{{i\in\{1, 2, \cdots, n\}}}\left|w_0(r+k'_ih')-w_{\zero}^{\sbe}(r+k_ih)\right| +\Delta t \omega \left( \max_{i\in\{1, 2, \cdots, n\}}|k'_ih'-k_ih|,\left|\beta\right|, h\right).
\end{aligned}
\end{equation}
For any $\Delta t>0$, by iterating the resolvent operators, we obtain
\begin{equation}
\begin{aligned}
    &\max_{i\in\{1, 2, \cdots, n\}}\left|\left(I-\Delta t \tilde{H}_{h'}\right)^{-\left[\frac{t}{\Delta t}\right]}w_0-\left(I-\Delta t \tHrhpb \right)^{-\left[\frac{t}{\Delta t}\right]}w_{\zero}^{\sbe}\right|\\
    \leq &  \max_{{i\in\{1, 2, \cdots, n\}}}\left|w_0(r+k'_ih')-w_{\zero}^{\sbe}(r+k_ih)\right| +\left[\frac{t}{\Delta t}\right]\Delta t \omega \left( \max_{i\in\{1, 2, \cdots, n\}}|k'_ih'-k_ih|,\left|\beta\right|, h\right)\\
    \leq &  \max_{{i\in\{1, 2, \cdots, n\}}}\left|w_0(r+k'_ih')-w_{\zero}^{\sbe}(r+k_ih)\right| +t \omega \left( \max_{i\in\{1, 2, \cdots, n\}}|k'_ih'-k_ih|,\left|\beta\right|, h\right),
\end{aligned}
\end{equation}
and by sending $\Delta t \to 0$, we have
\begin{equation}
\begin{aligned}
    &\max_{i\in\{1, 2, \cdots, n\}}\left|\wrh(r+k_ih, t)-\wrhpb(r+k'_ih',t)\right| \\
    \leq & \max_{{i\in\{1, 2, \cdots, n\}}}\left|w_0(r+k'_ih')-w_{\zero}^{\sbe}(r+k_ih)\right| +t \omega \left( \max_{i\in\{1, 2, \cdots, n\}}|k'_ih'-k_ih|,\left|\beta\right|,h\right),\\
    \leq & \omega_t\left(\max_{i\in\{1, 2, \cdots, n\}}|k'_ih'-k_ih|, |\beta| h\right)
\end{aligned}
\end{equation}
for some modulus of continuity $\omega_t$, where the last inequality follows from \eqref{eqn:w0dif}.
\end{proof}

We are now ready to prove that the discrete solution converges to the continuous limiting solution $w$, even when the discrete grid points lie outside the line segment $S_0$ passing through $\vx_0$ in the direction of $\vec{\nu}$.


\begin{thm}\label{thm:convergelbeta}
Assume $\Omega$ is convex. Let $u_0 \in C(\overline{\Omega})$. Fix $t \geq 0$ and $\vx_0 \in \overline{\Omega}$. Let $h > 0$, $\rh \in \mathbb{R}$ with $|\rh|<h$, and suppose $\beh \in \mathbb{R}$ satisfies $\max\left\{\omega_0(|\beh|), |\beh|\right\}\leq h$, where $\omega_0$ is the modulus of continuity defined in \eqref{eqn:endpointsdif}.

Let $w:[a,b] \to \mathbb{R}$ be the solution to \eqref{HJE_NM} with initial condition $w_0(\alpha) = u_0(\vx_0 + \alpha \vec{\nu})$ for any $\alpha \in [a, b]$, and let $\wrhhbh:\tobhrhh\to \mathbb{R}$ be the solution to \eqref{eqn:1ddistbeta} with initial condition
\[
w_{\zero}^{\sbe_{\V}}\left(\rh + \alhi\right) = u_0\left(\vx_0 + \beh \vec{\nu}_\perp + (\rh + \alhi) \vec{\nu}\right)
\quad \text{for any } \rh + \alhi \in \tobhrhh.
\]
Assume that $\rh + k_{\V} h \in \tobhrhh$ for some $k_{\V} \in \mathbb{Z}$ such that $\lim_{h \to 0^+}  k_{\V} h = \alpha_0$ for some $\alpha_0 \in[a, b]$. Then,
\[
\lim_{h\to 0^+}\wrhhbh(\rh+ k_{\V} h, t)=w(\alpha_0,t).
\]
\end{thm}

\begin{proof}
We establish the result in two steps.

Step 1: We first prove the result under the assumption that $u_0 \in W^{1,\infty} (\Omega)\cap C(\overline{\Omega})$, and later extend the result to the case $u_0 \in  C(\overline{\Omega})$.

Since $\omega_0\left(|\beh|\right) \leq h$, it follows from \eqref{eqn:endpointsdif} that
\begin{equation}\label{eqn:endpointsdifLip}
|a_{\sbe_{\V}} - a| \leq h, \quad |b_{\sbe_{\V}} - b| \leq h.
\end{equation}

Let $\wrhh$ be the solution to \eqref{eqn:1dv_h} with initial condition $w_0(\rh+\aldh):=u_0(\vec{x}_0 +(\rh+\aldh) \vec{\nu})$ for any $\rh+\aldh \in \torhh$. Consider the domains $\torhh$ and $\tobhrhh$ associated with $\wrhh$ and $\wrhhbh$, respectively.

There are several cases to be considered.
\begin{itemize}
    \item [(i)] Suppose for all sufficiently small $h$,  $\left|\torhh\right|=1$, that is, $\torhh=\left\{0\right\} $, $a=b=0$, and $\alpha_0=0$. Then
    \begin{equation}\label{eqn:whw0u0}
    \wrhh(0,t)=w_0(0)=u_0\left(\vx_0\right).
    \end{equation}
    From \eqref{eqn:endpointsdifLip}, for any $\rh + \aldh \in \tobhrhh$, we have
    \[
    \vx_0 + \beh \vec{\nu}_\perp + (\rh + \aldh) \vec{\nu} \in B(\vx_0, Ch),
    \]
    for some constant $C=C\left(|\vec{\nu}|,|\vec{\nu}_{\perp}|\right)>0$.

    Hence, for any $\rh + \aldh \in \tobhrhh$,
    \begin{equation}\label{eqn:u0u0betadif}
        \left|u_0(\vx_0)-w_{\zero}^{\sbe_{\V}}\left(\rh + \aldh\right)\right|\leq C\left\|Du_0\right\|_{L^\infty\left(\Omega\right)}h.
    \end{equation}
    From Proposition \ref{prop:propdis}, we know
    \begin{equation}\label{eqn:whbetabdminmax}
    \min_{\rh + \aldh \in \tobhrhh}w_{\zero}^{\sbe_{\V}}(\rh+\aldh) \leq \wrhhbh \left(\rh+k_{\V}h, t\right) \leq \max_{\rh + \aldh \in \tobhrhh}w_{\zero}^{\sbe_{\V}}(\rh+\aldh).
    \end{equation}
    Combining \eqref{eqn:whw0u0}, \eqref{eqn:u0u0betadif} and \eqref{eqn:whbetabdminmax}, we have
\begin{equation}
        \left|\wrhh(0,t)-\wrhhbh \left(\rh+k_{\V}h, t\right) \right| \leq C\left\|Du_0\right\|_{L^\infty\left(\Omega\right)}h,
\end{equation}
    for some constant $C=C\left(|\vec{\nu}|,|\vec{\nu}_{\perp}|\right)>0$.
    \item [(ii)] Suppose there exists some $h>0$ such that $\left|\torhh\right|>1$. Without loss of generality, we may assume that the cardinalities satisfy
$\left|\torhh\right|>10,\left|\tobhrhh\right|>10$.

\end{itemize}

    \begin{itemize}
        \item[(a)] If $\left|\torhh\right|=\left|\tobhrhh\right|=n$ for some $n > 10$, we label the elements in $\torhh$ as
        \begin{equation} \label{eqn:labelab}
        a\leq \rh+k_1 h < \cdots <\rh+k_i h< \cdots < \rh+k_n h \leq b, \quad i=1, 2, \cdots, n,
        \end{equation}
        and elements in  $\tobhrhh$ by
        \begin{equation}\label{eqn:labelabbeta}
        a_{\sbe_{\V}} \leq \rh+k'_1 h < \cdots <\rh+k'_i h< \cdots < \rh+k'_n h \leq b_{\sbe_{\V}}, \quad i=1, 2, \cdots, n.
        \end{equation}

    Since $|a - a_{\sbe_{\V}}| \leq h$, we have
\[
\left| (\rh + k_i h) - (\rh + k'_i h) \right| \leq 2h \quad \text{for all } i = 1, 2, \ldots, n.
\]
Then, by Lemma \ref{lem:endisbeta}, it follows that
        \begin{equation}\label{eqn:estnsame}
    \left|\wrhh\left(\rh+k_i h,t\right)-\wrhhbh\left(\rh+k^\prime_ih,t\right)\right| \leq \omega \left(h\right),
        \end{equation}
    for some modulus of continuity $\omega$.
Since $\rh+k_{\V}h \in \tobhrhh$, $k_{\V}=k_{i(h)}'$ for some $i(h)\in \{1, 2, \cdots, n\}$.

    Therefore,
    \begin{equation}
    \begin{aligned}
    \left|\wrhhbh\left(\rh+k_{\V}h,t\right)-w(\alpha_0,t)\right| \leq &\left|\wrhh\left(\rh+k_{i(h)} h,t\right)-w(\alpha_0,t)\right|
    \\& \quad + \left|\wrhh\left(\rh+k_{i(h)} h,t\right)-\wrhhbh\left(\rh+k^\prime_{i(h)}h,t\right)\right|\\
    \leq & \tilde{\omega}(h),
    \end{aligned}
    \end{equation}
    where $\tilde{\omega}$ is some modulus of continuity and the last inequality comes from Corollary \ref{cor:1dconverge}, Lemma \ref{lem:endisbeta} and \eqref{eqn:estnsame}.

    \item[(b)] If $\left|\torhh\right|>\left|\tobhrhh\right|$, choose $\hat{h}>h$ such that $\left|\tilde{\Omega}_{\sr_{\V},\hat{\V}}\right|=\left|\tobhrhh\right|=n$. We also label the elements in $\tilde{\Omega}_{\sr_{\V},\hat{\V}}$ and $\tobhrhh$ as in \eqref{eqn:labelab} and \eqref{eqn:labelabbeta}, respectively. We claim that
    \begin{equation}\label{eqn:hhathdiff}
     \left|\hat{h}-h\right|\leq O(h^2).
    \end{equation}
    First, by the definitions of $\tilde{\Omega}_{\sr_{\V},\hat{\V}}$ and $\tobhrhh$, we have
    \[
    (n-1)h\leq \left|a_{\sbe}-b_{\sbe}\right|\leq (n+1)h, \qquad (n-1) \hat{h}\leq \left|a-b\right|\leq (n+1) \hat{h},
    \]
    which implies
    \begin{equation}\label{eqn:hhhatest}
    \frac{\left|a_{\sbe}-b_{\sbe}\right|}{n+1} \leq h \leq \frac{\left|a_{\sbe}-b_{\sbe}\right|}{n-1}, \qquad  \frac{\left|a-b\right|}{n+1} \leq \hat{h}\leq \frac{\left|a-b\right|}{n-1}.
    \end{equation}
    Hence,
    \begin{equation}
    \left|\hat{h}-h\right|\leq \max \left\{A, B\right\},
    \end{equation}
    where \[
A := \left|\frac{|a - b|}{n - 1} - \frac{|a_{\sbe_{\V}} - b_{\sbe_{\V}}|}{n + 1}\right|, \quad
B := \left|\frac{|a_{\sbe_{\V}} - b_{\sbe_{\V}}|}{n - 1} - \frac{|a - b|}{n + 1}\right|.
\]
From \eqref{eqn:hhhatest}, it follows that
\begin{equation}\label{eqn:nhestimate}
   \frac{1}{n+1} < \frac{1}{n-1} \leq C h,
\end{equation}
for some constant $C>0$.
Now we estimate
\[
\begin{aligned}
A&\leq \frac{\left||a-b|-|a_{\sbe}-b_{\sbe}|\right|}{n-1}+\frac{2|a_{\sbe}-b_{\sbe}|}{(n-1)(n+1)}\\
& \leq    \frac{2h}{n-1}+\frac{2h}{n-1}\\
& \leq Ch^2,
\end{aligned}
\]
where the second inequality follows from \eqref{eqn:endpointsdifLip} and \eqref{eqn:hhhatest}, and the last inequality uses \eqref{eqn:nhestimate}. Similarly, $B \leq Ch^2$ for some constant $C>0$. This proves that $|\hat{h}-h| \leq O(h^2)$.

Since $|a-a_{\sbe_{\V}}|\leq h$, it follows that
\[
|\rh + k_i \hat{h} - (\rh + k'_i h)| \leq C h, \quad \text{for all } i = 1, 2, \ldots, n.
\]
Then, by Lemma \ref{lem:endisbeta}, we obtain
\begin{equation}\label{eqn:hhath}
\left|\wrhh(\rh+k_i\hat{h}, t)-\wrhhbh(\rh+k'_ih,t)\right|\leq \omega(h), \quad \text{for all } i = 1, 2, \ldots, n.
\end{equation}
Since $\rh+k_{\V}h \in \tobhrhh$, $k_{\V}=k_{i(h)}'$ for some $i(h)\in \{1, 2, \cdots, n\}$.
Hence,
    \[
    \begin{aligned}
    &\left|\wrhhbh\left(\rh+k_{\V}h,t\right)-w(\alpha_0,t)\right|\\
    \leq &\left|\wrhhath\left(\rh+k_{i(h)} \hat{h},t\right)-w(\alpha_0,t)\right| + \left|\wrhhath \left(\rh+k_{i(h)} \hat{h},t\right)-\wrhhbh\left(\rh+k^\prime_{i(h)}h,t\right)\right|\\
    \leq & \tilde{w} (h),
    \end{aligned}
    \]
    where the last inequality comes from Corollary \ref{cor:1dconverge}, Lemma \ref{lem:endisbeta}, \eqref{eqn:hhathdiff} and \eqref{eqn:hhath}. Therefore, the conclusion holds.
    \item[(c)] If $\left|\torhh\right|<\left|\tobhrhh\right|$, choose $\hat{h}<h$ such that $\left|\tilde{\Omega}_{\sr_{\V},\hat{\V}}\right|=\left|\tobhrhh\right|=n$. The proof is similar to that in case (b).
    \end{itemize}
In each case, we showed that
\[
  \left|\wrhhbh\left(\rh+k_{\V}h,t\right)-w(\alpha_0,t)\right| \leq \tilde{\omega}(h)
\]
for some modulus of continuity $\tilde{\omega}$. Therefore, we have established the conclusion for $u_0 \in W^{1,\infty}(\Omega) \cap C(\overline{\Omega})$.

Step 2: Now suppose $u_0 \in C(\overline{\Omega})$. Then there exists a sequence $\left\{u_{\zero, \sm}\right\}_{\sm\in \mathbb{N}} \subset W^{1,\infty}(\Omega)\cap C(\overline{\Omega})$ such that $|u_{\zero, \sm}(x)-u_0(x)|\leq \frac{1}{m}$ for all $x \in \overline{\Omega}$.

For any $\alpha \in [a, b]$, define $w_{\zero, \sm}(\alpha):=u_{\zero, \sm}(\vec{x}_0 +\alpha\vec{\nu})$. Let $w_{\sr_{\V}, \V, \sm}:\torhh \to \mathbb{R}$ be the solution to \eqref{eqn:1dv_h} with initial data $w_{\zero, \sm}\big|_{\torhh}$, and $w_{\sm}:[a, b] \to \mathbb{R}$ be the solution to \eqref{HJE_NM} with initial data $w_{\zero, \sm}$.

For any $\alpha' \in [a_{\sbe_{\V}}, b_{\sbe_{\V}}]$, define $w_{\zero, \sm}^{\sbe_{\V}}(\alpha'):=u_{\zero, \sm}\left(\vxv+\alpha'\vec{\nu}+\beh\vec{\nu}_\perp\right)$. Let $w^{\sbe_{\V}}_{\sr_{\V}, \V, \sm}:\tobhrhh \to \mathbb{R}$ be the solution to \eqref{eqn:1dv_h} with initial data $w_{\zero, \sm}^{\sbe_{\V}}\big|_{\tobhrhh}$.

Since $|u_{\zero, \sm}(x)-u_0(x)|\leq \frac{1}{m}$ for all $x \in \overline{\Omega}$, it follows that
\[
\left|w_{\zero, \sm}(\alpha) - w_0(\alpha)\right| \leq \frac{1}{m} \quad \text{for all } \alpha \in [a, b],
\]
and
\[
\left|w_{\zero, \sm}^{\sbe_{\V}}(\alpha') - w_{\zero}^{\sbe_{\V}}(\alpha')\right| \leq \frac{1}{m} \quad \text{for all } \alpha' \in [a_{\sbe_{\V}}, b_{\sbe_{\V}}].
\]

Then, by Proposition \ref{prop:propdis}, it follows that for any $h>0$,
\begin{equation}\label{eqn:difdistbetalip}
    \left| w_{\rh, \V,\sm}^{\sbe_{\V}}(\rh+k_{\V}h, t)-\wrhhbh(\rh+k_{\V}h, t) \right| \leq \left\|w_{\zero, \sm}^{\sbe}-w_{\zero}^{\sbe_{\V}}\right\|_{L^\infty\left(\tobhrhh\right)} \leq \frac{1}{m}.
\end{equation}
From Corollary \ref{cor:1dcontcontract}, for any $\alpha \in [a, b]$,
\begin{equation}\label{eqn:contwlip}
    \left|w_{\sm}(\alpha,t)-w(\alpha,t)\right| \leq \left\|w_{\zero,\sm}-w_0\right\|_{L^\infty([a, b])}\leq\frac{1}{m}.
\end{equation}
Combining \eqref{eqn:difdistbetalip} and \eqref{eqn:contwlip}, we obtain
\[
\begin{aligned}
&\limsup_{h\to 0^+} \left| \wrhhbh(\rh+k_{\V}h, t)-w(\alpha_0, t)\right|\\
\leq & \limsup_{h \to 0^+} \left|w_{\sr_{\V}, \V,\sm}^{\sbe_{\V}}(\rh+k_{\V}h, t)-\wrhhbh(\rh+k_{\V}h, t)\right|
\\ &\qquad \qquad \qquad \quad  +\limsup_{h \to 0^+} \left|w_{\sr_{\V}, \V,\sm}^{\sbe_{\V}}(\rh+k_{\V}h, t)-w_{\sm}(\alpha_0,t)\right| +\left|w_{\sm}(\alpha_0, t)-w(\alpha_0, t)\right|\\
\leq &\frac{2}{m},
\end{aligned}
\]
where the last inequality comes from the conclusion already established for $w_{\sr_{\V}, \V,\sm}^{\sbe_{\V}}$.
Sending $m \to \infty$ completes the proof.
\end{proof}

\begin{rem} \label{rem:convex}
The estimate \eqref{eqn:endpointsdifLip} relies on the assumption that $\Omega$ is convex.
In general, for a nonconvex domain, this property may fail. Nonconvexity can cause a sudden change in the length of the line segment $S_{\sbe}$, even for arbitrarily small values of $\beta$.
Consequently, the convergence result may no longer hold in this case.
\end{rem}

Building on the previous results, we now establish the large deviation principle for the rescaled process $X^{\V}$ on $\overline{\Omega}$.

\begin{thm}\label{thm:LDPOmega}
Assume $\Omega$ is convex. Let $\vec{x}_0 \in \overline{\Omega}$, and for each small $h > 0$, define
\[
\vxzh = \vec{x}_0 + \beh \vec{\nu}_\perp + (\rh + \alzh)\vec{\nu} \in \overline{\Omega},
\]
where $\beh \in \mathbb{R}$ satisfies $\max\left\{\omega_0(|\beh|), |\beh|\right\}\leq h$ with modulus of continuity $\omega_0$ as in~\eqref{eqn:endpointsdif}, $\rh \in \mathbb{R}$ with $|\rh| < h$, and $\alzh = \kzh h$ for some $\kzh \in \mathbb{Z}$. Suppose that
\[
\lim_{h \to 0^+} \rh = 0, \quad \lim_{h \to 0^+} \alzh = 0.
\]
Then, the rescaled chemical reaction process $\cv(t)$ with initial condition $\cv(0) = \vxzh$ satisfies a large deviation principle at each time $t$, with the good rate function $\tilde{I}$ given by
\begin{equation}\label{eqn:deftildeI}
\tilde{I}\left(\vec{y}; \vec{x}_0, t\right):=\begin{cases}
			I\left(\frac{(\vec{y}-\vx_0)\cdot \vec{\nu}}{|\vec{\nu}|^2};0,t\right), & \quad \text{if $\, \vec{y} \, $ in $S_0$},\\
            \infty, & \quad \text{otherwise},
		 \end{cases}
\end{equation}
where $S_0=\left\{\vec{x}_0 + \alpha \vec{\nu}: a\leq \alpha \leq b\right\} $ is the maximal closed line segment through $\vx_0$, and $I$ is defined in \eqref{eqn:defIratef}.
\end{thm}

\begin{proof}
The process $\cv(t)$ evolves along the discrete mesh
\[
\begin{aligned}
\hat{\Omega}_{\beh, \sr_{\V}, \V}
&=\left\{\vec{x}_0 +\beta _h\vec{\nu}_\perp+(\rh+\alpha) \vec{\nu} \in \overline{\Omega}: \alpha =kh, k  \in \mathbb{Z}\right\}\\
&=\left\{\vec{x}_0 +\beh \vec{\nu}_\perp+(\rh+\alpha) \vec{\nu}: \alpha =kh, k\in \mathbb{Z}, k_{\sa_{\sbe_{\V}}, \sr_{\V}, \V}\leq k\leq k_{\scb_{\sbe_{\V}}, \sr_{\V}, \V} \right\}
\end{aligned}
\]
for some integers $k_{\sa_{\sbe_{\V}}, \sr_{\V}, \V} \leq k_{\scb_{\sbe_{\V}}, \sr_{\V}, \V}$. Hence $\cv(t)=\vx_0+\beh \vec{\nu}_\perp+(\rh+\aluh(t))\vec{\nu}$ for some $\alpha^{\V}(t) = k^{\V}(t)h$ where $k^{\V}(t) \in \mathbb{Z}$ and $\alpha^{\V}(0)=\alzh=\kzh h$.

Similar to the proof of Theorem \ref{thm:1dLDP}, we now establish the large deviation principle for the process $\cv(t)$ at a single time through Varadhan's inverse lemma \cite{Bryc_1990}. Since the exponential tightness of $\cv(t)$ for each $t\in[0,T]$ is already verified in \eqref{eq:exp_tight}, we only need to show the convergence of
the discrete Varadhan's nonlinear semigroup to the variational representation of the continuous viscosity solution.

Combining Varadhan's lemma \cite{Varadhan_1966} and Brysc's theorem \cite{Bryc_1990}, we know $\cv$ satisfies the large deviation principle with the good rate function $\tilde{I}(y)$ if and only if for any bounded continuous function $u_0$,

\begin{equation}\label{eqn:LDPbetato0}
\lim_{h \to 0^+}h \log \mathbb{E}^{\vxzh}\left(e^\frac{u_0\left(\cv(t)\right)}{h}\right) =\sup_{\vec{y}\in \overline{\Omega}}\left\{u_0(\vec{y})-\tilde{I}\left(\vec{y}; \vx_0, t\right) \right\}.
\end{equation}
Hence, it suffices to prove \eqref{eqn:LDPbetato0}.

From the WKB reformulation for the backward equation, we know
    \[
    \uh(\vxzh, t) = h \log \mathbb{E}^{\vxzh}\left(e^\frac{u_0\left(\cv(t)\right)}{h}\right),
    \]
    and hence by the definition of $\wrhhbh$, we have
    \[
    \wrhhbh(\rh+\alzh, t)=h\log \mathbb{E}^{\rh+\alzh}\left(e^\frac{w_{\zero}^{\sbe_{\V}}\left(\rh+\aluh(t)\right)}{h}\right).
    \]
 Since $\lim_{h \to 0^+}\rh = 0$ and $\lim_{h\to 0^+}\alzh = 0$, by Theorem \ref{thm:convergelbeta},
 \[
 \lim_{h\to 0^+} \wrhhbh (\rh+\alzh, t)=w(0,t),
 \]
 that is,
\begin{equation}
    \begin{aligned}
\lim_{h \to 0^+}h \log \mathbb{E}^{\vxzh}\left(e^\frac{u_0\left(\cv(t)\right)}{h}\right)&= \lim_{h \to 0^+} h\log \mathbb{E}^{\rh+\alzh}\left(e^\frac{w_{\zero}^{\sbe_{\V}}\left(\rh+\aluh(t)\right)}{h}\right)\\
&=\lim_{h\to 0^+} \wrhhbh(\rh+\alzh, t)\\
&=w(0, t)\\
& = \sup_{\alpha\in[a, b]}\left\{w_0(\alpha)-I\left(\alpha;0,t\right) \right\}\\&=\sup_{\alpha\in[a, b]}\left\{u_0(\vx_0+\alpha\vec{\nu})-I\left(\alpha;0,t\right) \right\}\\
&=\sup_{\vec{y}\in S_0}\left\{u_0(\vec{y})-I\left(\frac{(\vec{y}-\vx_0)\cdot \vec{\nu}}{|\vec{\nu}|^2};0,t\right) \right\}\\
&=\sup_{\vec{y}\in \overline{\Omega}}\left\{u_0(\vec{y})-\tilde{I}\left(\vec{y}; \vx_0, t\right) \right\},
    \end{aligned}
\end{equation}
where
\[
\tilde{I}\left(\vec{y}; \vec{x}_0, t\right):=\begin{cases}
			I\left(\frac{(\vec{y}-\vx_0)\cdot \vec{\nu}}{|\vec{\nu}|^2};0,t\right), & \quad \text{if $\, \vec{y} \, $ in $S_0$},\\
            \infty, & \quad \text{otherwise}.
		 \end{cases}
\]

Hence, the proof is complete.
\end{proof}

\begin{cor}
Assume $\Omega$ is convex. Let $\vec{x}_0 \in \overline{\Omega}$, and for each small $h > 0$, define
\[
\vxzh = \vec{x}_0 + \beh \vec{\nu}_\perp + (\rh + \alzh)\vec{\nu} \in \overline{\Omega},
\]
where $\beh \in \mathbb{R}$ satisfies $\omega_0(|\beh|) \leq h$ with modulus of continuity $\omega_0$ as in~\eqref{eqn:endpointsdif}, $\rh \in \mathbb{R}$ with $|\rh| < h$, and $\alzh = \kzh h$ for some $\kzh \in \mathbb{Z}$. Suppose that
\[
\lim_{h \to 0^+} \rh = 0, \quad \lim_{h \to 0^+} \alzh = 0.
\]
Let $\eta^\ast \in AC([0,t];\mathbb{R})$ be the unique path such that $I(\eta^\ast(t);0,t)=0$. Then $\vx^\ast(t):=\vx_0+\eta^\ast(t)\vec{\nu}$ is the mean-field limit path in the sense that the weak law of large numbers for process $\cv(t)$ holds at any time $t$. Moreover, for any $t>0$ and $\varepsilon>0$, there exists $h_0>0$ such that if $h<h_0$, then
    \[
    \mathbb{P}\left\{\left|\cv(t)-\vx^\ast(t)\right|\geq\varepsilon\right\} \leq e^{-\frac{\beta(\varepsilon, t)}{2h}},
    \]
    where $\beta(\varepsilon, t): = \inf_{\left|\vy-\vx^\ast(t)\right|\geq \varepsilon, \, \vy \in \overline{\Omega}} \tilde{I}\left(\vec{y};\vx_0,t\right) >0$ with $\tilde{I}$ denoting the rate function defined in \eqref{eqn:deftildeI}.
\end{cor}


\section{example}\label{sec:example}
Now we present a concrete example showing that, in general, the limiting function $u$ does not satisfy the state constraint boundary condition for the continuous equation throughout the entire domain $\Omega$, which emphasizes the necessity of the discussion in Section 3 regarding the restriction of $u$ onto line segments.

Consider the following reaction with two species $X_1, X_2$:

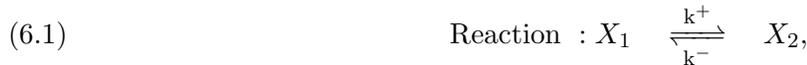
\begin{equation}
\text{Reaction }: X_1 \quad \ce{<=>[k^+][k^-]} \quad X_2,
\end{equation}
where, for simplicity, the reaction rates are taken as $k^+=k^-=1$. The reaction vector for this reaction is $\vec{\nu}=(-1, 1)$.

Let $\Omega :=\left\{ (x_1, x_2) \in \mathbb{R}^2: (x_1-7)^2+(x_2-3)^2<2 \right\}$, which is the ball centered at $(7, 3)$ with radius $\sqrt{2}$. By definition, $\Phi^+(x_1, x_2):=x_1$ and $\Phi^-(x_1, x_2):=x_2$. Take the initial condition $u_0(x_1, x_2)=x_1+x_2$.

On the grid
\[
\Omegah := \{ \vxi = \vec{i} h; \,\,  \vec{i} \in \mathbb{N}^N,\, \vxi \in \overline{\Omega}\},
\]
the discrete backward HJE is
\begin{equation}\label{eqn:exdisc}
\begin{aligned}
\pt_t \uh(x_{i, 1}, x_{i, 2},t) =& \chi_{\{\left(x_{i,1}- h, x_{i, 2}+h\right) \in \Omegah\}} x_{i, 1}\left(e^{\frac{\uh\left(x_{i,1}- h, x_{i, 2}+h,t\right)-\uh\left((x_{i, 1}, x_{i, 2}, t\right)}{h}} - 1\right) \\
&+\chi_{\{\left(x_{i,1}+ h, x_{i, 2}-h\right) \in \Omegah\}}x_{i, 2}\left( e^{  \frac{\uh\left(x_1+ h, x_2-h,t\right)-\uh\left(x_{i, 1}, x_{i, 2},t\right)}{h} } - 1\right), \quad \left(x_{i, 1}, x_{i, 2}\right)\in \Omegah.
\end{aligned}
\end{equation}
Since $\vec{\nu} \cdot Du_0=0$, the solution $\uh$ to \eqref{eqn:exdisc} with initial data $u_0\big|_{\Omegah}$ is stationary:
\[
\uh \left(x_{i,1}, x_{i,2}, t\right)=u_0\left(x_{i, 1}, x_{i, 2}\right)=x_{i, 1}+x_{i, 2}.
\]
Hence, along any sequence $\{{\vx}_{i_k}\} \subset \Omegah$ with ${\vx}_{i_k} \to \vx$ as $k \to \infty$, and with mesh sizes $h_k\to 0$, the limiting function remains
\[
u(x_1, x_2, t)=u_0(x_1, x_2)=x_1+x_2.
\]
It seems natural to consider the following Cauchy problem with state constraint boundary condition for the limit $u$:

\begin{equation}\label{eqn:excont}
\left\{\begin{aligned}
\partial_t \tilde{u}\left(x_1, x_2,t\right) &\geq    x_1\left(e^{\vec{\nu}\cdot \nabla \tilde{u}(x_1, x_2, t)}-1\right)     +   x_2\left( e^{-\vec{\nu}\cdot \nabla \tilde{u}(x_1, x_2, t)}-1\right) \qquad \quad  \text{in } \Omega \times (0, \infty),\\
\partial_t \tilde{u}\left(x_1, x_2,t\right) & \leq    x_1\left(e^{\vec{\nu}\cdot \nabla \tilde{u}(x_1, x_2, t)}-1\right)     +   x_2\left( e^{-\vec{\nu}\cdot \nabla \tilde{u}(x_1, x_2, t)}-1\right) \qquad \quad  \text{on } \overline{\Omega} \times (0, \infty),\\
\tilde{u}(x_1, x_2, 0)&= u_0(x_1, x_2) \qquad \qquad \qquad \qquad \qquad\qquad\qquad\qquad \qquad \quad \,\, \text{on } \overline{\Omega}.
\end{aligned}
\right.
\end{equation}
The solution $\tilde{u}$ admits the representation formula
\[
\tilde{u}(x, t)= \sup \left\{u_0(\gamma(t)) - \int_0^t L\left(\gamma(s), \dot{\gamma}(s)\right)\, ds : \gamma \in \mathrm{AC} \left([0,t] ; \overline{\Omega}\right), \gamma(0)=x\right\}.
\]
However, the limit $u$ obtained above is not the solution to \eqref{eqn:excont}.
More precisely, $u$ fails to be a subsolution of \eqref{eqn:excont} on $\overline{\Omega} \times (0, \infty)$.
To see this, consider the test function $\phi(x_1, x_2, t):=\frac{21}{20}x_1+\frac{19}{20}x_2-\frac{1}{10}$. We can verify that $\phi \geq u$ on $\overline{\Omega} \times [0, \infty)$, and $u-\phi $ attains a local maximum  at $(6, 4) \in \partial \Omega$ and $t=1$. The subsolution property would require
\[
\partial_t \phi \left(6, 4,1\right) \leq    6\left(e^{\vec{\nu}\cdot \nabla \phi(6, 4, 1)}-1\right)     +  4\left( e^{-\vec{\nu}\cdot \nabla \phi (6, 4, 1)}-1\right).
\]
But since $\partial_t \phi(6, 4, 1) =0$ and $\nabla \phi(6,4,1)=\left(\frac{21}{20}, \frac{19}{20}\right)$, we compute
\[
\partial_t \phi \left(6, 4,1\right) =0 > 6(e^{-\frac{1}{10}}-1) +4(e^{\frac{1}{10}}-1)=6\left(e^{\vec{\nu}\cdot \nabla \phi(6, 4, 1)}-1\right)     +  4\left( e^{-\vec{\nu}\cdot \nabla \phi (6, 4, 1)}-1\right),
\]
which violates the inequality. Hence, $u$ is not a solution to the state constraint problem \eqref{eqn:excont}.

\section*{Acknowledgment}
Yuan Gao was partially supported by NSF under award DMS-2204288 and CAREER award DMS-2440651. Yuxi Han was partially supported by NSF under award DMS-2440651.

\bibliographystyle{alpha}
\bibliography{LD_vis_bib}

\appendix

\section{Some omitted proofs and lemmas}




\subsection{Existence of a sequence of local maximum points}\label{sec:localmax}
\begin{lem}
    Let $\varphi \in C^1\left([a, b] \times [0, \infty)\right)$ such that $\overline{w}-\varphi$ attains a strict local max at $(\alpha_0, t_0) \in [a, b] \times (0, \infty)$. Then there exists a sequence $\{h_n\}_{n=1}^{\infty}$ with $\lim_{n \to \infty} h_n = 0$ and a sequence $\{(r_n+\alpha_n, t_n) \} \subset \tornhn \times (0, \infty)$ with $\lim_{n \to \infty} r_n =0$, $\lim_{n \to \infty} \alpha_n = \alpha_0$ and $\lim_{n \to \infty} t_n = t_0$ such that
    \[
    \wrnhn\left(r_n+\alpha_n, t_n\right) \to \overline{w} (\alpha_0, t_0)
    \]
 and $(r_n+\alpha_n,t_n)$ is a local maximum of $\wrnhn-\varphi$.
\end{lem}
\begin{proof}
Without loss of generality, we assume $\overline{w} (\alpha_0, t_0)=\varphi(\alpha_0, t_0)$. Let \(\{\delta_n\}_{n=1}^\infty\subset(0,\infty)\) be a strictly decreasing sequence with \(\delta_n\to0\). Since
\[
\overline w(\alpha_0, t_0) - \varphi(\alpha_0, t_0)=0,
\]
for each \(n\) we can choose \(\varepsilon_n>0\) with \(\varepsilon_n\to0\) such that
\begin{equation}\label{eqn:wsonbd}
w(\alpha,t) - \varphi(\alpha,t) < -\,\varepsilon_n
\quad\text{for all }
(\alpha,t)\in \partial B\bigl((\alpha_0,t_0),\delta_n\bigr).
\end{equation}
Since $\varphi \in C_1$, there exists a constant $\tilde{\delta}_n >0$ such that for any $(\alpha, t), (\beta,s) \in [a, b]\times [0,\infty)$ with $\left|\alpha-\beta\right|+\left|t-s\right|<\tilde{\delta}_n$, we have
\begin{equation}
\left|\varphi(\alpha, t)-\varphi(\beta, s)\right| <\frac{\varepsilon_n}{4}.
\end{equation}
From \eqref{eqn:wsonbd}, for every $(\alpha, t) \in \partial B\bigl((\alpha_0,t_0),\delta_n\bigr)$, there exists a constant $\delta^\ast_{\alpha, t, n} \in \left(0, \tilde{\delta}_n\right)$ such that for any $r, h$ with $|r|<h< \frac{\delta^\ast_{\alpha, t, n}}{4}$ and $(r+\alhi, s ) \in B \left((\alpha, t), \delta^\ast_{\alpha, t, n} \right) \cap \left(\torh \times (0, \infty)\right)$, we have
\[
\wrh (r+\alhi, s)- \varphi(r+\alhi,s)<-\frac{3}{4}\varepsilon_n.
\]

By the compactness of the circle $\partial B\bigl((\alpha_0,t_0),\delta_n\bigr)$, we may deduce that there exists a constant $\delta'_n \in \left(0, \min\left\{\frac{\delta_n}{2}, \tilde{\delta}_n\right\}\right)
$ such that for any $r, h$ with $|r|<h< \frac{\delta'_n}{4}$ and $(r+\alhi, s ) \in \left( \overline{B((\alpha_0, t_0),\delta_n)}\setminus B((\alpha_0, t_0), \delta_n-\delta'_n) \right) \cap \left(\torh \times (0, \infty)\right)$,
\begin{equation}\label{eqn:wsout}
    \wrh (r+\alhi, s)- \varphi(r+\alhi,s)<-\frac{3}{4}\varepsilon_n.
\end{equation}
On the other hand, by the definition of $\overline{w}(\alpha_0, t_0)$, there exists $r_n \in \mathbb{R}, h_n>0$ with $|r_n|<h_n <\frac{\delta'_n}{4}$ and  $(r_n+\alpha_n, t_n) \in B((\alpha_0, t_0),\delta'_n)\cap \tornhn \times (0, \infty)$ such that
\begin{equation}\label{eqn:wlin}
    \overline{w}\left(\alpha_0, t_0\right) -\frac{1}{4}\varepsilon_n <\wrnhn\left(r_n+\alpha_n, t_n\right).
\end{equation}
Combining \eqref{eqn:wsout} and \eqref{eqn:wlin}, we have
\[
\begin{aligned}
&\max_{ (r_n+\alpha_{h_n, i}, s) \in \left( \overline{B((\alpha_0, t_0),\delta_n)}\setminus B((\alpha_0, t_0), \delta_n-\delta'_n) \right) \cap \tornhn \times (0, \infty)}\left(\wrnhn (r_n+\alpha_{h_n,i}, s)- \varphi(r_n+\alhi,s)\right)\\
\leq & \overline{w}(\alpha_0, t_0)-\varphi(\alpha_0, t_0)-\frac{3}{4} \varepsilon_n\\
\leq & \wrnhn\left(r_n+\alpha_n, t_n\right) -\varphi(\alpha_0, t_0)-\frac{1}{2}\varepsilon_n\\
\leq & \wrnhn\left(r_n+\alpha_n, t_n\right) -\varphi\left(r_n+\alpha_n, t_n\right) -\frac{\varepsilon_n}{4},
\end{aligned}
\]
which implies $\wrnhn- \varphi$ has a local max in the interior of $B((\alpha_0, t_0), \delta'_n)$.
Finally, as \(n\to\infty\), note that
\[
\delta'_n \to 0, \quad h_n\to0, \quad r_n\to0, \quad \alpha_n\to\alpha_0, \quad t_n\to t_0,
\]
and
\[
\wrnhn(r_n+\alpha_n,t_n)\;\longrightarrow\;\overline w(\alpha_0,t_0),
\]
as required.
\end{proof}

\subsection{Boundedness of $v$ in an optimal triple $(\eta, v, l) \in \text{SP}(\alpha)$ of $w(\alpha,t)$}
\begin{lem}\label{lem:velbd}
Assume $u_0 \in W^{1, \infty}(\Omega) \cap C(\overline{\Omega})$. Suppose that $\beta>0$ is sufficiently small. Let $w : [a, b] \times [0, \infty) \to \mathbb{R}$ and $w^{\sbe}: [a_{\sbe}, b_{\sbe}] \times [0, \infty) \to \mathbb{R}$ be the solutions to \eqref{HJE_NM} and \eqref{eqn:limitbeta} respectively. For any $\alpha \in [a, b]$, let $(\eta, v, l) \in \mathrm{SP}(\alpha)$ satisfy \eqref{eqn:wpath}. Similarly, for any $\alpha' \in [a_{\sbe}, b_{\sbe}]$, let $(\eta_{\sbe}, v_{\sbe}, l_{\sbe}) \in \mathrm{SP}_{\sbe}(\alpha')$ satisfy \eqref{eqn:wbetaocf}.

Then both $v$ and $v_{\sbe} \in L^\infty\left([0, \infty)\right)$. Moreover, there exists a constant $C=C\left(\Omega, H, \left\|Du_0\right\|_{L^\infty(\Omega)}\right)>0$ independent of $\alpha, \alpha'$ such that
\[
\|v\|_{L^\infty\left([0,\infty)\right)}, \|v _{\sbe}\|_{L^\infty\left([0,\infty)\right)} \leq C.
\]
\end{lem}
\begin{proof}
We prove the result for $v$ first. By \cite[Theorem 7.2]{Ishii2011}, it suffices to show that $w \in \mathrm{Lip}\left([a, b] \times [0, \infty)\right)$. We first establish a Lipschitz bound in time for $w$ by following the arguments in the proof of \cite[Lemma 5.3, 5.4]{Ishii2011}, and then derive a Lipschitz estimate in $\alpha$ using the coercivity of $\tilde{H}$.

Since $u_0 \in W^{1, \infty}(\Omega) \cap C(\overline{\Omega})$, we have $w_0 \in \mathrm{Lip}\left([a, b] \right)$. Let $\varepsilon>0$, and choose $w_{0,\varepsilon} \in C^1([a, b])$ such that
\[
\left\|w_{0, \varepsilon}-w_0\right\|_{L^\infty([a,b])} \leq \varepsilon, \qquad \|w'_{0, \varepsilon}\|_{L^\infty([a, b])} \leq \left\|Du_0\right\|_{L^\infty(\Omega)}.
\]

Next, select a function $\psi \in C^1([a, b])$ such that $\psi (\alpha )<0$ and $\left|\psi'(\alpha)\right| \leq 1$ for all $\alpha \in (a, b)$, $\psi(a)=\psi(b)=0$, $\psi'(a)=-1$, and $\psi'(b)=1$. Multiplying $\psi$ by a positive constant if necessary, we may assume $w_{0, \varepsilon}'(a)+\psi'(a)<0$ and $w_{0, \varepsilon}'(b)+\psi'(b)>0$ for all $\varepsilon>0$, and $\left\|\psi'\right\|_{L^\infty([a, b])} \leq \max\left\{3 \left\|Du_0\right\|_{L^\infty(\Omega)}, 1\right\}$.

We now approximate the function $r \mapsto (-\varepsilon)\vee (\varepsilon \wedge r)$ on $\mathbb{R}$ by a smooth function $\zeta_\varepsilon \in C^1(\mathbb{R})$ with $\left\|\zeta_\varepsilon\right\|_{L^\infty(\mathbb{R})} \leq \varepsilon$, $\left\|\zeta_\varepsilon'\right\|_{L^\infty(\mathbb{R})}\leq1$, and $\zeta_\varepsilon'(0)=1$. Note that $\left(\zeta_\eps \circ \psi(\alpha)\right)'=\psi'(\alpha)$ at $\alpha=a, b$, and $\left|w_0(\alpha)-w_{0, \varepsilon}\left(\alpha\right)-\zeta_\varepsilon\circ \psi(\alpha)\right| \leq 2 \varepsilon$ for all $\alpha \in [a, b]$. Furthermore, we can choose a constant $C_0=C_0\left(\Omega, H,  \left\|Du_0\right\|_{L^\infty(\Omega)}\right)>0$, independent of $\varepsilon$, such that
\[
\tilde{H}\left(\alpha, w_{0, \varepsilon}'+(\zeta_\varepsilon \circ \psi)'(\alpha)\right) \leq C_0
\]
for all $\alpha \in [a, b]$.

Define the function $\phi_\varepsilon(\alpha, t): =2\varepsilon +w_{0, \varepsilon}(\alpha)+\zeta_\varepsilon \circ \psi(\alpha) +C_0t$. Note that $\partial_\alpha \phi_\varepsilon(a, t)=w_{0, \varepsilon}'(a)+\psi'(a)<0$ and $\partial_\alpha \phi_\varepsilon(b, t)=w_{0, \varepsilon}'(b)+\psi'(b)>0$. Then $\phi_\varepsilon \in C^1([a, b]\times[0, \infty))$ is a classical supersolution to \eqref{HJE_NM}. By Proposition \ref{prop:cpn}, we have
\[
\phi_\varepsilon(\alpha, t) \geq w(\alpha, t) \quad \text{for all }\alpha \in [a, b], \, t \geq 0,
\]
which implies
\[4\varepsilon + w_0(\alpha) +C_0t\geq w(\alpha, t) \quad \text{for all }\alpha \in [a, b], \, t \geq 0.
\]
Since $\varepsilon$ is arbitrary, we conclude
\begin{equation}\label{eqn:liptgeq}
w_0(\alpha) +C_0t\geq w(\alpha, t) \quad \text{for all }\alpha \in [a, b], \, t \geq 0.
\end{equation}

Let $(\alpha, t) \in [a, b] \times [0, \infty)$. Set $\eta(s)=\alpha, v(s)=0$, and $l(s)=0$ for $s\geq 0$. Then $\left(\eta, v, l\right) \in \mathrm{SP}(\alpha)$. Hence
\begin{equation}\label{eqn:liptleq}
 \begin{aligned}
w(\alpha, t) &\geq w_0(\alpha) -t\tilde{L}(\alpha, 0)\\
&\geq w_0(\alpha)+t \inf_{\alpha \in [a, b], p\in \mathbb{R}}\tilde{H}(\alpha, 0)\\
&\geq w_0(\alpha)+t  \inf_{x \in \overline{\Omega}, p\in \mathbb{R}^2}H(x, p)\\
&\geq  w_0(\alpha)-tC
\end{aligned}
\end{equation}
where $C>0$ is a constant such that $-C < \inf_{x \in \overline{\Omega}, p\in \mathbb{R}^2}H(x, p)$.

Combining \eqref{eqn:liptgeq} and \eqref{eqn:liptleq}, for any $\alpha \in [a, b]$ and $t\geq 0$,
\begin{equation}\label{eqn:lipwte0}
\left|w_0(\alpha)-w(\alpha, t)\right| \leq Ct
\end{equation}
for some constant $C= C\left(\Omega, H, \left\|Du_0\right\|_{L^\infty(\Omega)}\right)>0$.

For fixed $s>0$, $v(x, t):=w(\alpha, s+t)$ is a solution to \eqref{HJE_NM} with the initial condition $v_0(\alpha)=v(\alpha, 0)= w(\alpha, s)$. Since
\[
v_0 -\left\|w_0-v_0\right\|_{L^\infty([a, b])} \leq w_0\leq v_0 +\left\|w_0-v_0\right\|_{L^\infty([a, b])},
\]
by Proposition \ref{prop:cpn}, we have for any $\alpha \in [a, b]$ and $t\geq 0$,
\[
v(\alpha, t) -\left\|w_0-v_0\right\|_{L^\infty([a, b])} \leq w(\alpha, t)\leq v(\alpha, t) +\left\|w_0-v_0\right\|_{L^\infty([a, b])}.
\]
Therefore,
\[
\left|w(\alpha, t) - w(\alpha, s+t)\right| \leq \left|w_0(\alpha) -w(\alpha, s)\right|\leq Cs,
\]
where $C= C\left(\Omega, H, \left\|Du_0\right\|_{L^\infty(\Omega)}\right)>0$ comes from \eqref{eqn:lipwte0}. Hence, $w$ is Lipschitz in time.

Let $C_1, C_2 >0$ be such that
\[
C_1 \leq \left\|\Phi^+\right\|_{L^\infty\left(\overline{\Omega}\right)}, \left\|\Phi^-\right\|_{L^\infty\left(\overline{\Omega}\right)} \leq C_2.
\]
Then the Hamiltonian $\tilde{H}$ satisfies the coercivity estimate
\begin{equation}\label{eqn:coertH}
\tilde{H}(\alpha, p) \geq C_1\left(e^{|p|}-1\right)-C_2.
\end{equation}
Since $w$ solves \eqref{HJE_NM} and $\left\| w_t\right\|_{L^\infty \left([a,b]\times [0, \infty)\right)} \leq C$ for some constant $C= C\left(\Omega, H, \left\|Du_0\right\|_{L^\infty(\Omega)}\right)>0$, the coercivity \eqref{eqn:coertH} implies
\begin{equation}
\left\|\partial_\alpha w\right\|_{L^\infty \left([a,b]\times [0, \infty)\right)} \leq C,
\end{equation}
for some constant $C= C\left(\Omega, H, \left\|Du_0\right\|_{L^\infty(\Omega)}\right)>0$.
Hence, $w \in \mathrm{Lip}\left([a, b] \times [0, \infty)\right)$ and
\begin{equation}\label{eqn:lipbdtalphaw}
 \left\| w_t\right\|_{L^\infty \left([a,b]\times [0, \infty)\right)},  \, \left\|\partial_\alpha w\right\|_{L^\infty \left([a,b]\times [0, \infty)\right)} \leq \tilde{C}
\end{equation}
for some constant $\tilde{C}=\tilde{C}\left(\Omega, H, \left\|Du_0\right\|_{L^\infty(\Omega)}\right)>0$.

The same argument applies to $w^{\sbe}$, showing that $w^{\sbe} \in \mathrm{Lip}([a_{\sbe}, b_{\sbe}]\times [0,\infty))$ and
\[
 \left\| w_t^{\sbe}\right\|_{L^\infty \left([a,b]\times [0, \infty)\right)},  \, \left\|\partial_\alpha w^{\sbe}\right\|_{L^\infty \left([a,b]\times [0, \infty)\right)} \leq \tilde{C}
\]
with the same constant $\tilde{C}$ as in \eqref{eqn:lipbdtalphaw}.

By the proof of \cite[Theorem 7.2]{Ishii2011}, there exists a constant $C=C\left(\Omega, H, \left\|Du_0\right\|_{L^\infty(\Omega)}, \tilde{C}\right)>$ such that
\[
\|v\|_{L^\infty\left([0,\infty)\right)}, \|v _{\sbe}\|_{L^\infty\left([0,\infty)\right)} \leq C,
\]
where $\tilde{C}$ is the Lipschitz bound for $w, w^{\sbe}$.

\end{proof}

\end{document}